\documentclass{amsart}
\usepackage[T1]{fontenc}

\usepackage{lmodern}

\usepackage{amssymb,amsmath,amsfonts,amsthm,epsfig,amscd,graphicx,tikz-cd}
\usepackage{stmaryrd,mathdots}
\usepackage{IEEEtrantools}
\usepackage[all,cmtip,poly]{xy}
\usepackage{svn}
\usepackage{mathrsfs}
 \usepackage[utf8]{inputenc}

 \usepackage[pdfusetitle]{hyperref}

\hypersetup{
  colorlinks = false,
}

\newtheorem{thm}[subsubsection]{Theorem}
\newtheorem{lemma}[subsubsection]{Lemma}
\newtheorem{lem}[subsubsection]{Lemma}

\newtheorem{cor}[subsubsection]{Corollary}

\newtheorem{prop}[subsubsection]{Proposition}

\newtheorem{defn}[subsubsection]{Definition}

\theoremstyle{remark}
\newtheorem{remark}[subsubsection]{Remark}
\newtheorem{rem}[subsubsection]{Remark}

\newtheorem{example}[subsubsection]{Example}

\newtheorem{assumption}[subsubsection]{Assumption}

\numberwithin{equation}{subsection}
\def\nummultline{\addtocounter{subsubsubsection}{1}\begin{multline}}
\def\anumequation{\addtocounter{subsection}{1}\begin{equation}}

\newif\iffinalrun
\iffinalrun
\else
\fi

\iffinalrun
  \newcommand{\need}[1]{}
  \newcommand{\mar}[1]{}
\else
  \newcommand{\need}[1]{{\tiny *** #1}}
  \newcommand{\mar}[1]{\marginpar{\raggedright\tiny #1}}
\fi

\newcommand{\A}{\AA}
\def\C{\CC}
\newcommand{\F}{\FF}

\newcommand{\Q}{\QQ}
\newcommand{\R}{\RR}
\newcommand{\Z}{\ZZ}

\newcommand{\m}{\frakm}
\newcommand{\p}{\frakp}

\renewcommand{\AA}{{\mathbb A}}

\newcommand{\CC}{{\mathbb C}}

\newcommand{\FF}{{\mathbb F}}
\newcommand{\GG}{{\mathbb G}}

\newcommand{\QQ}{{\mathbb Q}}
\newcommand{\RR}{{\mathbb R}}
\renewcommand{\SS}{{\mathbb S}}
\newcommand{\TT}{{\mathbb T}}

\newcommand{\ZZ}{{\mathbb Z}}

\renewcommand{\bf}{\ensuremath{\mathbf{f}}}

\newcommand{\cA}{{\mathcal A}}

\newcommand{\cD}{{\mathcal D}}

\newcommand{\cF}{{\mathcal F}}
\newcommand{\cG}{{\mathcal G}}
\newcommand{\cH}{{\mathcal H}}
\newcommand{\cI}{{\mathcal I}}

\newcommand{\cO}{{\mathcal O}}

\newcommand{\cS}{{\mathcal S}}

\newcommand{\cU}{{\mathcal U}}
\newcommand{\cV}{{\mathcal V}}

\newcommand{\cX}{{\mathcal X}}
\newcommand{\cY}{{\mathcal Y}}
\newcommand{\cZ}{{\mathcal Z}}

\newcommand{\frakm}{\mathfrak{m}}

\newcommand{\frakp}{\mathfrak{p}}

\newcommand{\frakX}{\mathfrak{X}}

\newcommand{\Zp}{\Z_p}

\newcommand{\Qp}{\Q_p}

\newcommand{\DXG}[1][m]{\partial X^G_{K(m)}}
\newcommand{\XGP}[1][m]{X^{G,P}_{K_{#1}}}
\newcommand{\YGP}[1][m]{Y^{G,P}_{K_{#1}}}
\newcommand{\XP}[1][m]{X^{P}_{K_{#1}}}

\DeclareMathOperator{\Ad}{Ad}
\DeclareMathOperator{\Aut}{Aut}

\DeclareMathOperator{\End}{End}

\DeclareMathOperator{\Flr}{Fl}
\DeclareMathOperator{\Gal}{Gal}
\DeclareMathOperator{\Gen}{Gen}
\DeclareMathOperator{\Map}{Map}
\DeclareMathOperator{\GL}{GL}

\DeclareMathOperator{\Hom}{Hom}

\DeclareMathOperator{\ord}{ord}

\DeclareMathOperator{\rank}{rank}

\DeclareMathOperator{\Spa}{Spa}
\DeclareMathOperator{\Spec}{Spec}

\DeclareMathOperator{\Spf}{Spf}

\newcommand{\Frob}{\mathrm{Frob}}

\newcommand{\HT}{\mathrm{HT}}

\newcommand{\id}{\mathrm{id}}

\newcommand{\red}{\mathrm{red}}

\newcommand{\forget}{\mathrm{forget}}
\newcommand{\descent}{\mathrm{descent}}
\newcommand{\Can}{\mathrm{Can}}

\newcommand{\et}{\mathrm{\acute{e}t}}
\newcommand{\proet}{\mathrm{pro\acute{e}t}}

\newcommand*{\invlim}{\varprojlim}

\newcommand{\toisom}{\xrightarrow{\sim}}

\newcommand{\mbf}{\mathbf}

\newcommand{\Fl}{\mathscr{F}\!\ell}

\newcommand{\cycl}{\mathrm{cycl}}
\newcommand{\Ha}{\mathrm{Ha}}

\newcommand{\ad}{\mathrm{ad}}

\newcommand{\tGam}{\widetilde{\Gamma}}
\newcommand{\tG}{\widetilde{G}}
\newcommand{\tK}{{\widetilde{K}}}
\newcommand{\tH}{\widetilde{H}}
\newcommand{\tN}{\widetilde{N}}
\newcommand{\tP}{\widetilde{P}}
\newcommand{\tX}{\widetilde{X}}
\newcommand{\tcX}{\widetilde{\mathcal{X}}}
\newcommand{\ocX}{\overline{\mathcal{X}}}
\newcommand{\nc}{\mathrm{nc}}
\newcommand{\Perf}{\mathrm{Perf}}
\newcommand{\Pfd}{\mathrm{Pfd}}

\newcommand{\ol}{\overline}
\newcommand{\ul}{\underline}
\newcommand{\wt}{\widetilde}
\newcommand{\mf}{\mathfrak}
\newcommand{\qp}{\mathrm{qpro\acute{e}t}}
\newcommand{\sub}{\subseteq}

\DeclareMathOperator\Lie{Lie}

\usepackage{colonequals}
\newcommand{\defeq}{\colonequals}

\newcommand{\tensor}{\otimes}
\newcommand{\suchthat}{\;\ifnum\currentgrouptype=16 \middle\fi|\;}

\begin{document}
\title{Shimura varieties at level $\Gamma_1(p^\infty)$ and Galois representations}
\author[A.~Caraiani, D.~Gulotta, C.~Hsu, C.~Johansson, L.~Mocz, E.~Reinecke, S.~Shih]{Ana Caraiani, Daniel R. Gulotta, Chi-Yun Hsu, Christian Johansson, Lucia Mocz, Emanuel Reinecke, Sheng-Chi Shih}
\address{Department of
  Mathematics, Imperial College London,
  London SW7 2AZ, UK}
\email{caraiani.ana@gmail.com}
\address{Department of Mathematics, Columbia University, 2990 Broadway, New York, NY 10027, USA}
\curraddr{Mathematical Institute, University of Oxford, Oxford OX2 6GG, UK}
\email{Daniel.Gulotta@maths.ox.ac.uk}
\address{Department of Mathematics, Harvard University, 1 Oxford Street, Cambridge, MA 02138, USA}
\email{chiyun@math.harvard.edu}
\address{Department of Mathematical Sciences, Chalmers University of Technology and the University of Gothenburg,
  SE-412 96, Sweden}
\email{chrjohv@chalmers.se}
\address{Mathematisches Institut der Universit\"at Bonn, Endenicher Allee 60, 53115 Bonn, Germany}
\email{lmocz@math.uni-bonn.de}
\address{Department of Mathematics, University of Michigan, Ann Arbor, MI 48109, USA}
\email{reinec@umich.edu}
\address{University of Lille, CNRS, UMR 8524 -- Laboratoire Paul Painlev\'e, 59000 Lille, France.}
\email{sheng-chi.shih@univ-lille.fr}

\thanks{A.C.\ was supported in part by a Royal Society University Research Fellowship and by ERC Starting Grant 804176. C.H.\ is partially supported by a Government Scholarship to Study Abroad from Taiwan. C.J.\ was supported in part by the Herchel Smith Foundation. E.R.\ was partially supported by NSF Grant No.~DMS-1501461.}

\maketitle
\begin{abstract}
We show that the compactly supported cohomology of certain $\mathrm{U}(n,n)$ or $\mathrm{Sp}(2n)$-Shimura varieties with $\Gamma_1(p^\infty)$-level vanishes above the middle degree. The only assumption is that we work over a CM field $F$ in which the prime $p$ splits completely. We also give an application to Galois representations for torsion in the cohomology of the locally symmetric spaces for $\mathrm{GL}_n/F$. More precisely, we use the vanishing result for Shimura varieties to eliminate the nilpotent ideal in the construction of these Galois representations. This strengthens recent results of Scholze~\cite{scholze-galois} and Newton-Thorne~\cite{newton-thorne}. 
\end{abstract}

\tableofcontents

\section{Introduction}

\subsection{Statement of results} The first goal of this paper is to study the cohomology of certain Shimura varieties with infinite level at $p$ and prove a vanishing theorem for their compactly supported cohomology above the middle degree. The second goal of this paper is to give an application to Galois representations for torsion in the cohomology of locally symmetric spaces. 

The first question is motivated by a deep conjecture of Calegari--Emerton on the completed (co)homology of general locally symmetric spaces~\cite[Conjecture 1.5]{calegari-emerton}, which in the case of tori is equivalent to the Leopoldt conjecture, cf.~\cite{hill}. This conjecture is motivated by the Langlands reciprocity conjecture and is expected to play an important role in the development of the classical and $p$-adic Langlands programs; see for example~\cite{emerton-icm} and~\cite{gee-newton}. When the locally symmetric spaces do not admit an algebraic structure, the Calegari--Emerton conjecture seems out of reach at the moment, outside the case of low-dimensional examples such as arithmetic hyperbolic $3$-manifolds. For Shimura varieties of Hodge type, Scholze~\cite{scholze-galois} recently made significant progress towards the Calegari--Emerton conjectures. 

Let $(G,X)$ be a Shimura datum of Hodge type.\footnote{To avoid complications in the introduction, we allow the non-standard setup of \cite[\S 4]{scholze-galois} as well as the usual setup, and will not make any explicit distinction between Shimura data and connected Shimura data.}
Let $K^p\subset G(\A^p_f)$ be a sufficiently small compact open subgroup which we fix; this will denote the tame level and we drop it from the notation for simplicity. Choose an integral model of $G$ over $\Z_p$. For $m\in \Z_{\geq 1}$, let 
\[
\Gamma(p^m)\defeq \left\{\gamma\in G(\Z_p)\mid \gamma\equiv \mathrm{Id}\mod{p^m}\right\}%
\]
and let $X_{\Gamma(p^m)}$ denote the corresponding Shimura variety of tame level $K^p$ and level $\Gamma(p^m)$ at $p$. Below, we consider this Shimura variety as a complex analytic space of dimension $d$ and we let $H^i_c$ denote the compactly supported singular cohomology. 

\begin{thm}[Corollary 4.2.2 of~\cite{scholze-galois}]\label{scholze's thm} Let $r\in \Z_{\geq 1}$. If $i>d$, then 
\[
\varinjlim_{m} H^i_c(X_{\Gamma(p^m)}, \Z/p^r\Z) = 0. 
\] 
\end{thm}

\noindent See~\cite[Cor. 4.2.3]{scholze-galois} for the direct connection to the Calegari--Emerton conjecture, which is phrased in terms of completed homology and completed Borel--Moore homology. We explain this further in \S~\ref{strategy}; for now, let us mention that Scholze's theory of perfectoid spaces and his $p$-adic Hodge theory for rigid analytic varieties both play a crucial role in the proof of Theorem~\ref{scholze's thm}. 

In this paper, we prove a stronger version of Scholze's result, in the particular case of $\mathrm{U}(n,n)$- and $\mathrm{Sp}_{2n}$-Shimura varieties. More precisely, let $F$ be a CM field in which the prime $p$ splits completely. Let $G/\Q$ be the group defined in \S~\ref{Borel-Serre}: it is either a quasi-split unitary group defined with respect to the extension $F/F^+$, when $F$ is an imaginary CM field with totally real subfield $F^+$, or the Weil restriction of scalars from $F$ to $\Q$ of a symplectic group, when $F$ is a totally real field. After choosing an integral model for $G$, we 
let $K^p\subset G(\A^p_f)$ be a sufficiently small compact open subgroup,
and we set
\[
\Gamma_1(p^m)\defeq \left\{\gamma=(\gamma_{\p})_{\p}\in G(\Z_p)\mid \gamma_{\p}\equiv \left(\begin{array}{cc} I_n  & * \\ 0 & I_n \end{array}\right)\pmod{p^m}\right\}
\] 
for $m\in \Z_{\geq 1}$. (Here, $\frakp$ runs through primes of $F$ above $p$ induced by a fixed CM type $\Psi$ of $F$ when $F$ is imaginary CM, and all primes of $F$ above $p$ when $F$ is totally real.) 
Let $X_{K_p}$ be the Shimura variety for $G$ with tame level $K^p$ and level $K_p$ at $p$. We prove the following result. 

\begin{thm}[Theorem~\ref{main thm non-similitude}]\label{unified main theorem} 
Let $r\in \Z_{\geq 1}$. If $i>d$ and $K_{p,m} \sub \Gamma_1(p^m)$ is compact open for all $m$, then 
\[
 \varinjlim_{m} H^{i}_{c}(X_{K^p K_{p,m}}(\C),\Z/p^r)=0.
 \]
\end{thm}

\noindent The key new idea we use to prove Theorem~\ref{unified main theorem} is to exploit the \emph{Bruhat stratification} on the Hodge--Tate period domain associated to these Shimura varieties. As far as we are aware, this is the first instance where this stratification is considered in the context of Shimura varieties. We note that our result can be used to recover Scholze's result at level $\Gamma(p^\infty)$, by letting $K_{p,m}:=\Gamma(p^m)$. 
On the other hand, by going down to level $\Gamma_1(p^\infty)$, our result goes beyond what Calegari--Emerton have conjectured.   

\begin{remark} The question of vanishing of cohomology of Shimura varieties at \emph{finite level} has been studied extensively recently, motivated in part by the Calegari--Geraghty program for proving modularity beyond the Taylor--Wiles setting~\cite{calegari-geraghty}. At finite level, results have been obtained, among other works, in~\cite{lan-suh, lan-suh2, boyer, caraiani-scholze}. These are all under various assumptions and, except for~\cite{lan-suh2}, deal with compact Shimura varieties. 

For $\mathrm{U}(n,n)$-Shimura varieties, which are non-compact, the strongest result we expect is joint work in progress of one of us (A.C.) 
with Scholze~\cite{caraiani-scholze2}. There, we prove a vanishing result for compactly supported cohomology at finite level, after localizing at a system of Hecke eigenvalues that is sufficiently generic in a precise sense and that also satisfies a version of a non-Eisenstein condition. (We also need to assume for technical reasons that $F$ is not an imaginary quadratic field.) 
The genericity we impose is a representation-theoretic condition, 
which can be thought of as a mod $p$ analogue of temperedness, and the method of proof is completely different 
from that of Theorem~\ref{unified main theorem}. Unlike at level $\Gamma_1(p^\infty)$, the result at finite 
level will not hold true without some kind of non-degeneracy assumption on the system of Hecke eigenvalues. 
For example, if $\bar{\rho}_{\m}$ is a non-generic direct sum of characters, one can show that the corresponding system of Hecke eigenvalues 
$\m$ is in the support of some $H^0(X_K, \F_{\ell})$ and therefore also in the support of $H^{2d}_c(X_K, \F_{\ell})$. 
Moreover, proving the analogue of the main result of~\cite{caraiani-scholze2} for $\mathrm{Sp}_{2n}$-Shimura varieties would 
require substantially more work involving the trace formula, as would the unitary case when $F$ is an imaginary quadratic field.
\end{remark}

\begin{remark} 
We expect the argument presented in this paper to naturally go through for Hodge type Shimura varieties at primes where the group is \emph{split}, and also in some other cases (e.g. for Harris--Taylor Shimura varieties at primes which are split in the imaginary quadratic field). However, some assumption on the prime $p$ appears to be necessary for our current argument to work --- see Remark~\ref{explanation for p split} for more details. We intend to leave the question of generalizing Theorem~\ref{unified main theorem} (in particular, removing the assumption on $p$) for future work.
\end{remark}

As mentioned above, we also give an application of Theorem~\ref{unified main theorem}. More precisely, we use the result to eliminate the nilpotent ideal in the construction of Galois representations associated to torsion in the cohomology of locally symmetric spaces for $\GL_n/F$, due to~\cite{scholze-galois} and refined by~\cite{newton-thorne}. We explain this further below; we do not define all the notions here, but they are made precise in Section~\ref{Borel-Serre}.

Set $M\defeq \mathrm{Res}_{F/\Q}\GL_n$. For a sufficiently small compact open subgroup $K_M\subset M(\A_f)$, let $X^{M}_{K_M}$ be the corresponding locally symmetric space for $M$. Let $S$ be a finite set of primes of $\Q$, containing $p$ as well as all the bad primes. We consider the abstract spherical Hecke algebra $\mathbb{T}^S_M$ away from $S$. It acts in the usual way on $H^*(X^M_{K_M}, \Z_p)$ and gives rise to the Hecke algebra 
\[
\mathbb{T}^S_M(K_M)\defeq \mathrm{Im}\left(\mathbb{T}^S_M\to \mathrm{End}_{\Z_p}\left(H^*(X^{M}_{K_M},\Z_p)\right)\right). 
\]
Let $\m\subset \mathbb{T}^S_M(K_M)$ be a maximal ideal. The following is~\cite[Cor.~5.4.3]{scholze-galois} 
(with slightly different normalizations, which are consistent with~\cite{newton-thorne}). (See \S~\ref{nilpotent intro} for an explanation of any notation
that has not been introduced yet.)

\begin{thm} There exists a unique continuous semisimple Galois representation 
\[
\bar{\rho}_{\m}\colon \mathrm{Gal}(\overline{F}/F)\to \GL_n(\overline{\F}_p)
\]
such that, for every prime $w$ of $F$ above $l\not\in S$, the characteristic polynomial of $\bar{\rho}_{\m}(\Frob_w)$ is equal to the image of
\[
P_{M,w}(X) = X^{n} - T_{1,w}X^{n-1}+\dots + (-1)^iq_{w}^{i(i-1)/2}T_{i,w}X^{n-i} +\dots + q_{w}^{n(n-1)/2}T_{n,w}
\]
modulo $\m$. 
\end{thm}

We continue to assume that $p$ splits completely in $F$ and also assume that $\bar{\rho}_{\m}$ is absolutely irreducible. We
replace $H^*(X^{M}_{K_M}, \Z_p)$ by the complex $R\Gamma(X^M_{K_M}, \Z_p)$, which naturally lives in the 
derived category $D(\Z_p)$ of $\Z_p$-modules. Its cohomology recovers $H^*(X^{M}_{K_M}, \Z_p)$ and the
action of $\mathbb{T}^S_M$ on $H^*(X^{M}_{K_M}, \Z_p)$ lifts to an action on $R\Gamma(X^M_{K_M}, \Z_p)$. 
We consider the Hecke algebra 
\[
\mathbb{T}^S_M(K_M)^{\mathrm{der}}\defeq \mathrm{Im}\left(\mathbb{T}^S_M\to \mathrm{End}_{D(\Z_p)}\left(R\Gamma(X^M_{K_M}, \Z_p)\right)\right). 
\]
This surjects onto $\mathbb{T}^S_M(K_M)$ with kernel a nilpotent ideal. We localize $\mathbb{T}^S_M(K_M)^{\mathrm{der}}$ at $\m$. 
Using Theorem~\ref{unified main theorem}, we prove the following. 

\begin{thm}[Theorem~\ref{eliminating nilpotent ideal}]\label{intro: nilpotent ideal} 
There exists a unique continuous Galois representation 
\[
\rho_{\m}\colon \mathrm{Gal}(\overline{F}/F)\to \GL_n\left(\mathbb{T}^S_{M}\left(K_M,\lambda\right)^{\mathrm{der}}_{\m}\right)
\]
such that, for every prime $w$ of $F$ above $l\not\in S$, the characteristic polynomial of $\rho_{\m}(\Frob_w)$ is equal to 
\[
P_{M,w}(X) = X^{n} - T_{1,w}X^{n-1}+\dots + (-1)^iq_{w}^{i(i-1)/2}T_{i,w}X^{n-i} +\dots + q_{w}^{n(n-1)/2}T_{n,w}.
\]

\end{thm} 

\begin{remark}\label{relevance of nilpotent ideal}\leavevmode
\begin{enumerate}
\item Up to a nilpotent ideal $I\subset \mathbb{T}^S_M(K_M)^{\mathrm{der}}_{\m}$ of nilpotence degree $4$, this is proved in Theorem 1.3 of~\cite{newton-thorne} (see also Corollary 5.4.4 of~\cite{scholze-galois}). We eliminate the nilpotent ideal $I$ and thus prove a more natural statement about the existence of Galois representations.  
\item For simplicity, we have chosen to state this result with trivial $\Z_p$-coefficients in the introduction. However, in Theorem~\ref{eliminating nilpotent ideal} 
we also allow twisted coefficients corresponding to an irreducible algebraic representation of $M$. 
\item Under the assumption that $p$ splits completely in $F$, this finishes the proof of the first part of Conjecture B of~\cite{calegari-geraghty}. The nilpotent ideal does not usually cause problems for applications of Conjecture B to automorphy lifting theorems, since the key point there is to determine the support of a certain patched module and nilpotents do not affect the support. However, eliminating the nilpotent ideal is important for some more subtle questions, such as those concerning Bloch--Kato conjectures for automorphic motives, cf.~\cite{cgh}, or local-global compatibility for the $p$-adic local Langlands correspondence, cf.~\cite{gee-newton}. These potential applications would also need local-global compatibility at $p$ for the Galois representations $\rho_{\m}$, which is still open. 
\end{enumerate}
\end{remark}

To deduce Theorem~\ref{intro: nilpotent ideal} from Theorem~\ref{unified main theorem}, we use a detailed study of the boundary of the Borel--Serre compactification of the locally symmetric spaces for $G$ (in particular the strata relevant to the Levi subgroup $M$) and the derived Hecke algebra introduced in~\cite{newton-thorne}. For a compact open subgroup $K\subset G(\A_f)$, we denote by $X^G_K$ the corresponding locally symmetric space for $G$, with Borel--Serre compactification $X^{G,\mathrm{BS}}_K$ and boundary $\partial X^{G,\mathrm{BS}}_K$. After the work of Newton--Thorne, the unique source of nilpotence was the ambiguity coming from the excision long exact sequence used in the construction of the Galois representation $\rho_{\m}$:
\[
\dotsb \to H^i(X^G_K,\Z/p^r\Z)\to H^i(\partial X^{G,\mathrm{BS}}_K,\Z/p^r\Z)\to H^{i+1}_c(X^G_K,\Z/p^r\Z)\to \dotsb.
\] 
(Recall that this long exact sequence is the excision sequence attached to the decomposition
\[
X^{G,\mathrm{BS}}_K=X^G_K\sqcup \partial X^{G,\mathrm{BS}}_K,
\] 
where we have also used the fact that the open immersion $X^G_K\hookrightarrow X^{G,\mathrm{BS}}_K$ is a homotopy equivalence.) By showing vanishing of compactly supported cohomology above the middle degree, we eliminate this ambiguity, at least for the cohomology of the boundary above the middle degree. 

We emphasize that we genuinely need to use Theorem~\ref{unified main theorem} for this argument, and we could not have made use of Theorem~\ref{scholze's thm} instead. Very roughly speaking, we want to realize the completed cohomology on the side of the locally symmetric spaces for $M$ as a direct summand in the boundary cohomology of the Borel--Serre compactification at some infinite level (for $G$), above the middle degree $d$. This is so that we can make use of the vanishing result to eliminate the ambiguity mentioned above. 

At level $\Gamma(p^{\infty})$, this is not possible as the cohomology of the relevant part of the boundary vanishes above degree $d$. Let $P$ denote the Siegel parabolic of $G$, with Levi subgroup $M$. The relevant part of the boundary can be identified with a union of locally symmetric spaces for $P$. The cohomology of the locally symmetric space for $M$ is supported in degrees $[0,d-1]$, and any shift to higher degrees comes from the cohomology of a torus (the locally symmetric space for the unipotent part of $P$). At level $\Gamma(p^\infty)$, only the degree $0$ part of the cohomology of the torus survives, so the cohomology of $M$ does not contribute above degree $d-1$.\footnote{See~\cite[\S 1.5]{calegari-emerton} for a more in-depth discussion.} 

However, it is possible to realize the completed cohomology of the locally symmetric spaces for $M$ as a direct summand at level $\Gamma_{1}(p^{\infty})$ (up to a twist and degree-shift). This is the content of Theorem \ref{direct summand}, and we remark that this is significantly more subtle at level $\Gamma_1(p^\infty)$ than at finite level. At finite level, the connected components of the boundary stratum corresponding to $P$ are indexed by a certain finite set of double cosets; in order to split off the cohomology of the locally symmetric space for $M$, one simply restricts to the identity double coset. At level $\Gamma_1(p^\infty)$, the boundary stratum corresponding to $P$ is indexed by a pro-finite set of double cosets, and restriction to the identity double coset does not give a direct summand. To prove Theorem \ref{direct summand} we introduce another new idea, namely we apply a $P$-ordinary projector in the sense of Hida theory. We then show that we do obtain a direct summand after applying $P$-ordinary parts, and that this is precisely the completed cohomology of the locally symmetric spaces for $M$ (with an appropriate twist and degree-shift). This is inspired by arguments with ordinary parts developed simultaneously in~\cite[\S 5]{accghlnstt}, but the key difference in this paper is that we consider $P$-ordinary parts.   

\subsection{Strategy}\label{strategy} We now give a short sketch of the proof of Theorem~\ref{unified main theorem}. In fact, we think it might be helpful to the reader to first explain (a modified version of) the proof of Theorem~\ref{scholze's thm}. The perfectoid Shimura variety, minimally compactified, admits a Hodge--Tate period morphism
\[
\pi_{\mathrm{HT}}\colon (\cX^*_{\Gamma(p^\infty)})_{\et}\to |\Fl_{G,\mu}|, 
\] 
where $|\Fl_{G,\mu}|$ is the topological space underlying a certain flag variety of dimension $d$ and we think of $\pi_{\HT}$ as a morphism of sites. The Hodge--Tate period morphism has the property that $\Fl_{G,\mu}$ has an affinoid cover such that the preimage under $\pi_{\HT}$ of every member of the cover is an affinoid perfectoid space, with (strongly) Zariski closed boundary. This implies that the fibers of $\pi_{\HT}$ over points of rank $1$ (a condition that can be formulated on the level of the topological space) are affinoid perfectoid spaces with (strongly) Zariski closed boundary. To prove Theorem~\ref{scholze's thm}, it is enough to show that 
\[
H^i_{\et}(\cX^*_{\Gamma(p^\infty)}, j_{!}\F_p)=0\ \mathrm{for}\ i>d,
\]
where $j\colon\cX_{\Gamma(p^\infty)}\hookrightarrow \cX^*_{\Gamma(p^\infty)}$ is the open immersion. The primitive comparison theorem in $p$-adic Hodge theory (in the form \cite[Theorem 3.13]{scholze-survey}; see also \cite[\S 3, Theorem 8]{faltings}) reduces us to proving that 
\[
H^i_{\et}\left(\cX^*_{\Gamma(p^\infty)}, j_{!}(\cO^+/p)^a\right)=0\ \mathrm{for}\ i>d.
\] 
We can compute the latter cohomology groups using the Leray spectral sequence for $\pi_{\HT}$. Using the fact that the cohomological dimension of the topological space $|\Fl_{G,\mu}|$ is bounded by $d$ (which follows from a theorem of Scheiderer), we see that it is enough to prove for every point $x\in \Fl_{G,\mu}$ that
\[
R^i\pi_{\HT,*}j_{!}(\cO^+/p)^a_x=0\ \mathrm{for}\ i>0.
\]
We reduce to the case of points of rank $1$, where we use the fact that the fibers of $\pi_{\HT}$ are affinoid perfectoid spaces with (strongly) Zariski closed boundary. In this case, we have the desired vanishing for the \'etale cohomology of $j_{!}(\cO^+/p)^a$ (see Proposition~\ref{strongly Zariski closed vanishing}).   

We now explain how to adapt these ideas to prove Theorem~\ref{unified main theorem}, focusing on the case $K_{p,m}=\Gamma_1(p^m)$ (the argument for the general case is identical). We want to use the Leray spectral sequence for the Hodge--Tate period morphism at level $\Gamma_1(p^\infty)$, namely the morphism of sites
\[
\pi_{\HT/N_0}\colon (\cX^*_{\Gamma_1(p^\infty)})_{\et}\to |\Fl_{G,\mu}|/N_0,
\]
where $N_0=\cap_{m\geq 0}\Gamma_1(p^m)$ and we take the quotient $|\Fl_{G,\mu}|/N_0$ on the level of topological spaces. The sheaves $R^i\pi_{\HT/N_0,*}j_{!}(\cO^+/p)^a$ will in general not vanish for $i>0$. This is related to the fact that the Shimura variety $\cX^*_{\Gamma_1(p^\infty)}$ is not a perfectoid space; it is best thought of in the category of diamonds. However, we show that $|\Fl_{G,\mu}|/N_0$ admits a stratification by locally closed strata $|\Fl^w_{G,\mu}|/N_0$, induced from the generalized Bruhat stratification of the algebraic flag variety $\Flr_{G,\mu}$ by Schubert cells, such that for every $x\in |\Fl^w_{G,\mu}|/N_0$ we have
\[
R^i\pi_{\HT/N_0,*}j_{!}(\cO^+/p)^a_x=0\ \mathrm{for}\ i>d-\dim \Fl^w_{G,\mu}.
\]
This is enough to prove the desired vanishing result for compactly supported cohomology. 

In order to control the stalks of $R^i\pi_{\HT/N_0,*}j_{!}(\cO^+/p)^a$ over different Schubert cells, we use crucially that certain open strata in the compactified Shimura variety at intermediate infinite levels are perfectoid. The precise result is Theorem~\ref{general perfectoid stratum}. We prove this theorem by generalizing Scholze's theory of the overconvergent anticanonical tower and combining it with an argument inspired by work of Ludwig~\cite{ludwig}. (A technical detail is that in the course of the proof we replace our connected Shimura varieties with the PEL-type Shimura varieties associated to the corresponding similitude groups. As we are working over a perfectoid field $(C,\cO_C)$, the difference is only on the level of connected components.)  

\begin{example} In the modular curve case, 
we have $G=\mathrm{GL}_2/\Q$ and $\Fl_{G,\mu} = \mathbb{P}^{1,\mathrm{ad}}$. 
The Schubert cells are indexed by the Weyl group of $\mathrm{GL}_2/\Q_p$, 
which consists of two elements $\{1,w\}$ with $w^2 = 1$. The non-canonical locus
is the Schubert cell corresponding to $1$, which can be identified with $\mathbb{A}^{1,\mathrm{ad}}$ and 
is therefore $1$-dimensional. The canonical locus is the Schubert cell corresponding to $w$,
which can be identified with the point $\infty\in \mathbb{P}^{1,\mathrm{ad}}$, 
and is therefore $0$-dimensional. For any rank one point $x\in \mathbb{A}^{1,\mathrm{ad}}$,
the stalk $R\pi_{\HT/N_0,*}j_{!}(\cO^+/p)^a_x$ is only supported in degree $0$
because the non-canonical locus is already perfectoid at level $\Gamma_1(p^\infty)$, 
cf.~\cite{ludwig}. On the other hand, the stalk $R\pi_{\HT/N_0,*}j_{!}(\cO^+/p)^a_\infty$
is only supported in degrees $0$ and $1$. This comes from the cohomological dimension
(for continuous group cohomology) of $\Z_p$: the canonical locus 
is not perfectoid itself, but it has a ``$\Z_p$-cover'', in a sense made precise in
\S~\ref{diamonds and cohomology}, which is perfectoid.  
\end{example}

\begin{remark}\label{explanation for p split} We explain why we need to impose the condition 
that $p$ splits completely in $F$. Roughly, the idea is that the geometry of $\Fl_{G,\mu}$ is 
controlled by the absolute Weyl group of $G_{\Q_p}$, whereas the action of $G(\Q_p)$ on 
$\Fl_{G,\mu}$ only gives us access to the relative Weyl group of $G_{\Q_p}$. This means 
that, when $G_{\Q_p}$ is not split, there will be in general ``absolute'' Schubert cells in 
$\Fl_{G,\mu}$ over which we will not have optimal control. 

For example, let $F=F^+\cdot F_0$, with $F^+$ real quadratic and $F_0$ imaginary quadratic. 
Assume that $p$ splits in $F_0$ but stays inert in $F^+$. Consider a $\mathrm{U}(1,1)$-Shimura variety 
defined with respect to $F/F^+$. Then $\Fl_{G,\mu}$ can be identified with $\mathrm{Res}_{F^+_{p}/\Q_p}\mathbb{P}^{1,\ad}$,
which over $\Spa(C,\cO_C)$ gives the product of two copies of $\mathbb{P}^{1,\ad}_{C}$. Inside $\Fl_{G,\mu}$, 
we have $\mathbb{A}^{1,\ad}_{C}\times \mathbb{A}^{1,\ad}_{C}$, which is the non-canonical locus, and $\infty\times\infty$, which is the canonical locus. 
We can control both of these loci using the methods developed in this paper. However, we also 
have the absolute Schubert cells $\mathbb{A}^{1,\ad}_{C}\times \infty$ and $\infty\times \mathbb{A}^{1,\ad}_{C}$, 
over which we cannot optimally bound the degrees in which $R\pi_{\HT,*}j_{!}(\cO^+/p)^a$ is supported. 

Assume that $F=F^+\cdot F_0$, where $F^+$ is a totally real field of degree $f$. Furthermore, 
assume that $p$ splits in $F_0$, but has arbitrary behavior in $F^+$. By embedding 
a $\mathrm{U}(n,n)$-Shimura variety defined with respect to $F/F^+$ into a $\mathrm{U}(nf, nf)$-Shimura variety 
defined with respect to $F_0/\Q$, we obtain some non-trivial bound such that compactly supported 
cohomology of the $\mathrm{U}(n,n)$-Shimura variety at level $\Gamma_1(p^\infty)$ vanishes for all degrees 
higher than this bound. Unfortunately, this bound is usually higher than the middle degree 
of cohomology. Still, this exercise suggests that Theorem~\ref{unified main theorem} is likely to hold more generally. 
\end{remark}

\subsubsection{Organization of the paper} 
In Section~\ref{diamonds and cohomology}, we collect some results about diamonds and their cohomology that will play a key role in the rest of the paper. This includes new results that may be independently useful such as a Hochschild--Serre spectral sequence for the compactly supported \'etale cohomology of diamonds (Theorem~\ref{hochschildserre}). 
Section~\ref{recollections} is devoted to preliminaries on Shimura varieties. This essentially consists of extending the results of~\cite{scholze-galois} to our Shimura varieties (for the corresponding similitude groups). One subtlety is that we need to consider the anticanonical tower at more general levels. 
Section~\ref{stratifications} discusses the generalized Bruhat decomposition on the Hodge--Tate period domain and proves that certain open strata in the Shimura variety at intermediate levels already have a perfectoid structure (Theorem~\ref{general perfectoid stratum}). 
Section~\ref{vanishing} gives a proof of the main theorem. A key geometric input comes from Proposition~\ref{rational neighborhood}, where we study adic spaces equipped with the action of a profinite group and construct invariant rational neighborhoods. 
Section~\ref{Borel-Serre} introduces the locally symmetric spaces for $\GL_n/F$ and 
gives the application to Galois representations using the derived Hecke algebra introduced by Newton--Thorne. 

\subsection{Acknowledgements} We are very grateful to Peter Scholze for suggesting the idea for this project and for many useful conversations. We are also very grateful to James Newton for many useful conversations and in particular for providing us with a complete proof of Theorem~\ref{factors through classical} (essentially due to~\cite{newton-thorne}), which is crucial for the arguments in Section~\ref{Borel-Serre}. We also thank Ahmed Abbes, Patrick Allen, David Hansen, Kiran Kedlaya, Judith Ludwig and Jack Thorne for useful comments and/or conversations. We are grateful to Mathilde Gerbelli-Gauthier for her contributions to this project at an initial stage. 

This project arose as a continuation of A.C. and C.J.'s project group at the 2017 Arizona Winter School. We thank the organizers of the A.W.S. for the wonderful experience, and we thank Andrea Dotto, Mohammed Zuhair Mullath, Lue Pan, Mafalda Santos, Jun Su and Peng Yu for their participation in the project group. The Winter School was supported by NSF grant DMS-1504537 and by the Clay Mathematics Institute. 

\section{Diamonds and cohomology}\label{diamonds and cohomology}

In this section we recall the definition of diamonds and some facts about them and their cohomology from \cite{diamonds}. For set-theoretic considerations we refer to \cite[\S 4]{diamonds}; we choose a cardinal $\kappa$ as in \cite[Lemma 4.1]{diamonds}, which allows us to consider countable inverse limits of rigid spaces over some (fixed) non-archimedean field, and all perfectoid spaces will be tacitly assumed to be $\kappa$-small (\cite[Definition 4.3]{diamonds}). We will make no further mention of set-theoretic considerations. For details and precise definitions of any concept or terminology, we refer to \cite{diamonds}. We will write ``$(K,K^{+})$ is a perfectoid field'' to mean that $K$ is a perfectoid field and $K^{+}$ is an open and bounded \emph{valuation} subring.

\subsection{Generalities}

To start, we define some categories of perfectoid spaces. The category of all perfectoid spaces will be denoted by $\Pfd$ and the full subcategory of all perfectoid spaces in characteristic $p$ will be denoted by $\Perf$. For a fixed perfectoid Huber pair $(R,R^{+})$, let $\Pfd_{(R,R^{+})}$ denote the slice category of $\Pfd$ of perfectoid spaces over $\Spa(R,R^{+})$. If $R^{+}=R^{\circ}$, we will simply write $\Pfd_{R}$. The tilting equivalence allows us to identify $\Pfd_{(R,R^{+})}$ with $\Pfd_{(R^{\flat},R^{\flat +})}$, which is a slice of $\Perf$.

The category $\Pfd$ carries two important Grothendieck topologies, the pro-\'etale topology and the v-topology, defined in \cite[Definition 8.1]{diamonds} (and they induce topologies with the same name on all the other categories defined above). Both topologies are subcanonical, cf.~\cite[Corollary 8.6, Theorem 8.7]{diamonds}, and we will conflate perfectoid spaces and their corresponding representable sheaves. A diamond is, by definition, a sheaf on $\Perf$ for the pro-\'etale topology which can be written as the quotient of a perfectoid space $X$ by a pro-\'etale equivalence relation \cite[Definitions 11.1, 11.2]{diamonds}. Diamonds turn out to be sheaves for the v-topology as well \cite[Proposition 11.9]{diamonds}. Any diamond $X$ has an underlying topological space, denoted by $|X|$. The underlying set of $|X|$ can be described as the set of equivalence classes of maps of diamonds
$$ x \colon \Spa(K,K^{+}) \to X, $$
where $(K,K^{+})$ is a perfectoid field of characteristic $p$, and two maps $x_{i} \colon \Spa(K_{i},K_{i}^{+}) \to X$ ($i=1,2$) are equivalent if there is a third perfectoid field $(K_{3},K^{+}_{3})$ (of characteristic $p$) with \emph{surjections} $f_{i} \colon \Spa(K_{3},K^{+}_{3}) \to \Spa(K_{i},K^{+}_{i})$ such that $x_{1}\circ f_{1} = x_{2} \circ f_{2}$ \cite[Proposition 11.13]{diamonds}.

All diamonds that appear in this paper are \emph{(locally) spatial} \cite[Definition 11.17]{diamonds}. If $X$ is (locally) spatial, then $|X|$ is a (locally) spectral space, and any quasicompact open subset $|U| \subseteq |X|$ defines an open subfunctor $U\subseteq X$, which is a (locally) spatial diamond \cite[Proposition 11.18, 11.19]{diamonds}. Any locally spatial diamond $X$ has an associated \'etale site $X_{\et}$ consisting of maps $Y \to X$ of diamonds which are \'etale (this implies that $Y$ is automatically locally spatial by \cite[Corollary 11.28]{diamonds}); see \cite[Definition 14.1]{diamonds}. Following \cite[Convention 10.2]{diamonds}, we require all \'etale maps to be locally separated. We also remark that fiber products of (locally) spatial diamonds are (locally) spatial \cite[Corollary 11.29]{diamonds}.

\medskip

There is a functor $Y \mapsto Y^{\lozenge}$ from the category of analytic adic spaces over $\Zp$ to the category of diamonds, defined in \cite[\S 15]{diamonds}. Any $Y^{\lozenge}$ is locally spatial, and it is spatial if and only if $Y$ is qcqs. Moreover, $|Y|=|Y^{\lozenge}|$. Of particular importance to us is that this functor induces an equivalence of \'etale sites $Y_{\et}\cong Y_{\et}^{\lozenge}$ \cite[Lemma 15.6]{diamonds}. In particular, we may compute the \'etale cohomology of rigid analytic varieties and perfectoid spaces using the associated diamonds.

\medskip

Next, we state a general result on inverse limits that we will use many times.

\begin{prop}\label{inverselimit}
Let $(X_{i})_{i\in I}$ be a cofiltered inverse system of (locally) spatial diamonds with qcqs transition maps. Then $X=\varprojlim_{i}X_{i}$ is a (locally) spatial diamond and the natural map $|X| \to \varprojlim_{i}|X_{i}|$ is a homeomorphism.
\end{prop}

\begin{proof}
This is (a special case of) \cite[Lemma 11.22]{diamonds}; we refer to the statement there for precise set-theoretic conditions, which will always be satisfied in the applications in this paper.
\end{proof}

We finish this subsection with some topological considerations. Let $X$ be a spatial diamond and let $S$ be a spectral space. Assume that we have a surjective spectral map $p \colon |X| \to S$. Given $s\in S$, set $\Gen(s)\defeq \bigcap_{V}V$ where $V$ ranges through the quasicompact opens containing $s$; this is the set of generalizations of $s$. We can put a canonical (spatial) diamond structure on $p^{-1}(\Gen(s))$, by defining $p^{-1}(\Gen(s))=\varprojlim_{V}p^{-1}(V)$, where $p^{-1}(V)$ is a spatial diamond as mentioned above, and using Proposition \ref{inverselimit}. Following a common abuse of notation in the theory of adic spaces, we will simply write $p^{-1}(s)$ for $p^{-1}(\Gen(s))$ and refer to it as the (topological) fiber of $p$ at $s$. When $S=\pi_{0}(|X|)$ and $p$ is the canonical projection, this gives a canonical diamond structure on each connected component of $|X|$.

\subsection{Some results on cohomology}

In this subsection we will record some results on the \'etale cohomology of diamonds that we will use in this paper. We start with a result on inverse systems.

\begin{prop}\label{invlimcoh}
Let $X_{i}$, $i\in I$, be a cofiltered inverse system of spatial diamonds with inverse limit $X$. Let $i\in I$ and assume that $\cF_{i}$ is an abelian sheaf on $(X_{i})_{\et}$, with pullback $\cF_{j}$ to $X_{j}$ for $j\geq i$ and pullback $\cF$ to $X$. Then the natural map
$$ \varinjlim_{j\geq i} H_{\et}^{q}(X_{j},\cF_{j}) \to H^{q}_{\et}(X,\cF) $$
is an isomorphism for all $q$.
\end{prop}

\begin{proof}
This is \cite[Proposition 14.9]{diamonds}.
\end{proof}

Next, we give a generalization of a base change property \cite[Lemma 4.4.1]{caraiani-scholze} to the setting of diamonds. To state it, we first discuss ``geometric points'' on $X_{\et}$, for $X$ a locally spatial diamond. This discussion is local, so we may assume that $X$ is spatial. For any $x\in |X|$, consider the cofiltered category $I_{x}$ of maps $f \colon U \to X$ which are composites of quasicompact open immersions and finite \'etale maps and which satisfy $x\in |f(U)|$ (note that this condition is stable under fiber products over $X$, using \cite[Proposition 12.10]{diamonds} to verify the last condition).
We may form the inverse limit
$$ \ol{x} = \varprojlim_{(U \to X) \in I_{x}} U; $$
this is a spatial diamond by Proposition \ref{inverselimit}. We call it the \emph{geometric point} above $x$. 

\begin{lemma}\label{geometricpt}
We have $\ol{x}=\Spa(C,C^{+})$ for some algebraically closed perfectoid field $(C,C^{+})$, and the image of the natural map $\pi \colon \ol{x} \to X$ contains $x$.
\end{lemma}

\begin{proof}
We begin with the last assertion; we claim that $\pi(|\ol{x}|)=\{y\in |X| \mid y\in |f(U)|\,\, \forall\,(f\colon U \to X) \in I_{x}\}=\Gen(x)$. The second equality follows since \'etale maps are open and $I_{x}$ contains all open neighborhoods of $x$. For the first, note that all $|f| \colon |U| \to |X|$ are spectral maps of spectral spaces, so we may use the constructible topology. Since the inclusion $\subseteq$ is clear, it remains to prove the opposite inclusion, which amounts to showing that $\varprojlim_{f \in I_{x}} |f|^{-1}(y)$ is non-empty as long as all $|f|^{-1}(y)$ are non-empty. The $|f|^{-1}(y)$ are compact Hausdorff spaces, so this is true.

\medskip

We now prove the first part. We have shown that $\ol{x}\neq \emptyset$, so by \cite[Propositions 7.16, 11.26]{diamonds} it suffices to show that $\ol{x}$ is connected and that every surjective \'etale map $Y \to \ol{x}$ which can be written as a composite of quasicompact open immersions and finite \'etale maps splits. For connectedness, note that if $\ol{x}=V_{1} \sqcup V_{2}$ then this disconnection comes from some disconnection $(f \colon U=U_{1} \sqcup U_{2} \to X) \in I_{x}$. Without loss of generality $x\in |f(U_{1})|$, but then $(U_{1} \to X) \in I_{x}$ and so $V_{2}$, which is the preimage of $U_{2}$, must be empty, so $\ol{x}$ is connected. To see that any $Y \to \ol{x}$ as above splits, we first note that by \cite[Proposition 11.23]{diamonds} it must come via pullback from some $V \to U $, where $(U \to X) \in I_{x}$ and $V \to U$ is the composite of quasicompact open immersions and finite \'etale maps. Since $Y \to \ol{x}$ is surjective, $x$ is in the image of the natural map $|Y| \to |X|$. Since this map factors via $|Y| \to |V|$, this forces $x$ to be in the image of the natural map $|V| \to |X|$ and hence $(V \to X) \in I_{x}$. Since $V \to U$ splits over $V \to U$, this implies that $Y \to \ol{x}$ splits.
\end{proof}

\begin{cor}
Let $X$ be a locally spatial diamond, let $x\in |X|$ and let $\cF$ be an \'etale sheaf on $X$. The assignment $\cF \mapsto \cF_{\ol{x}}=\varinjlim_{(U \to X) \in I_{x}} \cF (U)$ defines a topos-theoretic point of $X_{\et}$, and the collection $\{\ol{x} \mid x\in |X| \}$ is a conservative family (i.e. $\cF =0$ if and only if $\cF_{\ol{x}}=0$ for all $x$).
\end{cor}

\begin{proof}
The morphism $\pi \colon \ol{x} \to X$ is quasi-pro-\'etale and maps the closed point of $|\ol{x}|$ to $x$ (since $|\pi|(|\ol{x}|)=\Gen(x)$ by the proof of Lemma \ref{geometricpt}). The corollary now follows from \cite[Proposition 14.3]{diamonds}.
\end{proof}

We now come to the base change result, generalizing \cite[Lemma 4.4.1]{caraiani-scholze}. The result is a special case of \cite[Corollary 16.10(ii)]{diamonds} (we thank David Hansen for pointing this out to us), but we give the (short and simple) proof as the same argument will used in other places.

\begin{prop}\label{basechange}
Let $f \colon Y \to X$ be a qcqs map of locally spatial diamonds and let $x\in |X|$ with corresponding geometric point $\ol{x}$. Set $Y_{\ol{x}}=Y\times_{X}\ol{x}$ and let $\cF$ be an abelian sheaf on $Y_{\et}$. Then the natural map
$$ (R^{i}f_{\et,\ast}\cF)_{\ol{x}} \to H^{i}_{\et}(Y_{\ol{x}},\cF) $$
is an isomorphism for all $i$ (here $\cF$ also denotes the pullback of $\cF$ to $Y_{\ol{x}}$).
\end{prop}

\begin{proof}
Without loss of generality, $X$ is spatial. Since $\ol{x}=\varprojlim_{(U \to X) \in I_{x}}U$ (using the notation above), we have $Y_{\ol{x}} = \varprojlim_{(U \to X) \in I_{x}}f^{-1}(U)$. By Proposition \ref{invlimcoh}, we have
$$ (R^{i}f_{\et,\ast}\cF)_{\ol{x}} = \varinjlim_{(U\to X)\in I_{x}} H^{i}_{\et}(f^{-1}(U),\cF) \cong H^{i}_{\et}(Y_{\ol{x}},\cF) $$
via the natural map, as desired.
\end{proof}

We will also require a ``topological'' version of the base change result.

\begin{prop}\label{topbasechange}
Let $X$ be a spatial diamond and let $S$ be a spectral space. Assume that there is a spectral map $|f| \colon |X| \to S$; precomposing with the natural morphism of sites $X_{\et} \to |X|$ gives a morphism of sites $f \colon X_{\et} \to S$. Let $s\in S$ and consider the spatial diamond $f^{-1}(s)$ defined just before this section. If $\cF$ is an abelian sheaf on $X_{\et}$, then the natural map
$$ (R^{i}f_{\ast}\cF)_{s} \to H^{i}_{\et}(f^{-1}(s),\cF) $$
is an isomorphism for all $i$.
\end{prop}

\begin{proof}
We write $\Gen(s)=\bigcap_{s\in U}U$ with $U$ quasicompact open and then proceed as in the proof of Proposition \ref{basechange}.
\end{proof}

We now move on to a version of the Hochschild--Serre spectral sequence involving the lower shriek functor. For this, we will need to define the lower shriek functor for open immersions and the notion of $G$-torsors for profinite groups $G$. We start with the lower shriek functor. Since what we need is rather elementary, we give the definitions instead of appealing to the more abstract constructions of \cite{diamonds}, such as e.g.\ \cite[Definition/Proposition 19.1]{diamonds}. Let $j \colon U \to X$ be an \'etale map of locally spatial diamonds. As usual, we may define an extension by zero functor $j_{!}$ from sheaves on $U_{\et}$ to sheaves on $X_{\et}$ by sheafifying the presheaf
$$ \left( V \overset{f}{\to} X \in X_{\et} \right) \mapsto \bigsqcup_{(V \overset{g}{\to} U) \in U_{\et};\,j\circ g=f}\cF(V \overset{g}{\to} U). $$
One checks that $j_{!}$ is exact and is a left adjoint to $j^{\ast}$. We record the following base change results.

\begin{lemma}\label{shriekbasechange}
Let $f \colon X^{\prime} \to X$ be a map of spatial diamonds and let $j \colon U \to X$ be an \'etale map with pullback $j^{\prime} \colon U^{\prime} \to X^{\prime}$. Denote the map $ U^{\prime} \to U$ by $g$.
\begin{enumerate}
\item Let $\cF$ be an \'etale sheaf on $U$. Then the natural map $j^{\prime}_{!}g^{\ast}\cF \to f^{\ast}j_{!}\cF $ is an isomorphism.

\item Assume that $j$ is a partially proper open immersion and that $g$ is an isomorphism; we use it to identify $U^{\prime}$ and $U$ and think of $g$ as the identity. Let $\cG$ be an abelian sheaf on $U_{\et}$. Then the natural map $j_{!}\cG \to Rf_{\ast}j_{!}^{\prime}\cG$ is an isomorphism.
\end{enumerate}
\end{lemma}

\begin{proof}
We begin with part (1). It suffices to check this on geometric points of $X^{\prime}$, so let $z\in |X^{\prime}|$ and let $\ol{z}$ be the corresponding geometric point. The assertion then reduces to checking that $(f^{\ast}j_{!}\cF)_{\ol{z}}=0$ if $\ol{z} \to X^{\prime}$ does not factor through $U^{\prime} \to X^{\prime}$. As usual, one has $(f^{\ast}j_{!}\cF)_{\ol{z}} \cong (j_{!}\cF)_{\ol{f(z)}}$.
If $\ol{z}$ does not factor through $U^{\prime} \to X^{\prime}$, then $\ol{f(z)}$ does not factor through $U \to X$, so $(j_{!}\cF)_{\ol{f(z)}}=0$ and we get the conclusion.

\medskip
We now prove part (2). As above, it suffices to show that $(Rf_{\ast}j^{\prime}_{!}\cG)_{\ol{x}}=0$ for any $x\in |X|\setminus |U|$. By Proposition \ref{basechange}, we have
$$ (R^{i}f_{\ast}j^{\prime}_{!}\cG)_{\ol{x}} \cong H^{i}_{\et}(X^{\prime}_{\ol{x}},h^{\ast}j^{\prime}_{!}\cG) $$
for all $i$, where $h$ is the map $X^{\prime}_{\ol{x}} \to X^{\prime}$. By part (1), the stalk of $h^{\ast}j^{\prime}_{!}\cG$ is $0$ over any geometric point outside $U\times_{X^{\prime}}X^{\prime}_{\ol{x}}$. Since $j$ is partially proper and $x\in |X|\setminus |U|$, $\Gen(x) \cap |U| = \emptyset $ and hence $U\times_{X^{\prime}}X^{\prime}_{\ol{x}} =\emptyset$. Thus, $h^{\ast}j^{\prime}_{!}\cG=0$ and the conclusion follows.
\end{proof}

\medskip

Next, we move on to torsors. Let $G$ be a profinite group. Following \cite[Definition 10.12]{diamonds}, we define a v-sheaf $\ul{G}$
$$ X \mapsto \Hom_{cts}(|X|,G) $$
on $\Perf$, and define a (right) $G$-torsor to be a map $f \colon \widetilde{X} \to X$ of locally spatial diamonds with a right action of $\ul{G}$ on $\widetilde{X}$ over $X$ such that the induced map $\widetilde{X}\times \ul{G} \to \widetilde{X} \times_{X}\widetilde{X}$ is an isomorphism (this is equivalent to the definition given in \cite[Definition 10.12]{diamonds} by \cite[Lemma 10.13]{diamonds}). We may now state and prove our Hochschild--Serre spectral sequence.

\begin{thm}\label{hochschildserre}
Let $G$ be a profinite group and let $\Sigma$ be a set of open normal subgroups of $G$ which form a basis of neighborhoods of the identity. Assume also that $G\in \Sigma$. We let $(X_{N})_{N\in \Sigma}$ be an inverse system of spatial diamonds and we assume that each $X_{N}$ carries a right action of $\ul{G/N}$, and that these actions are compatible. Set $X=X_{G}$ and $\wt{X}=\varprojlim_{N\in \Sigma}X_{N}$.
The spatial diamond $\wt{X}$ carries a right $\ul{G}$-action and we assume that the natural map $\pi \colon \wt{X} \to X$ is surjective on topological spaces. Assume that there is a partially proper open immersion $j \colon U \to X$ with pullback $j_{N} \colon U_{N} \to X_{N}$ such that $U_{N} \to U$ is a $G/N$-torsor. If $\wt{j} \colon \wt{U} \to \wt{X} $ is the pullback, then $\wt{U} \to U$ is a $G$-torsor. Let $\cF$ be an abelian sheaf on $U_{\et}$. Then there is a spectral sequence
$$ E_{2}^{rs}=H^{r}_{cts}(G,H^{s}_{\et}(\wt{X},\wt{j}_{!}\cF)) \implies H^{r+s}_{\et}(X,j_{!}\cF), $$
where $H^{\ast}_{cts}(G,-)$ denotes the continuous group cohomology for $G$, and $H^{\ast}_{\et}(\wt{X},\wt{j}_{!}\cF)$ is given the discrete topology.
\end{thm}

\begin{proof}
There will be a lot of abuse of notation in this proof regarding pullbacks and lower shrieks; these cause no mathematical trouble by Lemma \ref{shriekbasechange}(1). We will use \cite[Propositions 14.7, 14.8]{diamonds} repeatedly in this proof, which together imply that \'etale cohomology of locally spatial diamonds may be computed on the v-site. Thus, we have $H_{\et}^{\ast}(X,j_{!}\cF)=H_{v}^{\ast}(X,j_{!}\cF)$. The map $\pi$ is a v-cover by \cite[Lemma 12.11]{diamonds} since it is quasicompact and surjective on topological spaces. We therefore have a \v{C}ech-to-derived functor spectral sequence
$$ E_{2}^{rs}=\check{H}^{r}(\mf{U},\cH^{s}(j_{!}\cF)) \implies H^{r+s}_{v}(X,j_{!}\cF) $$
where $\mf{U}$ is the v-cover $(\wt{X} \to X)$ and $\cH^{s}$ denotes the derived functors of the inclusion from v-sheaves to v-presheaves on $\wt{X}$. It remains to show that
$$ \check{H}^{r}(\mf{U},\cH^{s}(j_{!}\cF))=H^{r}_{cts}(G,H_{\et}^{s}(\wt{X},\wt{j}_{!}\cF)). $$
The \v{C}ech complex computing the left-hand side is
\[ H_{\et}^{s}(\wt{X},\wt{j}_{!}\cF) \to H_{\et}^{s}(\wt{X}\times_{X}\wt{X},\wt{j}_{!}\cF) \to \dotsb \to H_{\et}^{s}((\wt{X}_{/X})^{n},\wt{j}_{!}\cF) \to \dotsb, \label{cech} \tag{$\ast$} \]
where $(\wt{X}_{/X})^{n}$ denotes the fiber product $\wt{X} \times_{X} \dots \times_{X}\wt{X}$ with $n$ factors, and we have equated v-cohomology and \'etale cohomology. There is a morphism
$$ f=(f_{n})_{n\geq 1} \colon (\wt{X}\times \ul{G}^{n-1})_{n\geq 1} \to ((\wt{X}_{/X})^{n})_{n\geq 1}, $$
$$ f_{n}(x,g_{1},\dots,g_{n-1})=(x,xg_{1},\dots,xg_{n-1}), $$
of simplicial spatial diamonds. Since $\wt{U} \to U$ is a $G$-torsor, $f$ restricts to an isomorphism 
$$ (\wt{U}\times \ul{G}^{n-1})_{n\geq 1} \cong ((\wt{U}_{/U})^{n})_{n\geq 1}. $$ 
By abuse of notation, we also write $\wt{j}$ for the open immersions $\wt{U}\times\ul{G}^{n-1} \to \wt{X} \times \ul{G}^{n-1}$ and the open immersions $(\wt{U}_{/U})^{n} \to (\wt{X}_{/X})^{n}$. The natural map $\wt{j}_{!}\cF \to Rf_{n,\ast}\wt{j}_{!}\cF$ is an isomorphism of \'etale sheaves on $(\wt{X}_{/X})^{n}$ by Lemma \ref{shriekbasechange}(2), since the immersion $(\wt{U}_{/U})^{n} \to (\wt{X}_{/X})^{n}$ is partially proper (this follows from partial properness of $U \to X$ because partial properness is stable under base change). Thus, the complex (\ref{cech}) is equal to
$$ H_{\et}^{s}(\wt{X},\wt{j}_{!}\cF) \to H_{\et}^{s}(\wt{X}\times \ul{G},\wt{j}_{!}\cF) \to \dotsb \to H_{\et}^{s}(\wt{X}\times \ul{G}^{n-1},\wt{j}_{!}\cF) \to \dotsb $$
which in turn is equal to the direct limit over $N\in \Sigma$ of the complexes
$$ H_{\et}^{s}(X_{N},j_{N,!}\cF) \to H_{\et}^{s}(X_{N}\times \ul{G/N},j_{N,!}\cF) \to \dotsb \to H_{\et}^{s}(X_{N}\times \ul{G/N}^{n-1},j_{N,!}\cF) \to \dotsb $$
by Proposition \ref{invlimcoh}. As these complexes compute $H^{\ast}(G/N, H^{s}_{\et}(X_{N},j_{N,!}\cF))$, we deduce that $\check{H}^{r}(\mf{U},\cH^{s}(j_{!}\cF))=H^{r}_{cts}(G,H_{\et}^{s}(\wt{X},\wt{j}_{!}\cF))$, as desired, by taking direct limits.
\end{proof}

We finish this section with some facts about spectral spaces that we will need. First, recall that if $S$ is a spectral space, then it carries a notion of (Krull) dimension $\dim S$, where $1+\dim S$ is defined as the supremum of the lengths of finite chains of specializations. 

\begin{lemma}\label{spectral quotient}
Let $S$ be a spectral space with a continuous right action of a profinite group $G$. Then $S/G$ is a spectral space, and the natural map $S \to S/G$ is spectral, open and generalizing. Moreover, if $\dim S$ is finite, then so is $\dim S/G$ and $\dim S = \dim S/G$.
\end{lemma}

\begin{proof}
This is \cite[Lemma 3.2.3]{extensions} (note that openness of the quotient map is automatic for group quotients).
\end{proof}

We will also need the following:

\begin{lemma}\label{homeomorphism onto orbit}
Let $S$ be a spectral space with a continuous right action of a profinite group $G$. Let $s\in S$ be a point with no proper generalizations and let $G_{s} \subseteq G$ be the stabilizer of $s$. Then $G_{s}$ is closed and the natural map $G_{s}\backslash G \to S$ given by $G_{s}g \mapsto sg$ is a homeomorphism onto the orbit of $s$ (with the subspace topology).
\end{lemma}

\begin{proof}
The map is continuous and bijective and $G_{s}\backslash G$ is quasicompact, so the orbit is quasicompact and it suffices to show that it is also Hausdorff. Since $G$ acts by homeomorphisms, all $sg$ have no proper generalizations. The result now follows since the subspace of points with no proper generalizations is Hausdorff by \cite[Tag 0904]{stacks-project}.
\end{proof}

\subsection{Structure sheaves}\label{structure sheaves}

In this section, we fix a perfectoid field $(K,K^{+})$ and a pseudouniformizer $\varpi\in K$ and work with diamonds over $\Spa(K,K^{+})$. Let $X$ be a locally spatial diamond over $(K,K^{+})$. Consider the quasi-pro-\'etale site $X_{\qp}$ which consists of all quasi-pro-\'etale maps $Y \to X$ of locally spatial diamonds (see \cite[Definition 10.1]{diamonds})
with the quasi-pro-\'etale topology \cite[Definition 14.1]{diamonds}. By the definition of diamonds, the perfectoid spaces which are quasi-pro-\'etale over $X$ form a basis of $X_{\qp}$. Since we are working over $(K,K^{+})$, we think of these as perfectoid spaces over $(K,K^{+})$, using the tilting equivalence. The assignment
$$ Y \mapsto \cO_{Y}^{+}(Y) $$
defines a sheaf $\cO_{X}^{+}$ (of $K^{+}$-modules) on perfectoid spaces in $X_{\qp}$ by \cite[Theorem 8.7]{diamonds} (indeed even for the v-topology). It extends uniquely to a sheaf on $X_{\qp}$ which we will denote by $\cO_{X}^{+}$. We also get the quotient sheaf $\cO_{X}^{+}/\varpi$, which is what we will be interested in. Note that if $f\colon Y \to X$ is a quasi-pro-\'etale morphism of locally spatial diamonds over $(K,K^{+})$, then $f^{\ast}(\cO^{+}_{X}/\varpi)=\cO^{+}_{Y}/\varpi$ since $Y_{\qp}$ is a slice of $X_{\qp}$.

\begin{remark}
We emphasize that the sheaf $\cO_{X}^{+}$ depends not only on the diamond $X$, but also on the structure map $X \to \Spa(K,K^{+})$. We have nevertheless chosen not to include $(K,K^{+})$ in the notation, in part to keep it simple, and also because all diamonds (and perfectoid spaces) in this paper will naturally live over a perfectoid field which should be clear from the context. We hope that this does not cause any confusion.
\end{remark}

We thank Peter Scholze for pointing out that the map in the following lemma should be an isomorphism, and not merely an almost isomorphism (the fact that it is an almost isomorphism suffices for the applications in this paper).

\begin{lemma}
Let $X$ be a locally spatial diamond over $(K,K^{+})$. Let $\nu \colon X_{\qp} \to X_{\et}$ be the natural morphism of sites. Then the natural map $\nu^{\ast}\nu_{\ast}(\cO^{+}_{X}/\varpi) \to \cO_{X}^{+}/\varpi$ is an isomorphism.
\end{lemma}

\begin{proof}
First, we remark that the category of \'etale sheaves on $X$ embeds fully faithfully into the category of quasi-pro-\'etale sheaves on $X$ by \cite[Proposition 14.8]{diamonds} via $\nu^{\ast}$. Thus, we may rephrase the Lemma as saying that $\cO_{X}^{+}/\varpi$ is an \'etale sheaf on $X$. By \cite[Theorem 14.12(ii)]{diamonds}, this may be checked quasi-pro-\'etale locally on $X$ (indeed v-locally), so we may assume that $X$ is a strictly totally disconnected perfectoid space. In this case, quasi-pro-\'etale maps are the same as pro-\'etale maps by definition, and the topos of the site $X_{\proet}^{\mathrm{aff}}$ of affinoid pro-\'etale maps $Y \to X$ is equivalent to the topos of $X_{\qp}$, so we may work on $X_{\proet}^{\mathrm{aff}}$. So, let $Y=\varprojlim_{i}Y_{i} \to X$ be an affinoid pro-\'etale map with $Y=\Spa(A,A^{+})$, and with $Y_{i}=\Spa(A_{i},A_{i}^{+})$ \'etale over $X$. Let $F$ denote the sheaf quotient of $\cO_X^+$ by $\varpi$ on $X_\et$. Since $X$ has no higher \'etale cohomology, $F(Y_i)=A_i^+/\varpi$ for all $i$. We may then compute
\[
\nu^\ast F(Y) = \varinjlim_i F(Y_i) = \varinjlim_i A_i^+/\varpi = A^+/\varpi.
\]   
In particular, $\nu^\ast F$ is equal to the presheaf quotient of $\cO_X^+$ by $\varpi$ on $X_{\proet}^{\mathrm{aff}}$. But $\nu^\ast F$ is also a sheaf, so we must have $\nu^\ast F= \cO_X^+/\varpi$, which proves the lemma.
\end{proof}

Thus, we may think of $\cO_{X}^{+}/\varpi$ as a sheaf on $X_{\et}$. In this paper we will mostly be interested in the corresponding almost sheaf $(\cO^{+}_{X}/\varpi)^{a}$. We note that, on $X_{\qp}$, $(\cO^{+}_{X}/\varpi)^{a}$ can be described as the almost sheaf given by
$$ \left( \Spa(R,R^{+}) \to X \right) \mapsto (R^{+}/\varpi)^{a} $$
on affinoid perfectoid $\Spa(R,R^{+}) \to X \in X_{\qp}$. This follows from \cite[Proposition 8.5(iii)]{diamonds}, since $\Spa(R,R^{+})_{\qp}$ and $\Spa(R,R^{+})_{\proet}$ (defined as in \cite[Definition 8.1(ii)]{diamonds}) define the same topos. We have the following compatibility with the usual $\cO_{X}^{+}/\varpi$ on rigid spaces.

\begin{lemma}\label{compatibility of structure sheaves}
Let $Z$ be a rigid space over $(K,\cO_{K})$ and consider the usual sheaf $\cO_{Z}^{+}/\varpi$ on $Z_{\et}$. Then, under the equivalence $Z_{\et}\cong Z_{\et}^{\lozenge}$, $(\cO_{Z}^{+}/\varpi)^{a}=(\cO_{Z^{\lozenge}}^{+}/\varpi)^{a}$.
\end{lemma}

\begin{proof}
By \cite[Corollary 3.17]{scholze-padicHodge} we may compute the values of $(\cO_{Z}^{+}/\varpi)^{a}$ on the pro-\'etale site $Z_{\proet}$ as defined in \cite{scholze-padicHodge}; by the equivalence $Z_{\et}\cong Z^{\lozenge}_{\et}$ and Proposition \ref{inverselimit}, $Z_{\proet}$ is naturally a subcategory of $Z^{\lozenge}_{\qp}$, and any cover in $Z_{\proet}$ is a cover in $Z^{\lozenge}_{\qp}$ (but we make no assertion about the converse). By \cite[Lemma 4.10]{scholze-padicHodge}, $(\cO_{Z}^{+}/\varpi)^{a}(U)=(R^{+}/\varpi)^{a}$ on any affinoid perfectoid $U=\Spa(R,R^{+}) \to Z \in Z_{\proet}$. Since these form a basis for $Z_{\proet}$ and are also affinoid perfectoid as elements of $Z^{\lozenge}_{\qp}$, the lemma follows from the description of $(\cO_{Z^{\lozenge}}^{+}/\varpi)^{a}$ given just before this lemma.
\end{proof}

Let $X$ be a locally spatial diamond over $(K,K^{+})$. Since $(\cO^{+}_{X}/\varpi)^{a}$ on $X_{\et}$ is preserved under pullback by quasi-pro-\'etale maps, it will be convenient for us to record some (perhaps a priori surprising) examples of quasi-pro-\'etale maps.

\begin{prop}\label{finite maps}
Let $f \colon Y \to X$ be a finite map of rigid spaces over $(K,\cO_{K})$. Then the associated map $f^{\lozenge} \colon Y^{\lozenge} \to X^{\lozenge}$ of diamonds is quasi-pro-\'etale.
\end{prop}

\begin{proof}
  By \cite[Proposition 13.6]{diamonds}, we need to check that the pullback of $f^\lozenge$ along every $\Spa(C,\cO_{C}) \to X^{\lozenge}$, with $(C,\cO_{C})$ complete algebraically closed, is pro-\'etale. By definition of $X^{\lozenge}$, such a map corresponds precisely to a map $\Spa(C^{\sharp},\cO_{C^{\sharp}}) \to X$ of adic spaces over $(K,\cO_{K})$, where $C^{\sharp}$ is the untilt of $C$ over $K$. Since the diamondification functor commutes with fiber products, we are left with checking that if $Z \to \Spa(C^{\sharp},\cO_{C^{\sharp}})$ is finite, then $Z^{\lozenge} \to \Spa(C^{\sharp},\cO_{C^{\sharp}})$ is pro-\'etale. But, since $Z^{\lozenge}=(Z^\red)^{\lozenge}$ (by definition and reducedness of perfectoid rings) and $Z^\red$ is a finite set of copies of $\Spa(C^{\sharp},\cO_{C^{\sharp}})$, this is true.
\end{proof}

Note that in the case of closed immersions one can be more explicit; see (the argument in) \cite[Remark 7.9]{diamonds}.

\medskip

For our next result, we use the notion of a \emph{rank one point}, following \cite{extensions}. Let $X$ be a locally spatial diamond (not necessarily over $(K,K^{+})$). By definition (\cite[Definition 3.2.1]{extensions}), $x\in |X|$ is a rank one point if it satisfies any of the following equivalent conditions:
\begin{enumerate}
\item $x$ has no proper generalizations in $|X|$;

\item There is a perfectoid field $(L,\cO_{L})$ and a map $\Spa(L,\cO_{L}) \to X$ with topological image $x$;

\item $x= \bigcap |U|$ where $|U|$ ranges over the quasicompact opens of $|X|$ containing $x$.
\end{enumerate}
Equivalently, one may define the set of rank one points as the image of the ``Berkovich space'' $|X|^{B}$ inside $|X|$; see \cite[Definition 13.7]{diamonds}. We then get the following version of \cite[Lemma 4.4.2]{caraiani-scholze}.

\begin{lemma}\label{rankonepointsisom}
Let $X$ be a locally spatial diamond over $(K,K^{+})$ and let $\iota \colon U \to X$ be a quasicompact open immersion such that $|U|\subseteq |X|$ contains all rank one points of $|X|$. Let $j \colon V \to X$ be a partially proper open immersion; write $j$ also for the immersion $U \cap V \to U$. Then, for any $i$, the natural map
$$ H^{i}_{\et}(X,j_{!}(\cO^{+}_{V}/\varpi)^{a} ) \to H^{i}_{\et}(U,j_{!}(\cO^{+}_{U\cap V}/\varpi)^{a}) $$
is an isomorphism.
\end{lemma} 

\begin{proof}
It suffices to prove that the adjunction map $j_{!}(\cO_{V}^{+}/\varpi)^{a} \to R\iota_{\et,\ast}j_{!}(\cO_{U\cap V}^{+}/\varpi)^{a}$ is an isomorphism, which we may check on stalks on geometric points $\ol{x}$ of $X$. By Proposition \ref{basechange} and Lemma \ref{shriekbasechange}, this reduces us to proving the special case of the proposition when $X=\Spa(C,C^{+})$ for some algebraically closed extension $(C,C^{+})$ of $(K,K^{+})$. Since $U$ contains all rank one points of $X$, we must have $U=\Spa(C,C^{++})$ for some open bounded valuation subring $C^{++} \supseteq C^{+}$. Since $j$ is partially proper, we either have $V=X$ or $V=\emptyset$. In the latter case, the lemma is trivial. In the former case, we are left with checking that $C^{+}/\varpi \to C^{++}/\varpi $ is an almost isomorphism (which it is), and that higher cohomology vanishes, which it does since all \'etale covers of $X$ are split.
\end{proof}

\section{Preliminaries on Shimura varieties}
\label{recollections}
The following notation will be used through Section~\ref{vanishing}. 
We fix an integer $n\geq 1$. Let $F$ be a CM field with totally real subfield $F^+$; we allow both the totally real case $F=F^+$ and the imaginary CM case $F\neq F^+$. Let $f = [F:\QQ]$ and let $c$ denote complex conjugation in $\Gal(F/F^+)$ (this is independent of the choice of embedding $F \hookrightarrow \C$); we will also let $c$ denote complex conjugation on $\C$.
Let $p$ be a rational prime that \emph{splits completely} in $F$.  We fix an algebraically closed and complete extension $C$ of $\Q_p$ with ring of integers $\cO_C$ (one may take $C = \C_p$, for example). We fix a choice of $i=\sqrt{-1}\in \C$.

\subsection{Groups and Shimura varieties}\label{3.1}
\subsubsection{Symplectic and unitary similitude groups}\label{subsect:unitary_group} Let $\Psi_n$ denote the $n \times n$ matrix with $1$'s on the anti-diagonal and $0$'s elsewhere. We set
\[
J_n = \begin{pmatrix}
0 & \Psi_n \\ -\Psi_n & 0
\end{pmatrix}.
\]
Consider the free $\cO_F$-module $L \defeq \cO_F^{2n}$ of rank $2n$.
Then $J_n$ determines a non-degenerate alternating pairing
\[
  \psi \colon L \times L \rightarrow \Z,\quad \psi(x,y) \defeq \mathrm{Tr}_{\cO_F/\Z}(x^t J_n y^c),
\]
which is perfect when localized at any rational prime which is unramified in $F$. We define $G$ to be the group scheme over $\Z$ given on $R$-points, for $R$ a ring, by
\[
G(R) = \{ (g,r) \in \Aut_{\cO_F\otimes_{\Z}R}( L \tensor_\Z R) \times R^\times \mid \psi(gx,gy)=r \cdot \psi(x,y) \text{ for all $x,y \in L \tensor_\Z R$} \}.
\]
When $F$ is totally real, $G$ is a symplectic similitude group; when $F$ is imaginary CM, $G$ is a unitary similitude group (in both cases, we note that the similitude factor is required to lie in $\GG_{m,\Z}$). We remark that in a pair $(g,r)\in G(R)$, $r$ is determined uniquely by $g$ and hence projection onto the first factor defines an injection $G \to \mathrm{Res}_{\cO_F/\Z}(\GL_{2n,\cO_F})$. We will use this to represent elements of $G$ by $2 \times 2$ block $n\times n$ matrices, as with $J_n$ above. $G$ is manifestly a subgroup scheme of the \emph{symplectic similitude group} $\tG/\Z$, defined on $R$-points by 
\[
\tG(R) = \{ (g,r) \in \Aut_{R}( L \tensor_\Z R) \times R^\times \mid \psi(gx,gy)=r \cdot \psi(x,y) \text{ for all $x,y \in L \tensor_\Z R$} \}.
\]
Since $\psi$ is perfect when localized at primes $\ell$ unramified in $F$, we see that $\tG_{\Z_{(\ell)}} \cong \mathrm{GSp}_{2nf, \Z_{(\ell)}}$ for those $\ell$. 

\subsubsection{Shimura data}\label{subsubsect:Sh_datum}
In this section we recall the integral PEL data attached to $G$ and $\tG$, starting with $G$. In addition to the $\cO_F$-module $L$ and the pairing $\psi$, we need an $\R$-algebra homomorphism
\[
h \colon \C \to \End_{\cO_F \otimes_\Z \R}(L\otimes_\Z \R)
\] 
such that $\psi(h(z)x,y)=\psi(x,h(z^c)y)$ and such that $(x,y) \defeq \psi(x,h(i)y)$ is symmetric and positive definite. We set $h(i) = -J_n$; this uniquely determines the $\R$-algebra homomorphism $h$ and one checks by direct calculation that it has the required properties. The triple $(L,\psi,h)$ is then an integral PEL datum for $\cO_F$ according to \cite[Definition 1.2.13]{lan-thesis} (where this notion is called a PEL-type $\cO_F$-lattice). Restricting $h$ to $\C^\times$ gives a homomorphism $h : \SS \to G_{\R}$ of real algebraic groups, where $\SS = \mathrm{Res}_{\C/\R}\GG_{m,\C}$ is the Deligne torus. The Shimura datum of $G$ is then $(G_\Q, X)$, where $X$ is the $G(\R)$-conjugacy class of $h$. 

\medskip

To define the integral PEL datum attached to $\tG$, we simply define $\wt{h} : \C \to \End_{\R}(L\otimes_\Z \R)$ by postcomposing $h$ with the natural map $\End_{\cO_F \otimes_\Z \R}(L\otimes_\Z \R) \to \End_{ \R}(L\otimes_\Z \R)$. $(L,\psi,\wt{h})$ is then an integral PEL datum for $\Z$, and the corresponding Shimura datum $(\tG_\Q,\wt{X})$ is a Siegel Shimura datum (here $\wt{X}$ is the $\tG(\R)$-conjugacy class of $\wt{h}$). 

\medskip

We end with a short discussion of the Hodge cocharacter. Recall that the Hodge cocharacter for $G$ is given by the composition
\[ \mu^\prime \colon \GG_{m,\CC} \xrightarrow{z \mapsto (z,1)} \GG_{m,\CC} \times \GG_{m,\CC} \simeq \SS_\CC \xrightarrow{h_\CC} G_\CC. \] 
Using the definition, one checks (by a calculation entirely analogous to the classical case of $\GL_{2,\Q}$) that $\mu^\prime$ is conjugate over $\C$ to the cocharacter $\mu : \GG_{m,\Z} \to G$ defined by 
\[
\mu(z) = \begin{pmatrix} zI_n & 0 \\ 0 & I_n \end{pmatrix},
\]
\medskip
where $I_n$ denotes the $n \times n$ identity matrix. In particular, the reflex field of $(G_\Q,X)$ is $\Q$. By composing $\mu$ with $G \to \tG$, we get a Hodge cocharacter $\wt{\mu}$ for $\tG$. 

\subsubsection{Shimura varieties and canonical models}
Given a neat open compact subgroup $K$ of $G(\AA_f)$, the Shimura variety for $G$ with level $K$ is the double quotient 
\[
G(\QQ) \backslash X \times G(\AA_f) / K,
\]
which has the structure of a complex manifold. We will only consider neat $K$. As is well known, the above double quotient may be viewed as the complex points of a moduli space parametrizing abelian varieties with certain extra structures. We briefly recall these moduli spaces over $\Q$ (for general $K$) and $\Z_{(p)}$ (when $K=K^p G(\Zp)$ with $K^p \sub G(\A_f^p)$), following Deligne and Kottwitz. We refer to \cite[\S 5]{kottwitz-shimura} and \cite[\S 1-2]{lan-thesis} for more detailed information and proofs.

\medskip

Over $\Q$, initially consider as test objects the category of connected, locally noetherian $\QQ$-schemes. For any such test object $S$, one defines a functor $S \mapsto X_K(S)$ by letting $X_K(S)$ be the set of equivalence classes of quadruples $(A,\lambda, \iota, \eta K)$, where
\begin{itemize}
 \item $A$ is an abelian scheme over $S$;
 \item $\lambda\colon A \rightarrow A^\vee$ is a polarization of $A$ (i.e.\ an isogeny inducing a polarization over each geometric fiber);
 \item $\iota\colon F \hookrightarrow \End^0(A) \defeq \End(A)\tensor_\ZZ \QQ$
       is a $\Q$-algebra homomorphism satisfying $\lambda \circ \iota(x^c)=\iota(x)^\vee \circ \lambda$ for all $x\in F$, where $-^\vee$ denotes the dual quasi-isogeny;
 \item For any geometric point $\bar s$ of $S$, $\eta K$ is a $\pi_1^{\et}(S,\bar s)$-invariant $K$-orbit of $F \tensor_\QQ \AA_f$-isomorphisms
       \[
	\eta\colon L \tensor_\Z \AA_f \stackrel{\sim}{\rightarrow} V_f(A_{\bar{s}})
       \]
       identifying the pairing $\psi \tensor_\Z \AA_f$ with the $\lambda_{\bar{s}}$-Weil pairing up to an $\AA_f^\times$-multiple, 
       where $V_f(A_{\bar{s}})$ is the rational adelic Tate module of the geometric fiber $A_{\bar{s}}$ at $\bar s$. 
 \end{itemize} 
 
The quadruples $(A, \lambda, \iota, \eta \widetilde{K})$ in 
this moduli description are further required to satisfy Kottwitz's determinant condition; we refer to \cite[\S 5]{kottwitz-shimura} for the precise formulation of this.
Two quadruples $(A, \lambda, \iota, \eta K)$ and $(A', \lambda', \iota', \eta' K)$ are equivalent if there is a 
quasi-isogeny between $A$ and $A'$, which is compatible with the 
polarizations $\lambda$ and $\lambda'$ up to a $\QQ^\times$-multiple, respects the $F$-actions induced by $\iota$ 
and $\iota'$, and takes the $K$-orbit $\eta K$ to $\eta' K$.

\medskip

This moduli problem can be extended to locally noetherian $\Q$-schemes $S$ which are not necessarily connected by taking a disjoint union over the sets corresponding to the connected components of $S$, and is then representable by a smooth quasi-projective scheme $X_K$ over $\QQ$, whose $\CC$-points $X_K(\CC)$ can be canonically identified with the double quotient $G(\QQ) \backslash X \times G(\AA_f) / K$.%
\footnote{In our case, since $\dim_F V =2n$ is even, the Hasse principle holds and we get a single copy of the
double quotient \cite[\S 7]{kottwitz-shimura}.} We recall that there is a right action of $\tG(\A_f)$ on the tower $(X_K)_K$
by precomposition $\eta\mapsto \eta\circ g$ on the level structures, which we will make much use of later. We also note in passing that our moduli problem uses a ``definition by isogeny classes'' \cite[\S 1.4.2]{lan-thesis} whereas, for example, \cite{scholze-galois} (which we will compare with later) uses a ``definition by isomorphism classes'' \cite[\S 1.4.1]{lan-thesis} when $K \sub G(\widehat{\Z})$. These two formulations are equivalent, as demonstrated in \cite[\S 1.4.3]{lan-thesis}.

\medskip 

When $K=K^p G(\Zp)$ with $K^p \sub G(\A_f^p)$, Kottwitz has constructed a smooth quasi-projective (integral, canonical) model of $X_K$ over $\Z_{(p)}$, which represents an extension of the above moduli problem (suitably modified) to locally Noetherian $\Z_{(p)}$-schemes. We refer to \cite[\S 5]{kottwitz-shimura} for the details. By abuse of notation, we will also denote this model by $X_K$. Finally, the whole discussion above holds with $G$ replaced by $\tG$; in this case the models were constructed by Mumford. We will use the notation $\wt{X}_{\tK}$ for the (models of the) Shimura varieties for $\tG$, for $\tK \sub \tG(\A_f)$ a neat open compact subgroup. Whenever $K \sub \tK$, we have a natural finite map $X_K \to \wt{X}_{\tK}$.

\subsubsection{Compactifications of canonical models}
In this paper we will need to consider two types of compactifications of our Shimura varieties. The first is the minimal (Satake--Baily--Borel) compactification. We will denote all minimal compactifications by a superscript $-^\ast$. For the varieties $X_K$ and $\wt{X}_{\tK}$ over $\Q$, the minimal compactifications $X_K^\ast$ and $\wt{X}_{\tK}^\ast$ were constructed by Pink \cite{pink-thesis}. Over $\Z_{(p)}$, the minimal compactifications were constructed by Faltings--Chai \cite{faltings-chai} and Lan \cite{lan-thesis}. They are normal, and projective over the base. The right action of $G(\A_f)$ and $\tG(\A_f)$ extends to the tower of minimal compactifications $(X^\ast_K)_K$ and $(\wt{X}^\ast_{\tK})_{\tK}$, respectively, and when $K \sub \tK$, the finite map $X_K \to \wt{X}_{\tK}$ extends to a finite map $X_K^\ast \to \wt{X}_{\tK}^\ast$.

\medskip

For the Shimura varieties over $G$, we will mainly use compactifications that are different from the minimal compactifications, following \cite[\S 4.1]{scholze-galois}. In the following discussion, fix $K$ and assume that $K = \tK \cap G(\A_f)$. For sufficiently small such $\tK$, the map $X_K \to \wt{X}_{\tK}$ is a closed immersion by \cite[Proposition 1.15]{deligne-tshimura}, but this need to hold for the maps $X_K^\ast \to \wt{X}_{\tK}^\ast$. Following Scholze, we define the ``ad hoc''-compactification $\ol{X}_K$ of $X_K$ to be the universal finite map $X_K^\ast \to \ol{X}_K$ over which all the 
$X_K^\ast \to \wt{X}_{\tK}^\ast$ vanish; as noted by Scholze $\ol{X}_K$ is the scheme-theoretic image of $X_K^\ast \to \wt{X}_{\tK}^\ast$ for sufficiently small $\tK$. The right action of $G(\A_f)$ extends to the tower of ad hoc compactifications $(\ol{X}_K)_K$.

\subsubsection{Parabolic and level subgroups}\label{levels} We finish this subsection by defining some subgroups which we will need throughout the paper.
Recall that we can represent elements of $G$ by $2\times 2$ block $n\times n$ matrices. We define closed subgroup schemes $P_\mu \sub G$ and $N \sub G$ on $R$-points ($R$ a ring) by  
\begin{align*}
  P_\mu(R) & \defeq \left\{(g,r)\in G(R) \mid g = \left(\begin{array}{cc} * & *\\ 0 & * \end{array}\right) \right\}; \\
  N(R) & \defeq \left\{(g,r)\in G(R) \mid g = \left(\begin{array}{cc} I_n  & * \\ 0 & I_n \end{array}\right) \right\},
\end{align*}
where $*$ means arbitrary $n\times n$ matrix and we recall that $I_n$ is the $n\times n$ identity matrix. Equivalently, we may define $P_\mu \sub G$ as the stabilizer of $\cO_F^n \oplus 0 \sub L$, and $N \sub G$ as the subgroup which acts trivially on $\cO_F^n \oplus 0 $ and on the quotient $L/(\cO_F^n \oplus 0)$. One may also describe $P_\mu(\C)$ as 
\[
P_\mu(\C) = \left\{ g\in G(\C) \mid \lim_{t\to 0} \ad(\mu(t))g \text{ exists}\right\}.
\]
In particular, $P_{\mu,\C}$ is the parabolic denoted by $P_\mu$ (for our choice of $\mu$) in \cite[\S 2.1]{caraiani-scholze}. When $\ell$ is a prime unramified in $F$, $G_{\Z_{(\ell)}}$ is reductive and $P_{\mu,\Z_{(\ell)}}$ is a parabolic subgroup of $G$, with unipotent radical $N_{\Z_{(\ell)}}$. We then define, for any integer $m\geq 0$, some open subgroups of $G(\Zp)$ by
\begin{align*}
  \Gamma_0(p^m) & \defeq \left\{(g,r)\in G(\Zp) \mid g \equiv \left(\begin{array}{cc} * & *\\ 0 & * \end{array}\right) \bmod p^m \right\}; \\
  \Gamma_1(p^m) &\defeq \left\{(g,r)\in G(\Zp) \mid g \equiv  \left(\begin{array}{cc} I_n  & * \\ 0 & I_n \end{array}\right) \bmod p^m\right\}; \\
  \Gamma(p^m) & \defeq \left\{(g,r) \in G(\Zp) \mid g \equiv \left(\begin{array}{cc} I_n & 0 \\ 0 & I_n \end{array}\right) \bmod p^m \right\}.
\end{align*}
Finally, we make similar definitions for $\tG$ and the analogous statements hold; we define $\tP_{\mu} \sub \tG$ as the stabilizer of $\cO_F^n \oplus 0 \sub L$, and $\tN \sub \tG$ as the subgroup which acts trivially on $\cO_F^n \oplus 0 $ and on the quotient $L/(\cO_F^n \oplus 0)$ (we have elected to use the notation $\wt{P}_\mu$ instead of $\wt{P}_{\wt{\mu}}$ for simplicity). We then put  
\begin{align*}
  \wt{\Gamma}_0(p^m) & \defeq \left\{(g,r)\in \tG(\Zp) \mid (g \bmod p^m) \in \wt{P}_{\mu}(\Z/p^m) \right\}; \\
  \wt{\Gamma}_1(p^m) & \defeq \left\{(g,r)\in \tG(\Zp) \mid (g \bmod p^m) \in \tN(\Z/p^m) \right\}; \\
  \wt{\Gamma}(p^m) & \defeq \left\{(g,r) \in \tG(\Zp) \mid g \equiv I_{2n} \bmod p^m \right\}.
\end{align*}

\subsection{The anticanonical tower}\label{subsect: anticanonical tower}

We now focus on the $p$-adic geometry of our Shimura varieties. Our first goal is to prove some analogues and partial refinements of results from \cite[\S 3.2]{scholze-galois} for the Shimura varieties of $G$. Throughout most of our arguments, the tame levels $K^p$ and $\tK^p$ will be fixed, so for simplicity we will drop them from the notation unless otherwise noted. We will make the assumption that these fixed $K^p$ and $\tK^p$ are ``small'' (in the terminology of \cite{newton-thorne}), meaning that they are contained in the kernel of the reduction modulo $N$ map on $G(\widehat{\Z}^p)$ and $\tG(\widehat{\Z}^p)$, respectively, for some integer $N\geq 3$ coprime to $p$. This is done to be able to apply the results of \cite[\S 3]{scholze-galois}; for all other results it suffices that the tame levels are neat.

\subsubsection{Formal and Adic Models}\label{subsubsect:formal_adic_models}
Recall that we have fixed a complete non-archimedean algebraically closed extension $C$ of $\Q_{p}$. Ultimately, we will work over $C$ in this paper, so to keep the notation uniform throughout we will stick to working over $C$ in the remainder of this section, but we remark that we could have chosen to work over any perfectoid extension of $\Qp^\cycl$, the completion of $\Qp(\zeta_{p^{\infty}})$.\footnote{We could also work over $\Qp$, as long as some minor adjustments are made to some of the statements.} We set $X^* \defeq X^*_{G(\Zp),\Zp}$ and we
let $\mathfrak{X}^*_{\Z_p}$ be the formal completion of $X^*$ 
along $p=0$. This is a formal scheme over $\Spf \Z_p$. We let $\mathfrak{X}^*$ denote
the base change of $\mathfrak{X}^*_{\Z_p}$ to $\Spf \cO_{C}$ and
$\cX^*$ denote the generic fiber of $\mathfrak{X}^*$, viewed as an adic space. This is a proper adic space over $\Spa (C,\cO_{C})$. We define $\wt{X}^*$, $\wt{\frakX}^*_{\Zp}$, $\wt{\frakX}^*$ and $\wt{\cX}^*$ analogously for $\tG$.

\medskip

For any level $K_p\subset G(\Zp)$, we let $\ol{\cX}_{K_p}$ and $\cX^\ast_{K_p}$ be the adic spaces corresponding to the base change to $C$ of the $\Q$-schemes $\ol{X}_{K_p}$ and $X^\ast_{K_p}$, respectively.\footnote{If $Y$ is a scheme locally of finite type over $\Spec\ C$, we take $Y^{\mathrm{ad}}\defeq Y\times_{C}\Spa\ (C, \cO_C)$, which is an adic space locally of finite type over $\Spa (C, \cO_C)$, where the fiber product is in the sense of~\cite[Prop.~3.8]{huber-adic-spaces}.} Note that $\cX^\ast = \cX^\ast_{G(\Zp)}$. We define $\cX_{K_p}$ similarly, as the adic space corresponding to the base change to $C$ of the $\Q$-scheme $X_{K_p}$. This is not quasicompact, but it is open in $\ol{\cX}_{K_p}$. The closed complement is the boundary $\cZ_{K_p}\defeq \ol{\cX}_{K_p} \setminus \cX_{K_p}$. Again, we define $\wt{\cX}^\ast_{\tK_p}$, $\wt{\cX}_{\tK_p}$ and $\wt{\cZ}_{\tK_p} = \wt{\cX}^\ast_{\tK_p} \setminus \wt{X}_{\tK_p}$ analogously for $\tG$. Although we will only need the following result for $\tG$, we state it in general. For the analogue for schemes see~\cite[Cor. 7.2.5.2]{lan-thesis}.

\begin{lemma}\label{adic quotient} For two compact open subgroups $K' \sub K$ such that $K'$ is normal in $\tK$, $\cX^\ast_{K}$ can be identified with the quotient of $\cX^\ast_{K'}$ by the finite group $K/K'$. The analogue holds for $\tG$.
\end{lemma}

\begin{proof}

We adapt 
the proof in~\cite[Cor. 7.2.5.2]{lan-thesis} to the corresponding adic spaces. Note that both $\cX^*_{K}$ and $\cX^*_{K'}$ are normal, since they are analytifications of normal algebraic varieties. Note also that on the open part, $\cX_{K'}$ is a $(K/K')$-torsor over $\cX_K$. (This holds true for the algebraic varieties, and the analytification functor preserves torsors for finite groups, as it preserves fiber products.) The result now follows from Theorem~\ref{bartenwerfer} and Corollary~\ref{abstract adic quotient} below. 
\end{proof}

\begin{thm}[{\cite[\S3]{bartenwerfer}}]\label{bartenwerfer} Let $Y$ be a normal rigid space over $C$, $Z\subset Y$ a
nowhere dense Zariski closed subset, with open complement $U$. 
Let $j\colon U\hookrightarrow Y$ be the open immersion.
Then $\cO^+_{Y} \toisom j_*\cO^+_{U}$ and $\cO_{Y} \toisom (j_*\cO^+_{U})[1/p]$. In particular, if $Y$ is affinoid and $f\in \cO_Y(U)$ is bounded, then $f$ extends uniquely to an element of 
$\cO_Y(Y)$, so $\cO_Y(Y)\toisom \cO^+_Y(U)[1/p]$. 
\end{thm}

\begin{cor}\label{abstract adic quotient} Let $f\colon \widetilde{Y}\to Y$ be a finite 
morphism of normal rigid analytic spaces. Assume that there exists a finite group $H$ acting on $\widetilde{Y}$,
acting trivially on $Y$, and such that $f$ is $H$-equivariant. Assume that there exists a 
Zariski open and dense subset $U\subseteq Y$ such that $\widetilde{U}\defeq f^{-1}(U)$
is Zariski open and dense in $\widetilde{Y}$ and such that $f$ identifies $U$ with
the quotient $\widetilde{U}/H$. Then $Y=\widetilde{Y}/H$.  
\end{cor}

\begin{proof} By finiteness and the construction of quotients by finite groups, 
we are reduced to the case when $Y$ is affinoid, and then we need to prove that
$\cO_Y(Y)=\cO_{\widetilde{Y}}(\widetilde{Y})^{H}$. 
Since $U=\widetilde{U}/H$, $\cO_Y(U)=\left(\cO_{\widetilde{Y}}(\widetilde{U})\right)^{H}$, 
and from this it follows that $\cO^{bd}_Y(U)=\left(\cO^{bd}_{\widetilde{Y}}(\widetilde{U})\right)^{H}$, 
where $\cO^{bd}_Y(U)\defeq \cO^+_Y(U)[1/p]$ denotes the sub-presheaf of bounded functions. 
By Theorem~\ref{bartenwerfer}, $\cO^{bd}_Y(U)=\cO_Y(Y)$ and 
$\cO^{bd}_{\widetilde{Y}}(\widetilde{U})=\cO_{\widetilde{Y}}(\widetilde{Y})$, and the result follows. 
\end{proof}

\begin{remark} It is also possible to deduce Lemma~\ref{adic quotient} from the case of schemes, using that the adification functor from schemes to adic spaces preserves quotients by finite groups (this can be checked with a bit of work from the definitions).
\end{remark} 

\subsubsection{Notation for ``infinite level'' Shimura varieties}\label{notation for infinite level}
Before proceeding, we will set out the notation for infinite level Shimura varieties. We mostly discuss the case of $G$; the case of $\tG$ is entirely analogous. Recall that we have fixed a prime-to-$p$-level. Our finite level Shimura varieties $\cX_K$ are then indexed by the open subgroups $K \sub G(\Zp)$. Our infinite level Shimura varieties will arise as limits of towers $(\cX_K)_K$, where the $K$ run through a cofiltered inverse system of open subgroups. Thus, the resulting limit (if it exists in a suitable sense) will only depend on the intersection $H=\bigcap K$, which is a closed subgroup of $G(\Zp)$. It therefore makes sense to make the following general definition:

\begin{defn}
Let $H \sub G(\Zp)$ be a closed subgroup. We define a locally spatial diamond $\cX_H$ by
\[ 
\cX_H \defeq \varprojlim_{H \sub K}\cX^{\lozenge}_K, 
\]
where the limit ranges over the open subgroups $K$ with $H \sub K \sub G(\Zp)$. Note that this exists by Proposition \ref{inverselimit}. We define $\ol{\cX}_H$ and $\cX^{\ast}_{H}$ similarly; these are spatial diamonds (again by Proposition \ref{inverselimit}). We also define $\wt{\cX}_{\tH}$ and $\wt{\cX}^\ast_{\tH}$ analogously for $\tG$, when $\tH \sub \tG(\Zp)$ is a closed subgroup.
\end{defn} 

We have a natural identification $|\cX_H|=\invlim_{H \sub K}|\cX_K|$ by Proposition \ref{inverselimit}, and similarly for the other infinite level Shimura varieties.

\begin{rem} \leavevmode
\begin{enumerate}
\item When $H = K$ is open, the above definition gives $\cX_K \defeq \cX_K^{\lozenge}$. This abuse of notation will be in place throughout this paper. It can be justified by the fact that the diamondification functor is fully faithful on \emph{normal} rigid spaces; see \cite[Proposition 10.2.4]{berkeley-notes} and \cite[Theorem 8.2.3]{kedlaya-liu-2}. Also, it should be clear from the context whether $\cX_K$ is regarded as a rigid space or a diamond.

\item If $(X_{i})_{i\in I}$ is an inverse system of rigid spaces with qcqs transition maps, and $X$ is a perfectoid space with $X \sim \varprojlim_{i}X_{i}$ in the sense of \cite[Definition 2.4.1]{scholze-weinstein}, then $X=\varprojlim_{i}X_{i}^{\lozenge}$ as diamonds by \cite[Proposition 2.4.5]{scholze-weinstein}.
\end{enumerate}
\end{rem}

With this definition, we may extend the towers $(\cX_K)_K$ and $(\cX^{\ast}_K)_K$, for open subgroups $K$, to towers $(\cX_H)_H$ and $(\cX^{\ast}_H)_H$ where we index over all closed subgroups. The right actions of $G(\Qp)$ on $(\cX_K)_K$, $(\ol{\cX}_K)_K$ and $(\cX^{\ast}_K)_K$ extend naturally to right actions on $(\cX_H)_H$, $(\ol{\cX}_K)_K$ and $(\cX^{\ast}_H)_H$. If $H \sub G(\Zp)$ is normal then $G(\Zp)/H$ acts on $\cX_H$, $\ol{\cX}_H$ and $\cX^\ast_H$. Note that, for arbitrary $H$, if $g\in G(\Qp)$ and $H,g^{-1}H g \sub G(\Zp)$, then right multiplication of $g$ induces an isomorphism $\cX_H \toisom \cX_{g^{-1}H g}$ of diamonds, and similarly for the compactifications.

\medskip

Next, we establish notation for some closed subgroups that will occur frequently. Recall that $P_\mu$ and $N$ were defined in \S \ref{levels}.

\begin{defn}\label{frequent infinite level subgroups}
We let $\mbf{1} \sub G(\Zp)$ denote the trivial subgroup, and define
\[
 P_{\mu,0} \defeq P_\mu(\Zp) = \bigcap_{m\geq 0}\Gamma_{0}(p^{m}),\,\,\,\, N_{0} \defeq N(\Zp) = \bigcap_{m\geq 0}\Gamma_{1}(p^{m}).
\]
We define $\wt{\mbf{1}}$, $\tP_{\mu,0}$ and $\tN_0$ analogously for $\tG$.
\end{defn}

Note that $\mbf{1}=\bigcap_{m\geq 0}\Gamma(p^{m})$. We finish by recording some results about group actions. We begin with a result for open Shimura varieties.

\begin{lemma}\label{torsor}
Let $H_1 \sub H_2$ be closed subgroups of $G(\Zp)$ and assume that $H_1$ is normal in $H_2$. Then $\cX_{H_{1}} \to \cX_{H_{2}}$ is a $H_{2}/H_{1}$-torsor. The analogue for $\tG$ holds.
\end{lemma}

\begin{proof}
Set $K_{i,m}=H_{i}\Gamma(p^m)$ for $i=1,2$ and $m\geq 0$. Then $K_{1,m}$ is normal in $K_{2,m}$ and hence $\cX_{K_{1,m}} \to \cX_{K_{2,m}}$ is a $K_{2,m}/K_{1,m}$-torsor compatibly in $m$, i.e.\ we have compatible isomorphisms
\[
 \cX_{K_{2,m}} \times \ul{K_{2,m}/K_{1,m}} \toisom \cX_{K_{1,m}} \times_{\cX_{K_{2,m}}} \cX_{K_{1,m}}
 \]
(this holds for rigid spaces and follows formally for diamonds, since fiber products are preserved). Taking the inverse limit over $m$ gives us the isomorphism $\cX_{H_{2}} \times \ul{H_{2}/H_{1}} \toisom \cX_{H_{1}} \times_{\cX_{H_{2}}} \cX_{H_{1}}$, so $\cX_{H_{1}} \to \cX_{H_{2}}$ is a $H_{2}/H_{1}$-torsor as desired. 
\end{proof}

We now move on to compactifications, where we will content ourselves with proving statements at the level of topological space. We will only need the first lemma for $\tG$, but we state it in general.

\begin{lemma}\label{exhibiting a diamond as a quotient}
Let $H$ be a closed subgroup of $G(\Zp)$. Then $|\cX^{\ast}_H|\cong |\cX^{\ast}_{\mbf{1}}|/H$ via the natural map $|\cX^{\ast}_{\mbf{1}}| \to |\cX^{\ast}_H|$ (and similarly for open Shimura varieties, and for $\tG$).
\end{lemma}

\begin{proof}
We prove the statement for minimal compactifications; the statement for open Shimura varieties follows from Lemma \ref{torsor}. For each $m\geq 0$, we have $|\cX^{\ast}_{H \Gamma(p^{m})}|=|\cX^{\ast}_{\Gamma(p^{m})}|/H$ via the natural map $|\cX^{\ast}_{\Gamma(p^{m})}| \to |\cX^{\ast}_{H \Gamma(p^{m})}|$ by Lemma \ref{adic quotient}.  Taking inverse limits, one obtains $\varprojlim_{m\geq 0}|\cX^\ast_{H\Gamma(p^m)}| \cong \left( \varprojlim_{m\geq 0} |\cX^\ast_{\Gamma(p^m)}| \right) /H $ (using that $H$ is compact Hausdorff). We are now done, using that $|\cX_{\mbf{1}}^{\ast}|=\invlim_{m\geq 0}|\cX_{\Gamma(p^{m})}^{\ast}|$ and $|\cX^{\ast}_H|=\invlim_{m\geq 0}|\cX_{H\Gamma(p^{m})}^{\ast}|$.
\end{proof}

Our task is now to prove the corresponding statement for ad hoc compactifications. Because of the way the ad hoc compactifications are defined, it is easiest not to work with a fixed $K^p$, so we drop this from now on until the beginning of \S \ref{sec: anticanonical}; we make definitions analogous to the above for closed subgroups $H \sub G(\A_f)$. We will break up the proof into a sequence of short results.

\begin{prop}\label{ad hoc closure}
Let $K \sub G(\A_f)$ be an open compact subgroup (always neat). $\cX_K$ is dense in $\ocX_K$ (in the analytic topology). Moreover, if $\tK \sub \tG(\A_f)$ is a sufficiently small compact open subgroup, in the sense that $\ol{X}_K$ is the scheme-theoretic image of $X_K^* \to \tX^*_\tK$, then $|\ocX_K|$ is the closure of $|\cX_K|$ inside $|\tcX^\ast_\tK|$.
\end{prop}

\begin{proof}
The first statement follows since $\cX_K$ is dense $\cX^\ast_K$ in the analytic topology and $\cX^\ast_K \to \ocX_K$ is surjective. The second statement then follows since $\ocX_K$ is closed (indeed Zariski closed) in $\tcX_\tK^\ast$.
\end{proof}

Next, we state a general lemma on diamonds of rigid spaces. If $Z$ is a topological space, there is a v-sheaf $\ul{Z}$ defined just before \cite[Definition 10.12]{diamonds}.

\begin{lemma}\label{ad hoc Zariski closed}
Let $S \to T$ be a Zariski closed immersion of rigid spaces. Then $S^\lozenge = \ul{|S|}\times_{\ul{|T|}}T^\lozenge$.
\end{lemma}

\begin{proof}
The map $S^\lozenge \to T^\lozenge$ is quasicompact and an injection (at the level of sheaves), since if $Z$ is a perfectoid space and $f,g : Z \to S$ are two maps which are equal after composing with $S \to T$, then they must already be equal. The statement then follows from \cite[Proposition 11.20]{diamonds}.
\end{proof}

We may then give a convenient description of the diamond of the ad hoc compactification.

\begin{cor}\label{ad hoc description of diamond}
Let $K \sub G(\A_f)$ be an open compact subgroup. Then $\ocX^\lozenge_K$ is the diamond attached to the closure of $|\cX_K|$ inside $|\tcX_K^\ast|$. 
\end{cor}

\begin{proof}
Let $\tK \sub \tG(\A_f)$ be a sufficiently small open compact subgroup throughout this proof, in the sense of Proposition \ref{ad hoc closure}. Proposition \ref{ad hoc closure} and Lemma \ref{ad hoc Zariski closed} show that, if $S_{\tK}$ is the closure of $|\cX_K|$ in $|\tcX_\tK^\ast|$, then $\ocX_K^\lozenge = S_{\tK} \times_{\ul{|\tcX_\tK^\ast|}}\tcX_\tK^{\ast,\lozenge}$. Taking inverse limits over $\tK$ we get
\[
\ocX_K^\lozenge = S_K \times_{\ul{|\tcX_K^\ast|}}\tcX_K^{\ast}
\]
where $S_K = \varprojlim_{\tK}S_\tK$, so it suffices to prove that $S_K$ is the closure of $|\cX_K|$ in $|\tcX_K^\ast|$. By the definition of the inverse limit topology, a point $x \in 
|\tcX_K^*|$ is in the closure of $|\cX_K|$ if and only if for every 
$\tK$, every open neighborhood of the image of $x$ in 
$|\tcX_{\tK}^*|$ intersects $|\cX_K|$, and this is equivalent to the image of $x$ in each $|\tcX_{\tK}^*|$ being in the closure of $|\cX_K|$. This finishes the proof.
\end{proof}

We now get to the analogue of Lemma \ref{exhibiting a diamond as a quotient}.

\begin{prop}\label{ad hoc exhibiting a diamond as a quotient}
Let $H'\sub H \sub G(\A_f)$ be compact subgroups, with $H'$ normal in $H$. Then $|\ocX_H| = |\ocX_{H'}|/(H/H')$.
\end{prop}

\begin{proof}
It suffices to prove the case when $H'$ is open, the general case then follows as in the proof of Lemma \ref{exhibiting a diamond as a quotient} by taking inverse limits. The map $|\tcX^\ast_{H'}| \to |\tcX^\ast_H|$ is open, so taking closures commutes with preimages. In particular, it follows that $|\ocX_{H'}|$ is the preimage of $|\ocX_H|$. Since $|\tcX^\ast_H| = |\tcX^\ast_{H'}|/(H/H')$ by Lemma \ref{exhibiting a diamond as a quotient}, the proposition follows.
\end{proof}

\subsubsection{The anticanonical tower}\label{sec: anticanonical} In this section, we adapt the construction of the anticanonical tower over a neighborhood of the ordinary locus for the Shimura varieties of $\tG$ as in~\cite[Sec.~3]{scholze-galois} to the Shimura varieties of $G$. Our strategy is to deduce results for $G$ from those for $\tG$, but we will also need some refinements of the results of~\cite[Sec.~3]{scholze-galois} for $\tG$.

\medskip

The special fiber of the integral model $X^\ast$ admits a Newton stratification -- see, for example,~\cite[\S 3.3]{lan-stroh}. The reflex field of $G$ is $\Q$, so~\cite[Thm.\ 1.6.3]{wedhorn-thesis} 
(together with~\cite[\S 3.3]{lan-stroh} for the extension to the boundary) implies that there exists an open dense ordinary stratum in $X^*_{\overline{\mathbb{F}}_p}$. As usual, one can also recover the ordinary stratum as the complement of the vanishing locus of 
the Hasse invariant $\Ha$. Over $X^\ast$, we have the Hodge line bundle $\omega$, and the Hasse invariant $\Ha$ is a section of $\omega_{\overline{\mathbb{F}}_p}^{\otimes (p-1)}$. The above discussion also holds for $\tG$, and we note that $\Ha$ on $X^{\ast}_{\ol{\F}_{p}}$ is the pullback of the Hasse invariant on $\wt{X}^{\ast}_{\ol{\F}_{p}}$.

\medskip

Let us now recall the results from \cite[\S 3]{scholze-galois} for $\tG$ that we need (partially to set up notation), and prove an additional result that will be crucial for us in this paper (Proposition~\ref{anticanonical tower at intermediate levels Siegel}). 
We start by recalling the anticanonical tower for $\tG$. In this discussion, we will work over $\Qp$ instead of $C$ until further notice to simplify referencing to \cite{scholze-galois}, so all Shimura varieties for $\tG$ (of finite or infinite level) are considered to be defined over $\Qp$; we will sometimes add a subscript $\Qp$ when we wish to emphasize this. Let $0\leq \epsilon <1/2$. The anticanonical locus $\wt{\cX}^{\ast}_{\tGam_{0}(p)}(\epsilon)_{a}$ of level $\tGam_{0}(p)$ and radius of overconvergence $\epsilon$ is (essentially) defined in \cite[Theorem 3.2.15(iii)]{scholze-galois}. This is an open subset of $\wt{\cX}^{\ast}_{\tGam_{0}(p)}$, defined as the image of the map
\[
 \wt{\cX}^{\ast}(\epsilon) \to \wt{\cX}^{\ast}_{\tGam_{0}(p)},
\]
where $\wt{\cX}^{\ast}(\epsilon)$ is the locus where $|Ha|\geq p^{\epsilon}$ (here and elsewhere, this condition is defined in terms of local (integral) lifts as usual, and independent of the choices of lifts) and the map sends a principally polarized abelian variety $A$ (with tame level structure) to $(A/\Can,A[p]/\Can)$ (with the induced tame level structure), where $\Can \sub A[p]$ is the canonical subgroup. For any closed subgroup $\tH \sub \tGam_{0}(p)$, we define
\[
  \wt{\cX}^{\ast}_{\tH}(\epsilon)_{a} \defeq \wt{\cX}^{\ast}_{\tGam_{0}(p)}(\epsilon)_{a} \times_{\wt{\cX}^{\ast}_{\tGam_{0}(p)}} \wt{\cX}^{\ast}_{\tH}.
\]
When $\tH$ is open, we primarily view this as a rigid space.
By \cite[Corollary 3.2.19]{scholze-galois} (see remark below), there is a unique affinoid perfectoid space $\wt{\cX}_{\tP_{\mu,0}\cap \tG^{\mathrm{der}}(\Zp), \Qp}^\ast(\epsilon)_a$ such that
\[
\wt{\cX}^{\ast}_{\tP_{\mu,0} \cap \tG^\mathrm{der}(\Zp), \Qp}(\epsilon)_a\sim \varprojlim_m \mathcal{X}^\ast_{\tGam_{0}(p^m)^{\prime}, \Qp}(\epsilon)_{a},
\]
where $\tG^\mathrm{der}\sub \tG$ is the subgroup where the similitude factor is $1$, and $\tGam_{0}(p^{m})^{\prime}$ is the subgroup of $\tGam_{0}(p^{m})$ given by imposing the condition that the similitude factor is congruent to $1$ modulo $p^m$. Moreover, the boundary 
\[
\wt{\cZ}_{\tP_{\mu,0}\cap \tG^\mathrm{der}(\Zp), \Qp}(\epsilon)_{a} \sub \wt{\cX}^{\ast}_{\tP_{\mu,0}\cap \tG^\mathrm{der}(\Zp),\Qp}(\epsilon)_{a}
\]
is strongly Zariski closed, in the sense of~\cite[Definition 2.2.6]{scholze-galois}\footnote{In fact, the notions of Zariski closed and strongly Zariski closed, as defined in \cite[\S 2]{scholze-galois}, agree by \cite[Remark 7.5]{prismatic}, which has appeared since the first version of our paper was published. We have elected to keep the distinction in this paper, since the strong Zariski closure statements we need follow easily from those in \cite{scholze-galois}.}. 

\begin{rem}
The $\Gamma_{0}(p^{m})$-level structures of \cite[Definition 3.1.1]{scholze-galois} are slightly different to level structures $\tGam_{0}(p^{m})^{\prime}$ in this paper; the level structure $\Gamma_{0}(p^{\infty})$ corresponds to the kernel of the determinant map on $\tP_{\mu,0}$ in our notation. We assume that this is a typo and that the intention was to instead let it correspond to the kernel of the similitude map (this is what matches with the arguments of \cite[\S 3]{scholze-galois}). We also remark that the spaces constructed in \cite[\S 3]{scholze-galois} morally live over $\Qp$, in the sense that they are limits of rigid spaces defined over $\Qp$, but the limit admits a morphism to $\Qp^\cycl$ given by the similitude factor, which allows them to be regarded as spaces over $\Qp^\cycl$. This is the reason for working over $\Qp$ in the current discussion.
\end{rem}

Going further up, by \cite[Proposition 3.2.34, Theorem 3.2.36]{scholze-galois} there are unique affinoid perfectoid spaces $\wt{\cX}^{\ast}_{\tN_{0},\Qp}(\epsilon)_a$ and $\wt{\cX}^{\ast}_{\wt{\mbf{1}},\Qp}(\epsilon)_a$ such that
\[ 
\wt{\cX}^{\ast}_{\tN_{0},\Qp}(\epsilon)_a \sim \varprojlim_{m} \wt{\cX}^{\ast}_{\tGam_1(p^m),\Qp}(\epsilon)_{a}; 
\]
\[ \wt{\cX}^{\ast}_{\wt{\mbf{1}},\Qp}(\epsilon)_{a} \sim \varprojlim_{m} \wt{\cX}^{\ast}_{\tGam(p^m),\Qp}(\epsilon)_{a} . 
\]
Moreover, the boundaries $\wt{\cZ}_{\tN_{0},\Qp}(\epsilon)_{a} \sub \wt{\cX}^\ast_{\tN_{0},\Qp}(\epsilon)_{a}$ and $\wt{\cZ}_{\wt{\mbf{1}},\Qp}(\epsilon)_{a} \sub \wt{\cX}^{\ast}_{\wt{\mbf{1}},\Qp}(\epsilon)_{a}$, respectively, are strongly Zariski closed, and functions on $\wt{\cX}_{\wt{\mbf{1}},\Qp}(\epsilon)_{a}$ extend uniquely to $\wt{\cX}^{\ast}_{\wt{\mbf{1}},\Qp}(\epsilon)_{a}$. It will be essential for the arguments of this paper to have a slightly more general result. Before we state and prove this generalization, we recall \cite[Lemma 3.2.24(iii)]{scholze-galois}\footnote{In part because the notation we use is incompatible with the notation used in \cite[Lemma 3.2.24(iii)]{scholze-galois}} which will be used in the proof. For all levels $\tH$, we use $\wt{\cX}^{gd}_{\tH}(\epsilon)_{a}$ and $\wt{\cX}^{gd}_{\tH}$ to denote the good reduction loci in $\wt{\cX}_{\tH}(\epsilon)_{a}$ and $\wt{\cX}_{\tH}$ respectively; these are quasicompact open subspaces (or diamonds).

\begin{lemma}\label{3.2.24} \cite[Lemma 3.2.24(iii)]{scholze-galois}
Let $\cY_{m}^{\ast} \to \wt{\cX}^{\ast}_{\tGam_{0}(p^m)^\prime,\Qp}(\epsilon)_{a}$ be finite, \'etale away from the boundary, and assume that $\cY^{\ast}_{m}$ is normal and that no irreducible component of $\cY_{m}^{\ast}$ maps to the boundary of $\wt{\cX}^{\ast}_{\tGam_{0}(p^{m})^\prime,\Qp}(\epsilon)_{a}$. In particular, $\cY^{gd}_{m}\defeq \wt{\cX}^{gd}_{\tGam_{0}(p^{m})^{\prime},\Qp}(\epsilon)_{a} \times_{\wt{\cX}^{\ast}_{\tGam_{0}(p^{m})^{\prime},\Qp}(\epsilon)_{a}} \cY_{m}^{\ast}$ is finite \'etale over $\cX^{gd}_{\tGam_{0}(p^{m})^{\prime},\Qp}(\epsilon)_{a}$. For any $m^{\prime} \geq m$, let $\cY_{m^{\prime}}^{\ast}$ be the normalization of $\wt{\cX}^{\ast}_{\tGam_{0}(p^{m^{\prime}})^{\prime},\Qp}(\epsilon)_{a}\times_{\wt{\cX}^{\ast}_{\tGam_{0}(p^{m})^{\prime},\Qp}(\epsilon)_{a}} \cY_{m}^{\ast}$ and let $\cY_{m^{\prime}}^{gd}$ be the preimage of $\cY_{m}^{gd}$. Let $\cY_{\infty}^{gd}=\cY_{m}^{gd} \times_{\wt{\cX}^{gd}_{\tGam_{0}(p^{m})^{\prime},\Qp}(\epsilon)_{a}} \wt{\cX}^{gd}_{\tP_{\mu,0}\cap \tG^\mathrm{der}(\Zp), \Qp}(\epsilon)_{a}$; this exists as a perfectoid space since $\cY_{m}^{gd} \to \wt{\cX}^{gd}_{\tGam_{0}(p^{m})^{\prime},\Qp}(\epsilon)_{a}$ is finite \'etale. 
The spaces $\cY_{m^{\prime}}^{\ast}$ are affinoid for $m^{\prime}$ sufficiently large, so write $\cY_{m^{\prime}}^{\ast}=\Spa(S_{m^{\prime}},S_{m^{\prime}}^{+})$. Assume further that $S_{\infty} \defeq H^{0}(\cY^{gd}_{\infty},\cO_{\cY_{\infty}^{gd}})$ is a perfectoid $\Qp^{\cycl}$-algebra and define $\cY_{\infty}^{\ast} \defeq \Spa(S_{\infty},S_{\infty}^{\circ})$. Then $\cY_{\infty}^{\ast} \sim \invlim_{m^{\prime}}\cY_{m^{\prime}}^{\ast}$.
\end{lemma}

To simplify notation, we will not write out the subscripts `$\Qp$' for most of the discussion below until after the proof of Corollary \ref{anticanonical tower at intermediate levels 2 Siegel Qp}, and we will put $\tP_{\mu,0}^{\prime} \defeq \tP_{\mu,0} \cap \tG^\mathrm{der}(\Zp)$. Returning to the situation of the lemma above and keeping in mind these conventions, we recall that $\wt{\cX}^{\ast}_{\tGam_{1}(p^{m}) \cap \tP_{\mu,0}^{\prime}}(\epsilon)_{a}$ is proved to be perfectoid via Lemma \ref{3.2.24}. In particular,
\[
H^{0}(\wt{\cX}^{\ast}_{\tGam_{1}(p^{m})\cap \tP_{\mu,0}^{\prime}}(\epsilon)_{a},\cO_{\wt{\cX}^{\ast}_{\tGam_{1}(p^{m})\cap \tP_{\mu,0}^{\prime}}(\epsilon)_{a}}) = H^{0}(\wt{\cX}^{gd}_{\tGam_{1}(p^{m}) \cap \tP_{\mu,0}^{\prime}}(\epsilon)_{a},\cO_{\wt{\cX}^{gd}_{\tGam_{1}(p^{m}) \cap \tP_{\mu,0}^{\prime}}(\epsilon)_{a}})
\]
is a perfectoid $\Qp^{\cycl}$-algebra. As $\wt{\cX}^{\ast}_{\tGam(p^{m})}(\epsilon)_{a} \to \wt{\cX}^{\ast}_{\tGam_{1}(p^{m})}(\epsilon)_{a}$ is finite \'etale by \cite[Lemma 3.2.35]{scholze-galois}), one sees that 
\[
\wt{\cX}^{\ast}_{\tGam(p^{m}) \cap \tP_{\mu,0}^{\prime}}(\epsilon)_{a} = \wt{\cX}^{\ast}_{\tGam(p^{m})}(\epsilon)_{a}\times_{\wt{\cX}^{\ast}_{\tGam_{1}(p^{m})}(\epsilon)_{a}} \wt{\cX}^{\ast}_{\tGam_{1}(p^{m}) \cap \tP_{\mu,0}^{\prime}}(\epsilon)_{a}
\]
is affinoid perfectoid and a direct calculation using that $\wt{\cX}^{\ast}_{\tGam(p^{m})}(\epsilon)_{a} \to \wt{\cX}^{\ast}_{\tGam_{1}(p^{m})}(\epsilon)_{a}$ is finite \'etale shows that
\[
 H^{0}(\wt{\cX}^{\ast}_{\tGam(p^{m}) \cap \tP_{\mu,0}^{\prime}}(\epsilon)_{a},\cO_{\wt{\cX}^{\ast}_{\tGam(p^{m}) \cap \tP_{\mu,0}^{\prime}}(\epsilon)_{a}}) = H^{0}(\wt{\cX}^{gd}_{\tGam(p^{m}) \cap \tP_{\mu,0}^{\prime}}(\epsilon)_{a},\cO_{\wt{\cX}^{gd}_{\tGam(p^{m}) \cap \tP_{\mu,0}^{\prime}}(\epsilon)_{a}}),
\]
so the right hand side is perfectoid.

\begin{prop}\label{anticanonical tower at intermediate levels Siegel} We work over $\Qp$. Let $\tK \subset \tG(\Z_p)$ be a compact open subgroup. There exists an affinoid perfectoid space 
\[
\wt{\cX}^{\ast}_{\tK \cap \tP_{\mu,0}^{\prime}}(\epsilon)_a \sim \varprojlim_{m} \wt{\cX}^{\ast}_{\tK \cap \tGam_{0}(p^{m})^{\prime}}(\epsilon)_a .
\]
Moreover, the boundary $\wt{\cZ}_{\tK \cap \tP_{\mu,0}^{\prime}}(\epsilon)_a \subset \wt{\cX}^{\ast}_{\tK \cap \tP_{\mu,0}^{\prime}}(\epsilon)_a$ is strongly Zariski closed. 
\end{prop}

\begin{proof} We check that the conditions of Lemma \ref{3.2.24} apply to $\cY_m^\ast\defeq \wt{\cX}^{\ast}_{\tK \cap \tGam_0(p^m)^{\prime}}(\epsilon)_a$ for any large enough $m$. The projection map 
\[
\wt{\cX}^{\ast}_{\tK \cap \tGam_0(p^m)^{\prime}} \to \wt{\cX}^{\ast}_{\tGam_0(p^m)^{\prime}}
\]
is finite \'etale away from the boundary, $\wt{\cX}^{\ast}_{\tK \cap \tGam_0(p^m)^{\prime}}(\epsilon)_a$ is normal, 
and no irreducible component of $\wt{\cX}^{\ast}_{\tK \cap \tGam_0(p^m)^{\prime}}(\epsilon)_a$ maps into the boundary. 
Note that the space $\cY_{\infty}^{gd}\defeq \wt{\cX}^{gd}_{\tK \cap \tP_{\mu,0}^{\prime}}(\epsilon)_a$ exists and is perfectoid, since it is the pullback to the perfectoid space $\wt{\cX}^{gd}_{\tP_{\mu,0}^{\prime}}(\epsilon)_a$ of $\wt{\cX}_{\tK \cap \tGam_0(p^m)^{\prime}}(\epsilon)_a \to  \wt{\cX}_{\tGam_0(p^m)^{\prime}}(\epsilon)_a$,
which is finite \'etale. The last thing we need to show is that 
\[
S_\infty'\defeq H^0(\wt{\cX}^{gd}_{\tK \cap \tP_{\mu,0}^{\prime}}(\epsilon)_a, \cO_{\wt{\cX}^{gd}_{\tK \cap \tP_{\mu,0}^{\prime}}(\epsilon)_a})
\]
is a perfectoid $\Q_p^\cycl$-algebra. 

\medskip
To see this, we argue as follows. There exists an integer $t\geq 1$ such that $\tGam(p^t) \subset \tK$; then
the quotient $H\defeq \tK /\tGam(p^t)$ is a finite group. From the discussion preceding this Proposition, we know that $\wt{\cX}^{\ast}_{\tGam(p^t) \cap \tP_{\mu,0}^{\prime}}(\epsilon)_a$
is affinoid perfectoid and that
\[
 S_{\infty} \defeq H^{0}(\wt{\cX}^{\ast}_{\tGam(p^t) \cap \tP_{\mu,0}^{\prime}}(\epsilon)_{a},\cO_{\wt{\cX}^{\ast}_{\tGam(p^{t}) \cap \tP_{\mu,0}^{\prime}}(\epsilon)_{a}}) = H^{0}(\wt{\cX}^{gd}_{\tGam(p^t) \cap \tP_{\mu,0}^{\prime}}(\epsilon)_{a},\cO_{\wt{\cX}^{gd}_{\tGam(p^t) \cap \tP_{\mu,0}^{\prime}}(\epsilon)_{a}}).
\]
We now conclude using~\cite[Prop. 3.6.22]{kedlaya-liu}, since 
\[
S_\infty' = (S_\infty)^H,
\]
which follows from the fact that $\wt{\cX}^{gd}_{\tGam(p^t) \cap \tP_{\mu,0}^{\prime}}(\epsilon)_a$ is an $H$-torsor over $\wt{\cX}^{gd}_{\tK \cap \tP_{\mu,0}^{\prime}}(\epsilon)_a$. This concludes the proof of the first part of the proposition, using Lemma \ref{3.2.24}.

\medskip

To see that $\wt{\cZ}_{\tK \cap \tP_{\mu,0}^{\prime}}(\epsilon)_a\subset \wt{\cX}^{\ast}_{\tK \cap \tP_{\mu,0}^{\prime}}(\epsilon)_a$ is strongly Zariski closed, we note that $\wt{\cZ}_{\tK \cap \tP_{\mu,0}^{\prime}}(\epsilon)_a$ is the pullback of $\wt{\cZ}_{\tP_{\mu,0}^{\prime}}(\epsilon)_a \subset \wt{\cX}^{\ast}_{\tP_{\mu,0}^{\prime}}(\epsilon)_a$ over the map $\wt{\cX}^{\ast}_{\tK \cap \tP_{\mu,0}^{\prime}}(\epsilon)_a \to \wt{\cX}^{\ast}_{\tP_{\mu,0}^{\prime}}(\epsilon)_a$, and use \cite[Lemma 2.2.9]{scholze-galois}.
\end{proof}

We immediately deduce the following generalization.

\begin{cor}\label{anticanonical tower at intermediate levels 2 Siegel Qp} We work over $\Qp$. Let $\tH \sub \tP_{\mu,0}^{\prime}$ be a closed subgroup. Then the diamond $\wt{\cX}^{\ast}_{\tH}(\epsilon)_a$ is an affinoid perfectoid space, and the boundary $\wt{\cZ}_{\tH}(\epsilon)_a \subset \wt{\cX}^{\ast}_{\tH}(\epsilon)_a$ is strongly Zariski closed. 
\end{cor}

\begin{proof}
Let $\tH_{i} \sub \tP_{\mu,0}^{\prime}$, $i\in I$, be a collection of open subgroups of $\tP^{\prime}_{\mu,0}$ containing $\tH$ such that $\bigcap_{i\in I}\tH_{i} = \tH$. We may assume that each $\tH_{i}$ is equal to $\tP_{\mu,0}^{\prime} \cap \tK_{i}$ for some open subgroup $\tK_{i} \sub \tG(\Zp)$. Then each $\wt{\cX}^\ast_{\tH_{i}}(\epsilon)_{a}$ is affinoid perfectoid by Proposition \ref{anticanonical tower at intermediate levels Siegel}, so $\wt{\cX}^{\ast}_{\tH}(\epsilon)_a = \invlim_{i} \wt{\cX}^{\ast}_{\tH_{i}}(\epsilon)_a$ is affinoid perfectoid. The boundary $\wt{\cZ}_{\tH}(\epsilon)_a$ is the pullback of $\wt{\cZ}_{\tP_{\mu,0}^{\prime}}(\epsilon)_a$, hence strongly Zariski closed by \cite[Lemma 2.2.9]{scholze-galois}.
\end{proof}

We now go back to working over $C$, and deduce the following result.

\begin{cor}\label{anticanonical tower at intermediate levels 2 Siegel}
We work over $C$. Let $\tH \sub \tP_{\mu,0}$ be a closed subgroup. Then the diamond $\wt{\cX}^{\ast}_{\tH}(\epsilon)_a$
is an affinoid perfectoid space, and the boundary $ \wt{\cZ}_{\tH}(\epsilon)_a \subset \wt{\cX}^{\ast}_{\tH}(\epsilon)_a $ is strongly Zariski closed.
\end{cor}

\begin{proof}
First assume that $\tH \sub \tP_{\mu,0}^{\prime}$. Then 
\[
\wt{\cX}^{\ast}_{\tH}(\epsilon)_a = \left( \wt{\cX}^{\ast}_{\tH,\Qp}(\epsilon)_a \times_{\Spa(\Qp)^{\lozenge}} \Spa(\Qp^\cycl)^{\lozenge} \right) \times_{\Spa(\Qp^\cycl)^{\lozenge}} \Spa(C)^{\lozenge}
\]
where the omitted ring of integral elements is the power-bounded elements. The first fiber product is affinoid perfectoid by Corollary \ref{anticanonical tower at intermediate levels 2 Siegel Qp} since it is equal to
\[
\invlim_{m} \left( \wt{\cX}^{\ast}_{\tH,\Qp}(\epsilon)_a \times_{\Spa(\Qp)^{\lozenge}} \Spa(\Qp(\zeta_{p^m}))^{\lozenge} \right);
\]
the second fiber product is then a fiber product of affinoid perfectoid spaces, hence affinoid perfectoid. The statement about the boundary follows by pullback, using the map $\wt{\cX}^{\ast}_{\tH}(\epsilon)_a \to \wt{\cX}^{\ast}_{\tH,\Qp}(\epsilon)_a$.

\medskip

For general $\tH \sub \tP_{\mu,0}$, set $\tH^{\prime} = \tH \cap G^\mathrm{der}(\Zp)$. Then $\wt{\cX}^{\ast}_{\tH}(\epsilon)_a$ is a closed union of connected components of $\wt{\cX}^{\ast}_{\tH^{\prime}}(\epsilon)_a$, and the result for $\tH$ follows from that for $\tH^{\prime}$.
\end{proof}

We now return to the Shimura varieties for $G$, and continue to work over $C$ for the rest of this paper. For any closed subgroup $H \sub \Gamma_{0}(p)$, set
\[
 \ocX_{H}(\epsilon)_{a} \defeq \ocX_{H} \times_{\wt{\cX}^{\ast}_{\tGam_{0}(p)}} \wt{\cX}^{\ast}_{\tGam_{0}(p)}(\epsilon)_{a};
 \]
and $\cX_{H}(\epsilon)_{a}= \ocX_{H}(\epsilon)_{a} \cap \cX_{H}$. These are non-empty since the ordinary locus of the Shimura varieties of $G$ is non-empty. The following is the main result of this subsection.

\begin{thm}\label{anticanonical tower at intermediate levels 2} Let $H \sub P_{\mu,0}$ be a closed subgroup. Then $\ocX_{H}(\epsilon)_{a}$ is an affinoid perfectoid space over $C$, and the boundary $\cZ_{H}(\epsilon)_a \subset \ocX_{H}(\epsilon)_a$ is strongly Zariski closed. 
\end{thm}

\begin{proof}
In this proof we will write out tame levels for the Shimura varieties for $\tG$ that appear. Choose a (countable) shrinking set of compact open subgroups $K_{i} \sub \Gamma_{0}(p)$ such that $H= \bigcap_{i} K_{i}$, and choose shrinking sets of tame levels $\tK_{i}^{p}$ and compact open subgroups $\tK_{i} \sub \tGam_{0}(p)$ such that $\ol{\cX}_{K_{i}} \sub \wt{\cX}^{\ast}_{\tK_{i}^{p}\tK_{i}}$. Set $\tH^{p}=\bigcap_{i}\tK_{i}^{p}$, $\tH=\bigcap_{i}\tK_{i}$ and define
\[
 \wt{\cX}^{\ast}_{\tH^{p}\tH}(\epsilon)_{a} \defeq \invlim_{i} \wt{\cX}^{\ast}_{\tK_{i}^{p}\tK_{i}}(\epsilon)_{a}^{\lozenge} = \invlim_{i} \wt{\cX}^{\ast}_{\tK_{i}^{p}\tH}(\epsilon)_{a}.
 \]
The second equality and Corollary \ref{anticanonical tower at intermediate levels 2 Siegel} shows that this is an affinoid perfectoid space with strongly Zariski closed boundary. We have 
\[
\ocX_{H}(\epsilon)_{a} = \invlim_{i} \ol{\cX}_{K_{i}}(\epsilon)_{a}^{\lozenge}.
\]
For every $i$, set 
\[
Y_{i} \defeq \ol{\cX}_{K_{i}}(\epsilon)_{a} \times_{\wt{\cX}^{\ast}_{\tK_{i}^{p}\tK_{i}}(\epsilon)_{a}} \wt{\cX}^{\ast}_{\tH^{p}\tH}(\epsilon)_{a}.
\]
Since $\ol{\cX}_{K_{i}}(\epsilon)_{a} \sub \wt{\cX}^{\ast}_{\tK_{i}^{p}\tK_{i}}$ is Zariski closed, $Y_{i} \sub  \wt{\cX}^{\ast}_{\tH^{p}\tH}(\epsilon)_{a} $ is Zariski closed and hence affinoid perfectoid with strongly Zariski closed boundary. It follows that $\ocX_{H}(\epsilon)_{a} = \bigcap_{i} Y_{i}$ is affinoid perfectoid with strongly Zariski closed boundary, as desired.
\end{proof}

\subsection{The Hodge--Tate period morphism}\label{hodgetatemapsection}
In this subsection we discuss the perfectoid Shimura variety $\ocX_{\mbf{1}}$ 
and the Hodge--Tate period morphism 
\[
\pi_{\mathrm{HT}}\colon \ocX_{\mbf{1}} \to \mathscr{F}\ell_{G,\mu}.
\]
Recall that, roughly speaking, the Hodge--Tate period morphism measures the relative position of the Hodge--Tate filtration on the
universal abelian variety over the perfectoid Shimura variety. 

\medskip

Recall that we defined subgroup schemes $P_\mu \sub G$ and $\tP_\mu \sub \tG$ in \S \ref{levels}, which are parabolic subgroups at good primes (in particular at $p$). At the level of schemes,  we define flag varieties over $\Zp$ as the quotients $\Flr_{G,\mu}\defeq G_{\Zp}/P_{\mu,\Zp}$ and $\Flr_{\tG,\mu}\defeq \tG_{\Zp}/\tP_{\mu,\Zp}$ (we could have made these definitions over $\Z_{(p)}$, but we will have no need for that extra generality). Note that they carry natural left actions of $G_{\Zp}$ and $\tG_{\Zp}$, respectively, which we will also consider as right actions by inversion. We have a natural closed immersion $\Flr_{G,\mu} \to \Flr_{\tG,\mu}$, which is equivariant for the action of $G_{\Zp}$. We will mostly be interested in the rigid spaces
\[
\Fl_{G,\mu} \defeq (\Flr_{G,\mu} \times_{\Spec \Zp}\Spec C)^\mathrm{ad}, \,\,\,\,\,\, \Fl_{\tG,\mu} \defeq (\Flr_{\tG,\mu} \times_{\Spec \Zp}\Spec C)^\mathrm{ad}.
\]
Even though these are rigid spaces over $C$, we will implicitly remember that they naturally arise by base change from $(\Flr_{G,\mu} \times_{\Spec \Zp}\Spec \Qp)^\mathrm{ad}$ and $(\Flr_{\tG,\mu} \times_{\Spec \Zp}\Spec \Qp)^\mathrm{ad}$ respectively; in particular this will allow us to define the $\Qp$-points of $\Fl_{G,\mu}$ and $\Fl_{\tG,\mu}$. Recall that we have right actions of  $G(\Q_p)$ and $\tG(\Qp)$ on $\ocX_{\mbf{1}}$ and $\wt{\cX}^*_{\wt{\mbf{1}}}$ respectively, and that the natural map $\ocX_{\mbf{1}} \to \wt{\cX}^*_{\wt{\mbf{1}}}$ is $G(\Qp)$-equivariant. 

\begin{thm}\label{perfectoid Shimura variety and properties 2}
The spatial diamond $\ocX_{\mbf{1}}$ is a perfectoid space, and there is a $G(\Qp)$-equivariant Hodge--Tate period map $\ocX_{\mbf{1}} \to \Fl_{G,\mu}$ of adic spaces (for the right action of $G(\Qp)$ on $\Fl_{G,\mu}$).

\end{thm} 

\begin{proof}
The first part is \cite[Theorem 4.1.1]{scholze-galois}. For the second part, the construction proving \cite[Theorem 2.1.3]{caraiani-scholze} gives a commutative diagram
  \[
    \xymatrix{\cX_{\mbf{1}} \ar[r]^-{\pi_{\HT,G}} \ar[d] & \Fl_{G,\mu}\ar[d] \\ \wt{\cX}^{\ast}_{\wt{\mbf{1}}} \ar[r]^-{\pi_{\HT,\tG}} & \Fl_{\tG,\mu}}.
    \]
Since $\cX_{\mbf{1}}$ is dense in $\ocX_{\mbf{1}}$ and $\Fl_{G,\mu} \sub \Fl_{\tG,\mu}$ is Zariski closed, it follows that $\pi_{\HT,\tG}(q(|\ocX_{\mbf{1}}|)) \sub |\Fl_{G,\mu}|$, where $q : \ocX_{\mbf{1}} \to \wt{\cX}_{\wt{\mbf{1}}}^\ast$ is the natural map. From this, it follows that $\pi_{\HT,\tG} \circ q : \ocX_{\mbf{1}} \to \Fl_{\tG,\mu}$ factors through $\Fl_{G,\mu}$; this gives our desired extension of $\pi_{\HT,G}$ (one way to see this is via Lemma \ref{ad hoc Zariski closed}). Since $\ocX_{\mbf{1}}$ is perfectoid, the morphism $\ocX_{\mbf{1}} \to \Fl_{G,\mu}^{\lozenge}$ arises by diamondification from a unique morphism $\ocX_{\mbf{1}} \to \Fl_{G,\mu}$ of adic spaces by the definition of the diamondification functor.
\end{proof}

From now on we drop the subscript $G$ and write $\pi_{\HT}$ for $\pi_{\HT,G}$. Our remaining goal in this subsection is to establish some facts about the geometry of the Hodge--Tate period map. Recall from \S \ref{sec: anticanonical} that notions such as being ordinary (or more generally the valuation of the Hasse invariant), canonical or anti-canonical are compatible for $G$ and $\tG$ as they do not depend on the endomorphism structure. To discuss ordinarity in terms of flag varieties, we will need to define some subsets of $\Fl_{G,\mu}$ and $\Fl_{\tG,\mu}$. First, we define the ``opposites'' $\ol{P}_\mu$ and $\ol{N}$ of $P_\mu$ and $N$, respectively, by the functor of points ($R$ a ring)
\begin{align*}
  \ol{P}_\mu(R) & \defeq \left\{(g,r)\in G(R) \mid g = \left(\begin{array}{cc} * & 0\\ * & * \end{array}\right) \right\}; \\
  \ol{N}(R) & \defeq \left\{(g,r)\in G(R) \mid g = \left(\begin{array}{cc} I_n  & 0 \\ * & I_n \end{array}\right) \right\}.
\end{align*}
At good primes, $\ol{P}_\mu$ is the opposite parabolic of $P_\mu$ and $\ol{N}$ is its unipotent radical. Note that $\ol{P}_\mu = J_n P_\mu J_n^{-1}$ and $\ol{N} = J_n N J_n^{-1}$, where we recall that the element $J_n$ was defined in \S \ref{subsect:unitary_group}. We may define $\ol{\tP}_\mu \sub \tG$ and $\ol{\tN} \sub \tG$ analogously. We have an affine subspace 
\[
  J_n \ol{N}_{\Zp}P_{\mu,\Zp}/P_{\mu,\Zp} = P_{\mu,\Zp}J_n P_{\mu,\Zp}/P_{\mu,\Zp} \sub \Flr_{G,\mu},
\]
and we define $\Fl_{G,\mu}^a \sub \Fl_{G,\mu}$ to be the rigid generic fiber of the formal completion of $J_n \ol{N}_{\cO_C}P_{\mu,\cO_C}/P_{\mu,\cO_C}$ along $p=0$. This is an affinoid subspace. The construction provides us with a natural $\Qp$-structure, which we will use to talk about $\Qp$-points. One may define $\Fl_{\tG,\mu}^a \sub \Fl_{\tG,\mu}$ analogously. It is not hard to see that
\[
\Fl_{G,\mu}^a = \Fl_{G,\mu} \cap \Fl_{\tG,\mu}^a, 
\]
and that $\Fl_{\tG,\mu}^a$ is equal to the locus denoted by $\Fl_{\{g+1,\dotsc,2g\}}$ in \cite[\S 3.3]{scholze-galois}. We will need the following version for $G$ of \cite[Lemmas 3.3.19 and 3.3.20]{scholze-galois} (valid for $\tG$), which characterizes ordinarity in terms of $\pi_{\HT}$. Recall that the notation $-(0)$ denotes ordinary loci.

\begin{prop}\label{explicit description of ordinary stratum} $\pi_{\mathrm{HT}}^{-1}(\Fl_{G,\mu}(\Q_p))$ is the closure of the ordinary locus $\ocX_{\mbf{1}}(0) \sub \ocX_{\mbf{1}}$, and $\pi_{\mathrm{HT}}^{-1}(\Fl^a_{G,\mu}(\Q_p))$ is the closure of the anti-canonical ordinary locus $\ocX_{\mbf{1}}(0)_{a}\sub \ocX_{\mbf{1}}$.
\end{prop}

\begin{proof}
The first part follows from \cite[Lemma 3.3.19]{scholze-galois} since $\Fl_{G,\mu}(\Qp)=\Fl_{G,\mu}\cap \Fl_{\tG,\mu}(\Qp)$ and $\ocX_{\mbf{1}}(0)$ is the pullback of $ \wt{\cX}^\ast_{\wt{\mbf{1}}}(0)$ along $\ocX_{\mbf{1}} \to \wt{\cX}^\ast_{\wt{\mbf{1}}}$ (and moreover closure and pullback commute along $\ocX_{\mbf{1}} \to \wt{\cX}^\ast_{\wt{\mbf{1}}}$ for quasi-compact open subsets). The second part then follows similarly from \cite[Lemma 3.3.20]{scholze-galois}, since $\Fl_{G,\mu}^{a}(\Qp)= \Fl_{G,\mu} \cap \Fl_{\tG,\mu}^a (\Qp)$.
\end{proof}  

The affinoid $\Fl_{G,\mu}^a$ lies inside an affine open subset 
$$ \Fl_{G,\mu}^{\mathrm{nc}} \defeq (J_n \ol{N}_C P_{\mu,C}/ P_{\mu,C} )^{\ad} \sub \Fl_{G,\mu}. $$
Here, the superscript stands for ``non-canonical'', and we will refer to it as the non-canonical locus.\footnote{We have chosen this terminology to distinguish it from the ``anti-canonical'' spaces appearing, though we are aware that ``anti-canonical'' might fit better also for these spaces.} To relate $\Fl_{G,\mu}^{a}$ and $\Fl_{G,\mu}^{\nc}$, we consider the element 
\[
\gamma \defeq \mu(p) = \begin{pmatrix} pI_n & 0 \\ 0 & I_n \end{pmatrix} \in G(\Qp).
\]

\begin{prop}\label{cover of flag nc}
We have 
\[
\Fl_{G,\mu}^{\mathrm{nc}}=\bigcup_{k\in\Z_{\geq 0}}\Fl_{G,\mu}^{a}\cdot\gamma^{k}.
\]
Moreover, the sets $\Fl_{G,\mu}^{a}\cdot \gamma^{-k}$, $k\geq 0$, form a basis of quasicompact open neighborhoods of a point in $\Fl_{G,\mu}^{a}(\Qp)$.
\end{prop}

\begin{proof}
The right action of $\gamma$ on $J_n \ol{N}_{\Zp}P_{\mu,\Zp}/P_{\mu,Zp}$ is given by
\[
J_n \begin{pmatrix} I_n & 0 \\ A & I_n \end{pmatrix}P_{\mu,\Zp} \cdot \gamma = \gamma^{-1} J_n \begin{pmatrix} I_n & 0 \\ A & I_n \end{pmatrix}P_{\mu,\Zp} = J_n \begin{pmatrix} I_n & 0 \\ p^{-1}A & I_n \end{pmatrix}P_{\mu,\Zp}
\]
using that
\[
J_n^{-1} \gamma^{-1} J_n = \begin{pmatrix} I_n & 0 \\ 0 & p^{-1}I_n\end{pmatrix}.
\]
The proposition follows directly from this.
\end{proof}

To finish this section, we discuss the non-canonical locus on the level of Shimura varieties, which we define as
\[
\ocX^{\mathrm{nc}}_{\mbf{1}} \defeq \pi_{\HT}^{-1}(\Fl_{G,\mu}^{\nc}).
\]
The following result will be key for us in section \ref{stratifications}.

\begin{prop}\label{expanding the anticanonical locus} For any sufficiently small $\epsilon >0$, we have 
\[
\ocX^{\mathrm{nc}}_{\mbf{1}}= \bigcup_{k\in \Z_{\geq 0}} \ocX_{\mbf{1}}(\epsilon)_{a} \cdot \gamma^{k}.
\]
\end{prop}

\begin{proof}
By Proposition \ref{explicit description of ordinary stratum}, $\pi_{\HT}^{-1}(\Fl^{a}_{G,\mu}(\Qp))$ is the closure $\ol{\ocX_{\mbf{1}}(0)_{a}}$ of $\ocX_{\mbf{1}}(0)_{a}$ in $\ocX_{\mbf{1}}$. Note that 
\[
\ol{\ocX_{\mbf{1}}(0)_{a}} = \bigcap_{\epsilon > 0} \ocX_{\mbf{1}}(\epsilon)_a
\]
and that $\Fl_{G,\mu}^a(\Qp) = \bigcap_U U$, where $U$ runs over the set of quasicompact opens in $\Fl_{G,\mu}^a$ containing $\Fl_{G,\mu}^a(\Qp)$. In particular, we have
\[
\bigcap_{\epsilon > 0} \ocX_{\mbf{1}}(\epsilon)_a = \bigcap_U \pi_{\HT}^{-1}(U).
\]
In the constructible topology on $\ocX_{\mbf{1}}$ (which is compact and Hausdorff), the sets appearing in each intersection are both open and compact. By Cantor's intersection theorem, we may then find an $\epsilon >0$ and a $U$ such that
\[
 \pi^{-1}_{\HT}(U) \sub \ocX_{\mbf{1}}(\epsilon)_{a} \sub \pi_{\HT}^{-1}(\Fl_{G,\mu}^{a}).
\]
As $\Fl_{G,\mu}^{a}\cdot \gamma^{-m} \sub U$ for $m$ large enough by the second part of Proposition \ref{cover of flag nc}, the Proposition follows from the first part of Proposition \ref{cover of flag nc} by $G(\Qp)$-equivariance of $\pi_{\HT}$.
\end{proof}
 
\section{A stratification on the flag variety}\label{stratifications}

In this section, we will investigate the orbits of the action of the parabolic subgroup $P_{\mu}$ on the flag
variety $\Fl_{G,\mu}$ (in fact, it turns out to be advantageous to note that $P_{\mu}$ is conjugate to its opposite $\ol{P}_{\mu}$ and study $\ol{P}_{\mu}$-orbits, see Remark \ref{p vs p opp}). In the first two subsections, we briefly review the generalized Bruhat decomposition; see~\cite[\S 5]{borel-tits} for more information. We then study the corresponding stratification of $\Fl_{G,\mu}$, which we refer to as the Bruhat stratification, and its interaction with a certain collection of standard affine open subsets of $\Fl_{G,\mu}$. We finish by proving that various subsets of our Shimura varieties become perfectoid at intermediate infinite levels. This uses the aforementioned geometry of $\Fl_{G,\mu}$ and the equivariance of the Hodge--Tate period map to ``spread out'' Theorem \ref{anticanonical tower at intermediate levels 2}, and is the key geometric input we need to prove our vanishing results. We remark that our result on the perfectoidness of the non-canonical locus is generalization of a result of Ludwig \cite{ludwig} in the case of modular curves.

\subsection{Algebraic groups and Weyl groups}\label{algebraic groups and weyl groups}
Let $\mathrm{G}$ be a split connected reductive group over $\Q_p$, and let $\mathrm{T}$ be a maximal split torus of $\mathrm{G}$.
Let $\mathrm{B}$ be a Borel subgroup of $\mathrm{G}$, and let $\ol{\mathrm{B}}$ be the opposite Borel. We denote by $X^*(\mathrm{T})$ the group $\Hom(\mathrm{T}, \GG_m)$ of characters of $\mathrm{T}$.  Let $\mathfrak{g}=\Lie(\mathrm{G})$ be the Lie algebra of $\mathrm{G}$. For $\alpha\in X^*(\mathrm{T})$, we put
$$
\mathfrak{g}_{\alpha}\defeq \{X\in \mathfrak{g} \mid  \Ad(t)X=\alpha(t)X,\forall t\in \mathrm{T}\}
$$
and we define the relative root system $\Phi=\Phi(\mathrm{G},\mathrm{T})$ as
$$
\Phi \defeq \{\alpha\in X^*(\mathrm{T}) \mid \alpha\neq 0, \mathfrak{g}_{\alpha}\neq 0\}.
$$
We denote by $\Phi^+\subset \Phi$ the subset of positive roots corresponding to $\mathrm{B}$. The Weyl group $W=W(\mathrm{G},\mathrm{T})$ of the root system $\Phi$ is defined as $W=N_{\mathrm{G}}(\mathrm{T})/C_{\mathrm{G}}(\mathrm{T})$, where $N_{\mathrm{G}}(\mathrm{T})$ and $C_{\mathrm{G}}(\mathrm{T})$ are respectively the normalizer and the centralizer of $\mathrm{T}$ in $\mathrm{G}$. We denote by $\Delta$ the set of simple roots with respect to $\mathrm{T} \leq \mathrm{B}$, and denote by $S=\{s_{\alpha} \mid \alpha\in \Delta\} \subset W$ the set of simple reflections.

For  $I\subset S$, let $W_I$ be the subgroup of $W$ generated by all $s_i \in I$ and define the parabolic subgroup $\mathrm{P}_I$ as
$$
\mathrm{P}_I=\mathrm{B} W_I \mathrm{B}=\cup_{w\in W_I} \mathrm{B}w\mathrm{B}.
$$
For instance, $\mathrm{P}_\varnothing = \mathrm{B}$ and $\mathrm{P}_S = \mathrm{G}$. 
The group $\mathrm{P}_I$ is a closed, connected, self-normalizing subgroup of $\mathrm{G}$ containing $\mathrm{B}$. For $I,J\sub S$, if $\mathrm{P}_I=\mathrm{P}_J$, then $I=J$. Groups of this form are called standard parabolic subgroups, and every parabolic subgroup is conjugate to a unique standard parabolic subgroup. We let $\ol{\mathrm{P}}_I = \ol{\mathrm{B}} W_I \ol{\mathrm{B}}$ be the parabolic subgroup opposite to $\mathrm{P}_I$. As usual, we put $\Delta_{I}=\{\alpha \in \Delta \mid s_{\alpha}\in I \}$ and $\Phi_{I}=\Phi \cap \sum_{\alpha\in \Delta_{I}}\Z \alpha$; this is the root system of Levi factor $\mathrm{P}_{I}/\mathrm{N}_{I}$ (where $\mathrm{N}_{I}$ is the unipotent radical of $\mathrm{P}_{I}$, see below) and $\Delta_{I}$ are the simple roots with respect to the Borel subgroup $\mathrm{B}_{I}/\mathrm{N}_{I}$.

For $\alpha\in \Phi$, denote by $\mathrm{U}_{\alpha}$ the unipotent subgroup of $\mathrm{G}$ whose Lie algebra $\Lie(\mathrm{U}_\alpha)$ is $\mathfrak{g}_{\alpha}$. For $I\subset S$, the standard parabolic subgroup $\mathrm{P}_I$ admits a Levi decomposition, $$\mathrm{P}_I=\mathrm{L}_I \ltimes \mathrm{N}_I,$$ where
$\mathrm{N}_I\defeq \prod_{\alpha\in \Phi^+ - \Phi_I} \mathrm{U}_{\alpha}$ is the unipotent radical of $\mathrm{P}_{I}$ and $\mathrm{L}_I=\langle \mathrm{T},\mathrm{U}_{\alpha} \mid \alpha \in \Phi_I\rangle$.

\subsection{The generalized Bruhat decomposition}\label{generalized-Bruhat-decomposition}
Let the notation be as in the previous subsection. We first recall the notion of generalized Bruhat decomposition. The double cosets $\mathrm{P}_J \backslash \mathrm{G}/\mathrm{P}_I$ can be described in terms of certain Weyl group elements. 
Recall that $S$ is the set of simple reflections in $W$ corresponding to the set $\Delta$ of simple roots of $T$. The elements of $S$ generate $W$, and we call the \emph{length} of an element in $w \in W$ the smallest number $\ell$ such that $w$ can be written as $w = s_{i_1}\dotsm s_{i_\ell}$ with $s_i \in S$. We denote the length of $w$ by $\ell(w)$. We let $w_0$ be the longest element of $W$; it has order $2$. 

\medskip 

The Bruhat decomposition \cite[Thm. 5.15]{borel-tits} is
\[ \mathrm{G} = \coprod_{w \in W} \mathrm{B} w \mathrm{B} \,, \]
and there is also a generalized form \cite[Cor. 5.20]{borel-tits}
\[ \mathrm{G} = \coprod_{w \in W_J \backslash W / W_I} \mathrm{P}_J w \mathrm{P}_I \,. \]

\subsection{Some results on double cosets}
We will now state some general facts about double cosets that we will need.
These results are most naturally stated using an alternative form of the (generalized) Bruhat decomposition
\[ \mathrm{G} = \coprod_{w \in W} \ol{\mathrm{B}} w \mathrm{B}, \]
\[ \mathrm{G} = \coprod_{w \in W_J \backslash W / W_I} \ol{\mathrm{P}}_J w \mathrm{P}_I \,, \]
which can be obtained from the decompositions mentioned above using the identities
$\ol{\mathrm{B}} = w_0 \mathrm{B} w_0$, $\ol{\mathrm{P}}_J = w_0 \mathrm{P}_{w_{0}Jw_{0}} w_0$ and $W_{w_{0}Jw_{0}}=w_{0}W_{J}w_{0}$. We write $\ol{\mathrm{N}}_{I}$ for the unipotent radical of $\ol{\mathrm{P}}_{I}$.

\begin{lemma} \label{BwPcover}
For any $w \in W$,
\[ \ol{\mathrm{B}} w \mathrm{P}_I \subseteq w \ol{\mathrm{B}} \mathrm{P}_I = w \ol{\mathrm{P}}_I \mathrm{P}_I \,. \]
In particular, translates of $\ol{\mathrm{P}}_I \mathrm{P}_I$ cover $\mathrm{G}$.
\end{lemma}
\begin{proof}
By \cite[Thm.~5.15]{borel-tits},
\[ \ol{\mathrm{B}}w\mathrm{B} = (\ol{\mathrm{B}} \cap w\ol{\mathrm{B}}w^{-1}) w\mathrm{B} \sub \ol{\mathrm{B}}w\mathrm{B} \cap w\ol{\mathrm{B}}\mathrm{B} \subseteq w \ol{\mathrm{B}}\mathrm{B} \,. \]
Multiplying on the right by $\mathrm{P}_I$ gives $\ol{\mathrm{B}} w \mathrm{P}_I \subseteq w \ol{\mathrm{B}} \mathrm{P}_I$.
To see that $\ol{\mathrm{B}} \mathrm{P}_I = \ol{\mathrm{P}}_I \mathrm{P}_I$, observe that
\[ \ol{\mathrm{P}}_I \mathrm{P}_I = \ol{\mathrm{N}}_I \mathrm{M}_I \mathrm{P}_I = \ol{\mathrm{N}}_I \mathrm{P}_I \subseteq \ol{\mathrm{B}} \mathrm{P}_I \subseteq \ol{\mathrm{P}}_I \mathrm{P}_I \,. \qedhere \]
\end{proof}

\begin{lemma} \label{stratum dim matches group dim}
For any $w \in W$,
\[ \dim w \ol{\mathrm{P}}_I w^{-1} \cap \ol{\mathrm{N}}_I = \dim \ol{\mathrm{P}}_I w \mathrm{P}_I / \mathrm{P}_I \,. \]
\end{lemma}
\begin{proof}
  The intersection $w\ol{\mathrm{P}}_I w^{-1} \cap \ol{\mathrm{N}}_I$ is a unipotent
group whose set of roots is
\[ w(\Phi^- \cup \Phi_I) \cap (\Phi^- \setminus \Phi_I) = w(\Phi^- \cup \Phi_I) \setminus (\Phi^+ \cup \Phi_I) \,. \]
We have
\begin{IEEEeqnarray*}{rCl}
  \dim \ol{\mathrm{P}}_I w \mathrm{P}_I/\mathrm{P}_I & = & \dim \ol{\mathrm{P}}_I - \dim (\ol{\mathrm{P}}_I \cap w\mathrm{P}_Iw^{-1}) \\ & = & \#((\Phi^- \cup \Phi_I) \setminus w(\Phi^+ \cup \Phi_I)) \\ & = & \#((\Phi^+ \cup \Phi_I) \setminus w(\Phi^- \cup \Phi_I)) \,.
\end{IEEEeqnarray*}
Since $\#(\Phi^+ \cup \Phi_I) = \#(w(\Phi^- \cup \Phi_I))$,
\[ \# \left(w(\Phi^- \cup \Phi_I) \setminus (\Phi^+ \cup \Phi_I)\right) =  \#\left((\Phi^+ \cup \Phi_I) \setminus w(\Phi^- \cup \Phi_I)\right) \,. \qedhere \]
\end{proof}

\begin{remark}\label{p vs p opp} There is a further reason for stating our results in this fashion. We think of $\Flr_{\mathrm{G},I}\defeq \mathrm{G}/\mathrm{P}_I$ as the flag variety whose associated adic space receives a Hodge--Tate period morphism from a perfectoid Shimura variety (say, of Hodge type, with group split over $\Qp$). In this scenario, $\mathrm{P}_{I}$ is a parabolic associated with the Hodge--Tate filtration (which determines it up to conjugacy). The ``anti-canonical'' tower of the Shimura variety will live on the Shimura variety of level structure $\ol{\mathrm{P}}_{I}(\Zp)$ (or a conjugate of it), since it is supposed to parametrize complements of the canonical subgroup,
which split the Hodge--Tate filtration. Thus, to study how the perfectoid structure of the anti-canonical tower spreads out to the whole Shimura variety (either at full infinite level, or some intermediate levels), one is led to studying the $\ol{\mathrm{P}}_{I}$-orbits of $\Flr_{\mathrm{G},I}$. In the cases of interest in this paper, $\ol{\mathrm{P}}_{I}$ is conjugate to $\mathrm{P}_{I}$, so one can define the anti-canonical tower using $\mathrm{P}_{I}$. This is what we have done; it has the advantage of matching with \cite{scholze-galois}.
\end{remark}

\medskip
With this said, we will state the following direct consequences of Lemma \ref{BwPcover} and Lemma \ref{stratum dim matches group dim}.

\begin{cor}\label{key properties of stratification}
Assume that $I$ is such that $w_{0}Iw_{0}=I$ and let $w\in W$. Then $\mathrm{B}w\mathrm{P}_{I} \sub ww_{0}\mathrm{P}_{I}w_{0}\mathrm{P}_{I}$ and $\dim w\mathrm{P}_{I}w^{-1}\cap \mathrm{N}_{I} = \dim \mathrm{P}_{I}ww_{0}\mathrm{P}_{I}/\mathrm{P}_{I}$. In particular, Weyl group translates of $\mathrm{P}_{I}w_{0}\mathrm{P}_{I}$ cover $\mathrm{G}$. 
\end{cor}

\begin{proof}
We have 
$$\mathrm{B}w\mathrm{P}_{I} = w_{0}\ol{\mathrm{B}}w_{0}w\mathrm{P}_{I} \sub w\ol{\mathrm{P}}_{I}\mathrm{P}_{I} = ww_{0}\mathrm{P}_{I}w_{0}\mathrm{P}_{I}$$
by Lemma \ref{BwPcover}, proving the first part. For the second part, note that 
$$ w\mathrm{P}_{I}w^{-1} \cap \mathrm{N}_{I} = w_{0}\left( w_{0}ww_{0}\ol{\mathrm{P}}_{I}w_{0}w^{-1}w_{0} \cap \ol{\mathrm{N}}_{I} \right) w_{0}; $$
$$ \mathrm{P}_{I}ww_{0}\mathrm{P}_{I}/\mathrm{P}_{I} = w_{0}\left( \ol{\mathrm{P}}_{I}w_{0}ww_{0}\mathrm{P}_{I}/\mathrm{P}_{I} \right), $$
and use Lemma \ref{stratum dim matches group dim}. The last statement follows from the first and the Bruhat decomposition.
\end{proof}

As we alluded to above, the generalized Bruhat decomposition induces a decomposition of the flag variety $\Flr_{\mathrm{G},I}= \mathrm{G}/\mathrm{P}_I$ into locally closed subvarieties, which we loosely refer to as (generalized) \emph{Schubert cells}.
For $w \in W_I \backslash W / W_I$, we define
$\Flr_{\mathrm{G},I}^{w} \defeq \mathrm{P}_I w \mathrm{P}_I / \mathrm{P}_I$
to be the cells of the decomposition
\[
\Flr_{\mathrm{G},I}= \bigsqcup_{w\in W_I\backslash W/W_I}  \Flr_{\mathrm{G},I}^w \,.
\]
The cells $\Flr_{\mathrm{G},I}^{w}$
are locally closed in $\Flr_{\mathrm{G},I}$, and $\Flr_{\mathrm{G},I}^{w_0}$
is open in $\Flr_{\mathrm{G},I}$.

\subsection{The Bruhat stratification on $\mathscr{F}\ell_{G,\mu}$}\label{bruhat stratification}

We now go back to considering the case of $G$ (over $\Qp$) and its (algebraic and rigid) flag varieties $\Flr_{G,\mu}$ and $\Fl_{G,\mu}$ from \S \ref{recollections}. In this case, our parabolic is $P_{\mu,\Qp}$. The longest element $w_0$ of the Weyl group is represented by the element $J_n$ from \S \ref{subsect:unitary_group}. One way to see this is as follows. The diagonal matrices in $G_{\Qp}$ form a maximal split torus $T$, and one checks that $J_n$ normalizes $T$. The upper triangular matrices form a Borel $B$ containing $T$, and its opposite $\ol{B}$ consists of the lower triangular matrices. In particular, one checks that $\ol{B} = J_n B J_n^{-1}$, and this characterizes the longest element in the Weyl group. Since $\ol{P}_\mu = J_n P_\mu J_n^{-1}$, this means that the discussion above (including Corollary \ref{key properties of stratification}) is available to us. 

\medskip
We let $I$ be the subset of simple roots corresponding to $P_{\mu,\Qp}$, where we use $B$ and $T$ above to define simple roots (the role of $I$ itself will be notational from this point on). Let $W$ and $W_I$ be the Weyl groups as above, specialized to $G_{\Qp}$. For any $w\in W_I \backslash W / W_I$, we have a generalized Schubert cell $\Flr_{G,\mu}^w$, and we let
\[
\Fl_{G,\mu}^w \defeq (\Flr_{G,\mu}^w \times_{\Spec \Qp} \Spec C)^{\ad}
\]
denote their analytifications over $C$. From the definitions, we see that
\[
\Fl_{G,\mu}^{w_0} = \Fl_{G,\mu}^{\nc},
\] 
which links the discussion in \S \ref{hodgetatemapsection} to our current discussion.

\medskip
Related to these generalized Schubert cells, we will also consider the open cover of $\Fl_{G,\mu}$ consisting of translates of the open cell $\Fl_{G,\mu}^{w_0}$ by elements of $N_G(T)$. For $w\in N_G(T)$, we set
\[
\cU^w \defeq \Fl_{G,\mu}^{w_0}\cdot w^{-1}.
\] 
One checks directly that the action factors through the quotient $W/W_I$. These translates cover $\Fl_{G,\mu}$ by Lemma \ref{BwPcover}, and $\cU^{\id} = \Fl_{G,\mu}^{w_0} = \Fl_{G,\mu}^{\nc}$. The key result that relates this open cover to the Bruhat stratification is Lemma~\ref{finding a J} below, whose statement involves the natural projection map $\delta\colon W/W_I\to W_I\backslash W/W_I$ and the group $N_0$ from Definition~\ref{frequent infinite level subgroups}.

\begin{lemma}\label{finding a J} Let $w\in W_I\backslash W/W_I$ and $x\in \Fl_{G,\mu}^w$ be a point. Then there exists
$w^\prime \in \delta^{-1}(ww_0)$ such that $x \cdot N_0\subset \cU^{w^\prime}$. 
\end{lemma}
\begin{proof}
Note that right multiplication by $w_0$ induces a well defined automorphism of $W_I \backslash W / W_I$, since $w_0 W_I w_0 = W_I$. By partial properness it suffices to prove the lemma for geometric rank $1$ points; by base change it then suffices to restrict to the case of $(C,\cO_{C})$-points, so we only need to consider classical rigid points and hence we can write the proof in terms of algebraic geometry over $C$; we then omit notation for change of base fields etc. Since 
\[
\Fl_{G,I}^{w}=P_{\mu}wP_{\mu}/P_{\mu} = \bigsqcup_{w_1 \in \delta^{-1}(w)} B w_1 P_\mu /P_\mu,
\]
we can choose a $w_1 \in \delta^{-1}(w)$ such that $x\in Bw_{1}P_{\mu}/P_{\mu}$. We then have
\[
x \cdot N_{0} \sub Bw_1P_\mu /P_\mu \sub w_{1}w_{0}P_{\mu}w_{0}P_{\mu}/P_{\mu}=\Fl_{G,\mu}^{w_{0}} \cdot (w_{1}w_{0})^{-1} = \cU^{w_1w_0}
\]
by Corollary \ref{key properties of stratification}. Setting $w^\prime = w_1w_0$ finishes the proof.
\end{proof}

Of course, the proof shows that the orbit of $x$ under the entire Borel is contained in $\cU^{w^\prime}$; we have chosen the above formulation since it is the precise statement we will need later.

\subsection{Perfectoid strata at intermediate levels}\label{perfectoid strata at intermediate levels}
Let $w \in W/W_I$. For any integer $m\geq 1$, define the open subgroups
\[
\Gamma_{1,w}(p^m)\defeq \Gamma_1(p^m) \cap w\Gamma_0(p^m)w^{-1}
\]
and the closed subgroup
\[
 N_{0,w} = \bigcap_{m\geq 0}\Gamma_{1,w}(p^m).
\] 
In these definitions we represent elements in $W$ by elements in $G(\Zp)$ (which we may do since $G$ is split over $\Zp$); the subgroups above are then independent of the chosen representatives. When $w=\id$, $N_{0,\id}=N_{0}$. Note that in general, $N_{0,w}$ is a subgroup of $N_{0}$ and naturally a $\Zp$-submodule of $N_{0}$.

\begin{lemma}\label{ranks} We have $\rank_{\Z_p} N_{0,w} = \dim \Fl^{w w_0}_{G,\mu}$. Moreover, $N_{0,w}$ is a saturated $\Zp$-submodule of $N_{0}$, so $\rank_{\Z_p} N_{0}/N_{0,w} = d - \dim \Fl^{w w_0}_{G,\mu}$, where $d = \rank_{\Zp}N_0$ is also the dimension of the Shimura varieties for $G$.
\end{lemma}
\begin{proof}
By definition, we have $N_{0,w}= \left( N \cap w P_{\mu}w^{-1} \right)(\Zp)$ and  $N \cap wP_{\mu}w^{-1}$ is product of unipotent subgroups $U_{\alpha}$ for a certain subset (depending on $w$) of roots $\alpha$ contained in $N_{\Zp}$. This shows that $N_{0,w}$ is saturated in $N_0$ and that $\rank_{\Z_p} N_{0,w} = \dim N_{\Zp} \cap w P_{\mu,\Zp}w^{-1}$. Thus, by Corollary \ref{key properties of stratification}, $\rank_{\Z_p} N_{0,w} = \dim \Fl_{G,\mu}^{w w_{0}}$ as desired. The last statement then follows directly from the first two statements.
\end{proof}

\noindent Let $H \sub P_{\mu,0}$ be a closed subgroup. The Hodge--Tate period morphism induces the composition 
\[
\pi_{\HT/H}\colon (\ocX_{H})_{\et} \to 
|\ocX_{H}|\cong |\ocX_{\mbf{1}}|/ H \to |\Fl_{G,\mu}|/ H,
\] 
which is a morphism of sites; here we use Lemma \ref{ad hoc exhibiting a diamond as a quotient}. Note that the stratification $\Fl_{G,\mu} = \sqcup_{w\in W_I \backslash W/W_I} \Fl^w_{G,\mu}$
descends to the quotient $|\Fl_{G,\mu}|/H$ since $H \sub P_{\mu,0}$. In particular, the open Schubert cell $\Fl_{G,\mu}^{\nc}$ 
descends to an open stratum $|\Fl_{G,\mu}^{\nc}|/H$ in $|\Fl_{G,\mu}|/ H$, and we set
\[
\ocX_{H}^{\mathrm{nc}} \defeq \pi_{\HT/H}^{-1}\left(|\Fl_{G,\mu}^{\nc}|/H \right).
\]
We will be most interested in the subgroups $N_{0,w}\sub P_{\mu,0}$. In this case, for any closed subgroup $H \sub N_{0,w}$,
$|\cU^w|/H$ is an open subset of $|\Fl_{G,\mu}|/H$.
We put 
\[
\ocX^{w}_{H} \defeq 
\pi_{\HT/H}^{-1}\left(|\cU^w|/H\right).
\]
This is an open sub-diamond of $\ocX_{H}$. The main result of this section is that it is a perfectoid space. 

\medskip

We prove this in two steps: first, we show the result for the case $w=\id$, when $\ocX^{w}_{\tH} = \ocX^{\mathrm{nc}}_{\tH}$. This can be done more generally for all closed $H \sub P_{\mu,0}$ by adapting an argument of Ludwig from the setting of modular curves~\cite[Sect. 3.4]{ludwig}. For a general $w$, we use the action of $w$ to translate the perfectoid structure and deduce the result in general from the result for the non-canonical locus. We use freely the notation from subsection \ref{hodgetatemapsection}, in particular the element $\gamma = \mu(p) \in G(\Qp)$ considered there.

\begin{thm}\label{perfectoid non-canonical stratum} Let $H\sub P_{\mu,0}$ be a closed subgroup. Then the diamond 
  $\ocX^{\mathrm{nc}}_{H}$ is a perfectoid space. More precisely, $|\ocX^{\mathrm{nc}}_{H}|$ is covered by the increasing union of quasicompact open subsets $|\ocX_{\mbf{1}}(\epsilon)_{a}|\gamma^{k}/H$ for $k \in \Z_{\geq 0}$ (and sufficiently small $\epsilon >0$), and the corresponding spatial diamonds are affinoid perfectoid with strongly Zariski closed boundary.
\end{thm}

\begin{proof} 
We consider the projection map 
\[
\ocX_{\mbf{1}}\to \ocX_{H},
\]
which by Lemma \ref{ad hoc exhibiting a diamond as a quotient} identifies $|\ocX^{\mathrm{nc}}_{H}|$ with $|\ocX^{\mathrm{nc}}_{\mbf{1}}|/H$. By Corollary \ref{expanding the anticanonical locus}, the open subdiamonds corresponding to the open sets $|\ocX_{\mbf{1}}(\epsilon)_{a}|\gamma^{k}/H$ are increasing and cover $|\ocX^{\mathrm{nc}}_{H}|$. Note here that $|\ocX_{\mbf{1}}(\epsilon)_{a}|\gamma^{k}$ is $\gamma^{-k}P_{\mu,0}\gamma^{k}$-stable and that $\gamma^{-k}P_{\mu,0}\gamma^{k} \supseteq P_{\mu,0} \supseteq H$ (by direct computation), so this quotient makes sense. It remains to prove that these opens are affinoid perfectoid for all $k\geq 0$, with strongly Zariski closed boundary. On the level of topological spaces, we have an isomorphism
\[
 \gamma^{-k} \colon |\ocX_{\mbf{1}}(\epsilon)_{a}|\gamma^{k}/H \toisom |\ocX_{\mbf{1}}(\epsilon)_{a}|/\gamma^{k}H\gamma^{-k}
 \]
which is induced from an isomorphism 
\[ \gamma^{-k} \colon \ocX_{H} \toisom \ocX_{\gamma^{k} H \gamma^{-k}}
\]
of diamonds. Since $\gamma^{k}H\gamma^{-k} \sub \gamma^{k}P_{\mu,0}\gamma^{-k} \subseteq P_{\mu,0}$, $\ocX_{\gamma^{k} H \gamma^{-k}}(\epsilon)_{a}$ is affinoid perfectoid with strongly Zariski closed boundary by Theorem \ref{anticanonical tower at intermediate levels 2} and we may conclude that the subdiamond of $\ocX^{ \nc}_{H}$ corresponding to the open subset $|\ocX_{\mbf{1}}(\epsilon)_{a}|\gamma^{k}/H$ is affinoid perfectoid with strongly Zariski closed boundary, as desired.
\end{proof}

We now consider the case of a general $w \in W/W_I$.

\begin{thm}\label{general perfectoid stratum}\leavevmode Let $H_w\sub N_{0,w}$ be a closed subgroup. Then the diamond 
  $\ocX^{w}_{H_w}$ is a perfectoid space. More precisely, $|\ocX^{w}_{H_w}|$ is covered by the increasing union of quasicompact open subsets $|\ocX_{\mbf{1}}(\epsilon)_{a}|\gamma^{k}w^{-1}/H_w$ for $k \in \Z_{\geq 0}$ (and sufficiently small $\epsilon >0$), and the corresponding spatial diamonds are affinoid perfectoid with strongly Zariski closed boundary.
\end{thm}

\begin{proof} We choose a representative $w\in G(\Zp)$ of $w\in W/W_I$. Note that right multiplication by $w^{-1}$ induces an isomorphism of diamonds $ \ocX^{\mathrm{nc}}_{w^{-1}H_w w} \toisom \ocX^{w}_{H_w}$. To see this, note that right multiplication by $w^{-1}$ induces an isomorphism of the 
diamonds $\ocX_{w^{-1}H_w w}$ and $\ocX_{H_w}$. 
Now it remains to check that multiplication by $w^{-1}$ sends the non-canonical locus to the $w$-non-canonical locus. 
This can be done on the level of topological spaces, thus after using Lemma~\ref{ad hoc exhibiting a diamond as a quotient} 
to identify $|\ocX_{H_w}|$ with $|\ocX_{\mbf{1}}|/H_w$. The claim can then be checked on the flag varieties (using equivariance of the Hodge--Tate period map), and so follows from the definitions. Moreover, this isomorphism identifies $|\ocX_{\mbf{1}}(\epsilon)_{a}|\gamma^{k}/w^{-1}H_w w$ with $|\ocX_{\mbf{1}}(\epsilon)_{a}|\gamma^{k}w^{-1}/H_w$. The theorem then follows from Theorem \ref{perfectoid non-canonical stratum}.  
\end{proof}

\noindent We denote $\Fl_{G,\mu}^{a}\cdot\gamma^{k}$ by $\Fl_{G,\mu}^{\nc}(k)$. We denote $\Fl_{G,\mu}^{\nc}(k)\cdot w^{-1}$ by $\cU^w(k)$. Later, we will also need the following result. 

\begin{lemma}\label{standard neighborhoods are stable} The open subset $\cU^w(k)\subset \Fl_{G,\mu}$ is stable under the action of $N_{0,w}$. 
\end{lemma}

\begin{proof} By the definition of $N_{0,w}$, it is enough to see that $\cU^w(k)$ is stable under $w P_{\mu,0}w^{-1}$ or equivalently that $\Fl_{G,\mu}^{\nc}(k)$ is stable under $P_{\mu,0}$. As $\gamma^{-k}P_{\mu,0}\gamma^{k} \supseteq P_{\mu,0}$ this reduces us to the case $k=0$, which follows by explicit computation.
\end{proof}

\section{Vanishing of compactly supported cohomology}\label{vanishing}

In this section we put the ingredients together and prove the vanishing theorem. The first subsection is devoted to reducing it to a statement about the fibers of the Hodge--Tate period map (Theorem \ref{bound on higher direct images}), which is mostly formal using the machinery of \'etale cohomology of diamonds and our results from section \ref{diamonds and cohomology}. We then introduce a technical result about constructing invariant rational neighborhoods of orbits on affinoid adic spaces with a continuous action of a profinite group (Proposition \ref{rational neighborhood}), which we believe is important and interesting in its own right. We conclude by proving Theorem \ref{bound on higher direct images} using the results of section \ref{stratifications}, which finishes the proof of the vanishing theorem.

\subsection{The main result and first reductions}\label{subsect:first-reductions}

Recall that $d$ denotes the dimension of our Shimura varieties. Our main theorem is:

\begin{thm}\label{main thm} Let $r\in \Z_{\geq 1}$ and let $K \sub G(\Zp)$ be an open subgroup. If $i>d$, then 
\[
 \varinjlim_{m} H^{i}_{c}(X_{K \cap \Gamma_1(p^{m})}(\C),\Z/p^r)=0.
 \]
\end{thm}

\noindent Here $H_{c}$ denotes compactly supported singular cohomology. For the applications in this paper, the most important case is $K = G(\Zp)$, and the reader would not lose anything by making this assumption in the proof. On the other hand, our proof allows us to treat general $K$, so we have opted for this extra generality since it is not clear to us if the case of general $K$ follows from the special case $K = G(\Zp)$. We also have the following more general version of Theorem \ref{main thm}, which is a direct consequence of it.

\begin{cor}[to Theorem \ref{main thm}]\label{cor to main}
Let $r\in \Z_{\geq 1}$ and let $K_{m} \sub \Gamma_1(p^{m})$ be open subgroups for $m \geq 1$  with $K_{m+1}\sub K_{m}$ for all $m$. If $i>d$, then 
\[
 \varinjlim_{m} H^{i}_{c}(X_{K_{m}}(\C),\Z/p^r)=0.
 \]
\end{cor}

\noindent We remark that Theorem \ref{main thm} is false in general for usual cohomology; this fact is crucial for our strategy to remove the nilpotent ideal in section \ref{Borel-Serre}.

\medskip

To prove Theorem \ref{main thm}, note first that it suffices to treat the case $r=1$ by the usual d\'evissage argument. Then, applying the usual string of comparison theorems (between singular cohomology and \'etale cohomology of schemes, and between \'etale cohomology of schemes and \'etale cohomology of adic spaces), we find that Theorem \ref{main thm} is equivalent to the statement that 
\[
 \varinjlim_{m} H^{i}_{\et}(\ocX_{K \cap \Gamma_1(p^m)},j_{!}\F_{p})=0 
\]   
for $i>d$. Here, and in the rest of this section, we write $j$ for any open immersion $\cU \to \cV$, where $\cV$ is any locally spatial diamond that carries a quasi-pro-\'etale map $\cV \to \ocX_{H}$ for some closed subgroup $H \sub G(\Zp)$, $\cU = \cV \times_{\cX^{\ast}_H}\cX_H$ and $j$ is the projection onto the first factor. To proceed from here, the primitive comparison theorem for torsion coefficients \cite[Thm. 3.13]{scholze-survey} gives us
\[
 \varinjlim_{m} H^{i}_{\et}(\ocX_{K \cap \Gamma_1(p^{m})},j_{!}\F_{p})^{a} \otimes_{\F_p}\cO_C/p \cong \varinjlim_{m} H^{i}_{\et}(\ocX_{K \cap \Gamma_1(p^{m})},j_{!}\F_{p}\otimes \cO_{\ocX_{K \cap \Gamma_1(p^m)}}^{+}/p)^{a}
 \]
and it suffices to prove that the right hand side vanishes. We record the following (more convenient) descriptions of the sheaf appearing on the right-hand side. The lemma is stated in the generality needed in the paper; further generalizations should be possible whenever one has a sufficiently well behaved notion of Zariski closed sets.

\begin{lemma}\label{lower shriek and ideals}
Let $V$ be a rigid space over a nonarchimedean field $K/\Qp$, or an affinoid perfectoid space over $K$. Let $Z\sub V$ be a Zariski closed subset of $V$ and let $U$ be the open complement. Write $j \colon U \hookrightarrow V$ for the open immersion and let $\cI^{+}\sub \cO_{V}^{+}$ be the subsheaf of functions which vanish on $Z$. Then we have natural isomorphisms
$$ (j_{!}\F_{p}) \otimes \cO_{V}^{+}/p \cong \cI^{+}/p \cong j_{!}(\cO_{U}^{+}/p). $$ 
\end{lemma}

\begin{proof}
We start with the first isomorphism. By applying the adjunctions for $j_{!}$ and $\otimes$, the natural map $\F_{p} \to \cO_{U}^{+}/p = j^{\ast}(\cI^{+}/p)$ gives a map $(j_{!}\F_{p}) \otimes \cO_{V}^{+}/p \to \cI^{+}/p$. This is clearly an isomorphism over $U$, and on $Z$ all stalks on both sides are $0$, so it is an isomorphism on the whole of $V$. For the second isomorphism, adjunction gives a natural map $j_{!}(\cO_{U}^{+}/p) \to \cI^{+}/p$, which is an isomorphism for the same reason as before.
\end{proof}

\noindent For later use, we also explicitly record the following fact, which is implicit in \cite[\S 4.1]{scholze-galois}.

\begin{prop}\label{strongly Zariski closed vanishing}
Let $V$ be an affinoid perfectoid space over a perfectoid field $(K,K^{+})$ over $\Qp$ and let $Z\sub V$ be a strongly Zariski closed subset, with its induced structure of an affinoid perfectoid space. Set $j \colon U=V\setminus Z \hookrightarrow V$. Then $H^{i}_{\et}(V,j_{!}(\cO_{U}^{+}/p))^{a}=0$ for $i>0$.
\end{prop} 

\begin{proof}
Let $\cI^{+}\sub \cO_{V}^{+}$ be the subsheaf of functions vanishing on $Z$. By Lemma \ref{lower shriek and ideals}, $j_{!}(\cO_{U}^{+}/p)\cong \cI^{+}/p$. By the long exact sequence associated with 
$$0 \to \cI^{+} \to \cI^{+} \to \cI^{+}/p \to 0,$$
it suffices to show that $H^{i}_{\et}(V,\cI^{+})^{a}=0$ for $i>0$. Since $Z\sub V$ is strongly Zariski closed, we have a short exact sequence
$$ 0 \to \cI^{+a} \to \cO^{+a}_{V} \to \cO^{+a}_{Z} \to 0 $$
of almost sheaves on $V_{\et}$ by the definition and \cite[Lemma 2.2.9]{scholze-galois}. Taking the long exact sequence and using \cite[Proposition 7.13]{scholze-perfectoid}, we see that $H^{i}_{\et}(V,\cI^{+})^{a}=0$ for $i>1$ and that $H^{1}_{\et}(V,\cI^{+})^{a}$ is the cokernel of $H^{0}(V,\cO^{+}_{V})^{a} \to H^{0}(Z,\cO^{+}_{Z})^{a}$, which is $0$ since $Z\sub V$ is strongly Zariski closed. 
\end{proof}

Keeping Lemma~\ref{lower shriek and ideals} in mind, we apply (the almost version of) Proposition \ref{invlimcoh} to get
\[
  \varinjlim_{m} H^{i}_{\et}(\ocX_{K \cap \Gamma_1(p^m)},j_{!}(\cO_{\cX_{K \cap \Gamma_1(p^m)}}^{+}/p))^{a} \cong H^{i}_{\et}(\ocX_{K \cap N_{0}},j_{!}(\cO_{\cX_{K \cap N_{0}}}^{+}/p)^{a}), 
\]
where $(\cO_{\cX_{K \cap N_{0}}}^{+}/p)^{a}$ is the almost sheaf on $(\cX_{K \cap N_{0}})_{\et}$ constructed in subsection \ref{structure sheaves}. Here we have implicitly used some results from subsections  \ref{structure sheaves} on the sheaves $(\cO^{+}/p)^{a}$, namely that $(\cO^{+}/p)^{a}$ is compatible with the usual construction on rigid spaces and that its construction commutes with pullbacks by quasi-pro-\'etale maps. Furthermore we have used Lemma \ref{shriekbasechange} and Proposition \ref{finite maps}. To avoid some notational difficulties in the future we make the following convention: for any quasi-pro-\'etale $\cV \to \cX$ we write $(\cO_{\cX}^{+}/p)^{a}$ for the almost sheaf on $(\cO_{\cV}^{+}/p)^{a}$ on $\cV_{\et}$.
This is reasonable since, by the construction in subsection \ref{structure sheaves}, $(\cO_{\cV}^{+}/p)^{a}$ is the pushforward of the almost sheaf $(\cO_{\cX}^{+}/p)^{a}$ on $\cX_{\qp}$ via the natural map $\cX_{\qp} \to \cV_{\et}$.

\medskip
To proceed from here, we will use the morphism 
\[
\pi_{\HT/K \cap N_{0}} \colon \left(\ocX_{K \cap N_{0}}\right)_{\et} \to |\Fl_{G,\mu}|/(K \cap N_{0})
\]
from subsection \ref{perfectoid strata at intermediate levels}. We have a Leray spectral sequence 
\[
 E_{2}^{rs}=H^{r}(|\Fl_{G,\mu}|/(K \cap N_{0}),R^{s}\pi_{\HT/K \cap N_{0},\ast}j_{!}(\cO_{\cX}^{+}/p)^{a}) \implies H^{r+s}_{\et}(\ocX_{K \cap N_{0}},j_{!}(\cO_{\cX}^{+}/p)^{a}). 
\] 
The key result in this section is then the following vanishing result for the stalks of $R^{i}\pi_{\HT/\tK \cap \tN_0,\ast}j_{!}(\cO_{\cX}^{+}/p)^{a}$:

\begin{thm}\label{bound on higher direct images}
Let $w\in W_{I}\backslash W / W_{I}$ and let $x\in |\Fl_{G,\mu}^{w}|/(K \cap N_{0})$. Then we have 
\[
 R^{i}\pi_{\HT/K \cap N_0,\ast}j_{!}(\cO_{\cX}^{+}/p)_{x}^{a} = 0
 \]
for $i>d-\dim \Fl_{G,\mu}^{w}$.
\end{thm}

Our main theorem follows rather quickly from Theorem \ref{bound on higher direct images}, using a bound on the cohomological dimension of spectral spaces due to Scheiderer.

\begin{proof}[Proof of Theorem \ref{main thm}]
By the Leray spectral sequence for $\pi_{\HT/K \cap N_{0}}$, it suffices to prove that $H^{r}(|\Fl_{G,\mu}|/(K \cap N_{0}),R^{s}\pi_{\HT/K \cap N_{0},\ast}j_{!}(\cO_{\cX}^{+}/p)^{a})=0$ for $r+s>d$. Fix $r$ and assume that $s>d-r$. Let $S_{r}$ be the set of $w\in W_{I}\backslash W /W_{I}$ for which $\dim \Fl^w_{G,\mu} < r$, and let $Y_{r}=\bigcup_{w\in S_{r}}\Fl_{G,\mu}^{w}$. Let $\ol{Y}_r$ be the Zariski closure of $Y_r$; $\ol{Y}_r$ is $N_0$-invariant and of dimension $<r$ since $Y_r$ is. By Proposition \ref{spectral quotient}, $|\ol{Y}_{r}|/(K \cap N_{0}) \sub |\Fl_{G,\mu}|/(K \cap N_{0})$ is a closed spectral subspace of dimension $<r$. 

\medskip

We claim that $\cF^{s} \defeq R^{s}\pi_{\HT/K \cap N_0,\ast}j_{!}(\cO_{\cX}^{+}/p)^{a}$ is supported on $|\ol{Y}_{r}|/(K \cap N_{0})$. To see this, let $x\notin |Y_{r}|/(K \cap N_{0})$. Then $x\in |\Fl_{G,\mu}^{w}|/(K \cap N_{0})$ for some $w$ with $\dim \Fl_{G,\mu}^{w} \geq r$, so by Theorem \ref{bound on higher direct images}, $ \cF_{x}^{s} = 0$ since $s>d-r\geq d-\dim \Fl_{G,\mu}^{w}$. This proves that $\cF^{s}$ is supported on $|\ol{Y}_{r}|/(K \cap N_{0})$, and it then follows from \cite[Corollary 4.6]{scheiderer} that $H^{r}(|\Fl_{G,\mu}|/(K \cap N_{0}),\cF^{s})=H^{r}(|\ol{Y}_{r}|/(K \cap N_{0}),\cF^{s})=0$, as desired.
\end{proof}

\noindent The rest of this section is devoted to the proof of Theorem \ref{bound on higher direct images}.

\subsection{Constructing invariant rational neighborhoods}\label{sec: invariant neighborhoods}

Before we prove Theorem \ref{bound on higher direct images}, we prove an important technical result on profinite group actions on affinoid adic spaces, which we believe is interesting in its own right.

\begin{prop}\label{rational neighborhood}
Let $(A,A^+)$ be a Huber pair, and let $H$ be a profinite group acting
continuously on $(A,A^+)$.
Let $Y = \Spa(A,A^+)$, let $x \in Y$, and let $U \subset Y$ be an $H$-stable
open containing $x$.  Then there exists an $H$-stable open $V \subset U$
containing $x$ that is a rational subset of $Y$.
\end{prop}

\begin{lem} \label{open stabilizer}
Let $(A,A^+)$ be Huber pair,
and let $H$ be a profinite group acting continuously on $(A,A^+)$.
If $W \subset \Spa(A,A^+)$ is a rational subset, then the stabilizer of $W$ is open in $H$.
\end{lem}
\begin{proof}
Let $W$ be defined by the inequalities
\[ |a_1|,\dotsc, |a_n| \le |a_0| \ne 0 \]
for some $a_0,...,a_n \in A$ that generate an open ideal.
Let $M$ be the set of elements of the form
$\sum_{i=0}^n \lambda_i a_i$ with $\lambda_i \in A^{\circ \circ}$.
Then $M$ is open in $A$ by \cite[Lemma 1.1]{huber-adic-spaces}.
So there is an open normal subgroup $K$ of $H$ such that
$ka_i - a_i \in M$ for all $k \in K$, $i \in \{ 0,...,n \}$.

For any $x \in W$, $k \in K$, $i \in \{0,...,n\}$,
\[ |ka_i-a_i|_x < \max_{j=0,\dotsc,n} |a_j|_x = |a_0|_x \,. \]
Therefore, by the strong triangle inequality,
\begin{IEEEeqnarray*}{rCl}
|ka_i|_x & \le & \max \{ |ka_i - a_i|_x,|a_i|_x \} \le |a_0|_x \text{, and} \\
|ka_0|_x & = & \max \{|ka_0 - a_0|_x,|a_0|_x \} = |a_0|_x \ne 0.
\end{IEEEeqnarray*}
Hence $x\cdot k \in W$.  So $K$ stabilizes $W$.
\end{proof}

\begin{proof}[Proof of Proposition \ref{rational neighborhood}]
Let $W$ be a rational subset of $Y$ containing $x$ and contained in $U$.
Let $W$ be defined by the inequalities $|a_1|, ..., |a_n| \le |a_0| \ne 0$
for some $a_0,...,a_n \in A$ that generate an open ideal, and define $K$
as in the proof of Lemma \ref{open stabilizer}.

Choose a section (not necessarily a homomorphism) $\sigma\colon H/K \to H$ of the quotient $H \to H/K$.
Let $S$ be the set of $h \in H/K$ such that the $|\sigma(h) a_i|_x$
are not all zero.  From the definition of $K$, we see that
$S$ is independent of the choice of $\sigma$.
Let $\mathscr{S}$ denote the set of functions $S \to \{0,...,n\}$.
For any $f \in \mathscr{S}, h \in H/K$, define
\[ b_{f,h} \defeq \prod_{s \in S} \sigma(hs) a_{f(s)} \,. \]
The ideal of $A$ generated by the $b_{f,h}$ contains the product of the ideals
$\sigma(s) \cdot (a_0,...,a_n)$ for $s \in S$, and hence is open.

Let $m = \#(H/K)$.
For any $f \in \mathscr{S}$ and $i \in \{1,...,m\}$,
define $e_{f,i}$ to be the $i$th elementary symmetric polynomial in
$(b_{f,h})_{h \in H/K}$.
We observe that
\[ b_{f,h}^{m} = \sum_{i=1}^{m} (-1)^{i-1} e_{f,i} b_{f,h}^{m-i} \,. \]
In particular, the ideal of $A$ generated by the $e_{f,i}$ contains the ideal
generated by the $b_{f,h}^{m}$.  The latter ideal contains the
$\left(\#(\mathscr{S} \times H/K) (m-1) + 1\right)$-st power of the ideal generated by the $b_{f,h}$.
So the $e_{f,i}$ generate an open ideal of $A$.

Choose $\phi \in \mathscr{S}$ so that $\phi(1)=0$ and
for all $s \in S$ and all $i \in \{0,...,n\}$,
\[|\sigma(s) a_i|_x \le |\sigma(s) a_{\phi(s)}|_x  \,. \]
For all $f \in \mathscr{S}$ and $h \in H/K$,
\[|b_{f,h}|_x \le |b_{\phi,1}|_x \,; \]
to see this, note that if $|b_{f,h}|_x \ne 0$, then $hs \in S$ for all $s \in S$ and thus multiplication by $h$ permutes the elements of $S$.
By the theory of Newton polygons, there exists $r \in \{1,...,m\}$
such that
\[ |b_{\phi,1}|_x = |e_{\phi,r}|_x^{1/r} \,. \]
Choose one such $r$.  Then for all $f \in \mathscr{S}$ and all
$i \in \{1,...,m\}$,
\[ |e_{f,i}|_x^{1/i} \le  |e_{\phi,r}|_x^{1/r} \,. \]

Let $V$ be the subset of $Y$ defined by the inequalities
\[ |e_{f,i}|^{1/i} \le |e_{\phi,r}|^{1/r} \ne 0 \quad \forall i \in \{1,...,m\},\, f \in \mathscr{S} \,. \]
Equivalently, $V$ is defined by the inequalities $|e_{f,i}|^{m!/i} \le |e_{\phi,r}|^{m!/r} \ne 0 $.
Since the $e_{f,i}$ generate an open ideal of $A$, the $e_{f,i}^{m!/i}$ do as well, so $V$ is rational.
Moreover, $V$ contains $x$.

Now we check that $V \subseteq \bigcup_{h \in H} W \cdot h \subseteq U$.
Let $y \in V$.  By the theory of Newton polygons, there exists some
$h \in H/K$ so that
\[ |b_{f,h'}|_y \le |b_{\phi,h}|_y \ne 0 \quad \forall f \in \mathscr{S},\, h' \in H/K \,. \]
In particular, by considering the case where $h=h'$ and $f$ and $\phi$ agree except at the identity, we find that
\[ |\sigma(h) a_i|_y \le |\sigma(h) a_0|_y \ne 0 \quad \forall i \in \{0,...,n\} \,. \]
Hence $y \in W \cdot \sigma(h)^{-1}$.

Finally, we check that $V$ is $H$-stable.  Fix $y \in V$.
First, consider, for all $h,h' \in H$,
\[ hb_{f,h'} = \prod_{s \in S} h \sigma(h's) a_{f(s)} = \prod_{s \in S} \left( \sigma(hh's) a_{f(s)} + \left( h \sigma(h's) a_{f(s)} - \sigma(hh's) a_{f(s)} \right) \right) \,. \]
When we expand the product, one of the terms is $b_{f,hh'}$.
Observe that the definition of $K$ implies that for all $i \in \{0,...,n\}$,
\[ h\sigma(h's)a_i - \sigma(hh's)a_i \in \sigma(h'hs)\left(\sum^n_{j=0} A^{\circ\circ} \cdot a_j \right)= \sum^n_{j=0}A^{\circ\circ} \cdot \sigma(h'hs)a_j. \]
Thus, we can bound the norm of the remaining terms using the inequalities
\[ |\sigma(hh's) a_{f(s)}|_y \le \max_j |\sigma(hh's) a_j|_y \,, \]
\[ \left|h \sigma(h's) a_{f(s)} - \sigma(hh's) a_{f(s)} \right| \prec \max_j |\sigma(hh's) a_j|_y \,; \]
here, $\alpha \prec \beta$ means that either
$\alpha < \beta$ or $\alpha = \beta = 0$.
We then find that
\[ |hb_{f,h'} - b_{f,hh'}|_y \prec \prod_{s \in S} \max_j |\sigma(hh's) a_j|_y = \max_{f'} |b_{f',hh'}|_y \,. \]
Since $y \in V$, the $|b_{f',h''}|_y$, where $h''$ ranges over $H$, are not all zero, so
\[ |hb_{f,h'}-b_{f,hh'}|_y < \max_{f',h''} |b_{f',h''}|_y \,. \]
Finally, for all $h \in H$, $f \in \mathscr{S}$, and $i \in \{1,...,m\}$,
a similar argument shows that
\[ |he_{f,i}-e_{f,i}|_y^{1/i} < \max_{f',h'} |b_{f',h'}|_y = \max_{f',i'} |e_{f',i'}|_y^{1/i'} = |e_{\phi,r}|_y^{1/r}  \,. \]
Then the strong triangle inequality implies that for any $h \in H$,
$f \in \mathscr{S}$, and $i \in \{1,...,m\}$,
\[ |he_{f,i}|_y^{1/i} \le |e_{\phi,r}|^{1/r} = |he_{\phi,r}|^{1/r} \ne 0 \,. \]
Hence $y \cdot h \in V$.  So $V$ is $H$-stable.
\end{proof}

\subsection{Proof of Theorem \ref{bound on higher direct images}}\label{sec: proof of main thm}

In this subsection, we prove Theorem~\ref{bound on higher direct images}. To simplify notation, we write $H \defeq K \cap N_{0}$, $\pi \defeq \pi_{\HT/K \cap N_0}$, $\cF^{i} \defeq R^{i}\pi_{\ast}j_{!}(\cO_{\cX}^{+}/p)^{a}$. Let $w\in W_{I}\backslash W /W_{I}$; write $d(w)=\dim \Fl_{G,\mu}^{w}$ and pick $x\in |\Fl_{G,\mu}^{w}|/H$. We need to show that $\cF^{i}_{x}=0$ for $i > d-d(w)$. By Proposition \ref{topbasechange} and Lemma \ref{shriekbasechange},
\[
 \cF_{x}^{i} \cong H^{i}_{\et}(\pi^{-1}(x), j_{!}(\cO_{\cX}^{+}/p)^{a}). 
 \]
Recall from Subsection~\ref{structure sheaves} that a point of the topological space underlying a locally spatial diamond is said to be of rank one if it has no proper generalizations.

\begin{lemma}\label{reduction to rank one}
It suffices to prove Theorem \ref{bound on higher direct images} for points $x\in |\Fl_{G,\mu}^{w}|/H$ which have no proper generalizations.
\end{lemma}

\begin{proof}
Let $x\in|\Fl_{G,\mu}^{w}|/H$ be arbitrary. It has a unique maximal generalization $y\in |\Fl_{G,\mu}^{w}|/H$, which may be obtained by lifting $x$ to $\wt{x}\in |\Fl_{G,\mu}^{w}|$, letting $\wt{y}\in |\Fl_{G,\mu}^{w}|$ be the unique rank one generalization of $\wt{x}$ and setting $y=\wt{y}H$ (uniqueness of $y$ follows from the fact that $|\Fl_{G,\mu}^{w}| \to |\Fl_{G,\mu}^{w}|/H$ is generalizing). We then have a natural quasicompact open immersion of spatial diamonds $\pi^{-1}(y) \to \pi^{-1}(x)$ which induces maps
\[
 H^{i}_{\et}(\pi^{-1}(x), j_{!}(\cO_{\cX}^{+}/p)^{a}) \to H^{i}_{\et}(\pi^{-1}(y), j_{!}(\cO_{\cX}^{+}/p)^{a})
 \]
for all $i$. Since $\pi^{-1}(y)$ contains all rank one points of $\pi^{-1}(x)$, these maps are isomorphisms by Lemma \ref{rankonepointsisom}, and this proves the lemma.
\end{proof}

\noindent From now on, we assume that $x\in |\Fl_{G,\mu}^{w}|/H$ has no proper generalizations. Let $\wt{x}\in |\Fl_{G,\mu}^{w}|$ be any lift of $x$; this is a rank one point. By Lemma \ref{finding a J}, we choose a $w_1 \in \delta^{-1}(ww_{0})$ such that $\wt{x}N_{0} \sub |\cU^{w_1}|$; and we let $x_{1}\in |\Fl_{G,\mu}^{w}|/H_{1}$ denote the image of $\wt{x}$, where $H_{1} \defeq H \cap N_{0,w_1}$. Note that $x_1$ has no proper generalizations in $|\Fl_{G,\mu}|/H_1$, so the orbit $x_1(H/H_1)$ is profinite by Lemma \ref{homeomorphism onto orbit}. Write $\pi_1 \defeq \pi_{\HT/H_1}$ for the morphism of sites $(\ocX_{H_1})_{\et} \to |\Fl_{G,\mu}|/H_1$, and put $V_1 \defeq \pi_{1}^{-1}(x_1(H/H_1))$. $V_1$ is naturally a spatial diamond: The underlying map $|\pi_1| \colon |\ocX_{H_1}| \to |\Fl_{G,\mu}|/H_1$ is spectral and $x_1(H/H_1)$ can be written as an intersection of quasicompact opens, so $V_1$ is an inverse limit of open spatial subdiamonds of $\cX_{H_1}^{\ast}$, hence spatial by Lemma \ref{inverselimit}.

\begin{lemma}\label{cover pi^-1(x)}
The natural map $V_1 \to \pi^{-1}(x)$ is a $v$-cover, and the restriction $V_1\cap \cX_{H_1} \to \pi^{-1}(x) \cap \cX_{H}$ is an $H/H_1$-torsor. 
\end{lemma}

\begin{proof}
$V_1$ is equal to $\ocX_{H_1}\times_{\ocX_{H}}\pi^{-1}(x)$ and the natural map to $\pi^{-1}(x)$ is the projection onto the second factor. Now $\ocX_{H_1} \to \ocX_{H}$ is a (in fact quasi-pro-\'etale) map of spatial diamonds which is surjective on topological spaces by Lemma \ref{ad hoc exhibiting a diamond as a quotient}, hence a v-cover by \cite[Lemma 12.11]{diamonds}, so the pullback $V_1 \to \pi^{-1}(x)$ is a v-cover. Moreover, $\cX_{H_1} \to \cX_{H}$ is an $H/H_1$-torsor by Lemma \ref{torsor}, so the pullback $V_1\cap \cX_{H_1} \to \pi^{-1}(x) \cap \cX_{H}$ is an $H/H_1$-torsor.
\end{proof}

By this lemma, Theorem \ref{hochschildserre} gives us a spectral sequence
\[
 E_{2}^{rs}=H^{r}_{cts}(H/H_1,H^{s}_{\et}(V_1,j_{!}(\cO^{+}_{\cX}/p)^{a})) \implies H^{r+s}_{\et}(\pi^{-1}(x), j_{!}(\cO^{+}_{\cX}/p)^{a}).
 \]
Since $H/H_1 \sub N_{0}/N_{0,w_1}\cong \Zp^{d-d(w)}$ as $\Zp$-modules (using Lemma \ref{ranks}) and the index is finite, we must have $H/H_1 \cong \Zp^{d-d(w)}$ and hence continuous cohomology for $H/H_1$ has cohomological dimension $d-d(w)$ (this is presumably well known---it follows e.g.\ from \cite[Proposition 5.2.7]{nsw} and the fact that the Koszul complex of $\Zp[[\Zp^{d-d(w)}]]$ resolves the trivial module $\Zp$). Therefore, to prove that  $H^{i}_{\et}(\pi^{-1}(x), j_{!}(\cO^{+}_{\cX}/p)^{a})=0$ for $i>d-d(w)$ it suffices to show that $H^{s}_{\et}(V_1,j_{!}(\cO^{+}_{\cX}/p)^{a})=0$ for all $s>0$. To do this, we start by considering the natural map
\[
 \pi_{V_1} \colon (V_1)_{\et} \to x_1(H/H_1).
 \]
Since $x_1(H/H_1)$ is profinite, higher cohomology vanishes and we have
\[
 H^{s}_{\et}(V_1,j_{!}(\cO^{+}_{\cX}/p)^{a}) = H^{0}(x_1(H/H_1), R^{s}\pi_{V_1,\ast}j_{!}(\cO_{\cX}^{+}/p)^{a}).
 \]
Thus, we need to show that $R^{s}\pi_{V_1,\ast}j_{!}(\cO_{\cX}^{+}/p)^{a}=0$ for $s>0$, which we can check on stalks. If $y\in x_1(H/H_1)$, then
\[
 R^{s}\pi_{V_1,\ast}j_{!}(\cO_{\cX}^{+}/p)^{a}_{y}\cong H^{s}_{\et}(\pi_1^{-1}(y), j_{!}(\cO_{\cX}^{+}/p)^{a})
 \]
by Proposition \ref{topbasechange}, so we need to show that the right-hand side vanishes for $s>0$. As $\wt{x}$, and hence $x_1$, was chosen arbitrarily, without loss of generality $y=x_1$.

\medskip

We now apply Proposition \ref{rational neighborhood}. Consider the orbit $\wt{x}H_1 \sub |\cU^{w_1}|$. By Theorem \ref{general perfectoid stratum}, we may choose $k$ large enough such that $|\ocX_{\mbf{1}}(\epsilon)_{a}|\gamma^{k}w_1^{-1} \supseteq \pi^{-1}_\HT(\wt{x}H_1)$ (for some sufficiently small $\epsilon >0$, which we fix). We then choose a large enough $k^{\prime}$ such that $|\ocX_{\mbf{1}}(\epsilon)_{a}|\gamma^{k}w_1^{-1} \sub \pi^{-1}_{\HT}(\cU^{w_1}(k^{\prime}))$, which we may do by Theorem \ref{general perfectoid stratum} and Lemma \ref{standard neighborhoods are stable}. Note that $\cU^{w_1}(k^{\prime})$ is affinoid and $N_{0,w_1}$-stable by Lemma \ref{standard neighborhoods are stable}, hence $H_1$-stable. By Proposition \ref{rational neighborhood}, we can find a collection $U_{t}$, $t\in T$,
of $H_1$-stable opens of $\cU^{w_1}(k^{\prime})$ which are rational subsets and
\[
 \bigcap_{t\in T}U_{t} = \wt{x}H_1;
 \]
we may and will also assume that $\pi_{\HT}^{-1}(U_{t}) \sub |\ocX_{\mbf{1}}(\epsilon)_{a}|\gamma^{k}w_1^{-1}$ for all $t\in T$. It follows that $\bigcap_{t\in T}|U_{t}|/H_1 = x_1$ and hence that
\[
 H^{s}_{\et}(\pi_1^{-1}(x_1), j_{!}(\cO_{\cX}^{+}/p)^{a}) = \varinjlim_{t\in T} H^{s}_{\et}(\pi_1^{-1}(|U_{t}|/H_1), j_{!}(\cO_{\cX}^{+}/p)^{a}), 
 \]
so it suffices to show that $H^{s}_{\et}(\pi_1^{-1}(|U_{t}|/H_1), j_{!}(\cO_{\cX}^{+}/p)^{a})=0$ for $s>0$ and $t\in T$.

\begin{prop}\label{cohomology vanishing for rational neighborhoods}
$H^{s}_{\et}(\pi_1^{-1}(|U_{t}|/H_1), j_{!}(\cO_{\cX}^{+}/p)^{a})=0$ for $s>0$ and $t\in T$.
\end{prop}

\begin{proof}
In this proof we will (as we have occasionally done elsewhere) conflate an open subset of the topological space of a perfectoid space or diamond with the corresponding open subspace or open subfunctor; we hope that this will not cause any trouble. With this in mind, note that we can write 
\[
 |\ocX_{\mbf{1}}(\epsilon)_{a}|\gamma^{k}w_1^{-1} = \invlim_{m\geq 0} \left( |\ocX_{\mbf{1}}(\epsilon)_{a}|\gamma^{k}w_1^{-1}/(H_1 \cap \Gamma(p^{m})) \right)
\]
and that all spaces appearing are affinoid perfectoid with strongly Zariski closed boundary, by Theorem \ref{general perfectoid stratum}. We get a corresponding inverse limit description
\[
 |\pi^{-1}_{\HT}(U_{t})| = \invlim_{m} \left( |\pi_{\HT}^{-1}(U_{t})|/(H_1\cap \Gamma(p^{m})) \right).
\]
Note that all spaces here are quasicompact perfectoid spaces and that the bottom space $|\pi_{\HT}^{-1}(U_{t})|/H_1$ is equal to $\pi_1^{-1}(|U_{t}|/H_1)$. By our choices of $k$ and $k^{\prime}$ and the construction of $U_{t}$, $\pi_{\HT}^{-1}(U_{t})$ is a rational subset of $|\ocX_{\mbf{1}}(\epsilon)_{a}|\gamma^{k}w_1^{-1}$. As a rational subset in an inverse limit must come from a rational subset at some finite level, there exists an $m$ such that $|\pi_{\HT}^{-1}(U_{t})|/(H_1\cap \Gamma(p^{m}))$ is a rational subset of $|\ocX_{\mbf{1}}(\epsilon)_{a}|\gamma^{k}w_1^{-1}/(H_1 \cap \Gamma(p^{m}))$, and hence affinoid perfectoid. Fix such an $m$ and put $V_{m} \defeq |\pi_{\HT}^{-1}(U_{t})|/(H_1\cap \Gamma(p^{m}))$ and $G_{m} \defeq H_1/(H_1 \cap \Gamma(p^{m}))$. Note that $G_m$ is a finite group. The map
\[
 V_{m} \to \pi_1^{-1}(|U_{t}|/H_1)
 \]
is a $G_{m}$-equivariant map of perfectoid spaces and is a $G_{m}$-torsor away from the boundary by Lemma \ref{torsor}. By \cite[Theorem 1.2, Theorem 3.5]{hansen}, the quotient $V_{m}/G_{m}$ exists (in the category of adic spaces) and is affinoid perfectoid. We therefore get an induced map
\[
 V_{m}/G_{m} \to \pi_1^{-1}(|U_{t}|/H_1) 
 \]
which is an isomorphism away from the boundary (since we had a $G_{m}$-torsor away from the boundary). By Lemma \ref{shriekbasechange}, we have
\[
 H^{s}_{\et}(\pi_1^{-1}(|U_{t}|/H_1), j_{!}(\cO_{\cX}^{+}/p)^{a}) = H^{s}_{\et}(V_{m}/G_{m}, j_{!}(\cO^{+}_{\cX}/p)^{a})
\]
for all $s\geq 0$. Moreover, the boundary of $V_{m}/G_{m}$ is strongly Zariski closed since it is the pullback of the boundary on $|\cX_{\mbf{1}}^{\ast}(\epsilon)_{a}|\gamma^{k}w_1^{-1}/H_1$. By Proposition \ref{strongly Zariski closed vanishing} $H^{s}_{\et}(V_{m}/G_{m}, j_{!}(\cO^{+}_{\cX}/p)^{a})=0$, which finishes the proof.
\end{proof}

\noindent This finishes the proof of Theorem \ref{bound on higher direct images}, and hence also the proof of Theorem \ref{main thm}.

\section{Eliminating the nilpotent ideal}\label{Borel-Serre}

We begin by setting up notation for this section, which will be \emph{different} from that used in previous section. The main difference is that we will use symplectic/unitary group in this section instead of the similitude groups used in previous sections.

\medskip

As before, $F$ will denote a CM field with totally real subfield $F^+$ and complex conjugation $c$. We will let $n\geq 2$ be an integer, and we will fix a rational prime $p$ throughout which is assumed to be totally split in $F$. As in \S \ref{subsect:unitary_group}, we let $\Psi_n$ denote the $n\times n$ matrix with $1$'s along the anti-diagonal and $0$'s elsewhere, and we set
\[
J_n = \begin{pmatrix} 0 & \Psi_n \\ -\Psi_n & 0 \end{pmatrix}\in \GL_{2n}(\Z).
\]
$J_n$ defines a perfect pairing $(x,y) = x^t J_n y^c$ on $L \defeq \cO_F^{2n}$, which is alternating if $F$ is totally real and skew-Hermitian if $F$ is imaginary CM. We define the symplectic/unitary group $G_0$ over $\cO_{F^+}$ by
\[
G_0(R) = \{ g \in \Aut_{\cO_F\otimes_{\cO_{F^+}}R}( L \tensor_{\cO_{F^+}} R)  \mid g^t J_n g^c = J_n \},
\]
for $\cO_{F^+}$-algebras $R$. The condition defining the group is equivalent to preserving the pairing $(-,-)$. We then define 
\[
G \defeq \mathrm{Res}_{\cO_{F^+}/\Z}G_0.
\]
Over primes $\ell$ unramified in $F$, $G_{\Z_{(\ell)}}$ could have equivalently been defined by the condition $\psi(gx,gy)= \psi(x,y)$ for all $x,y\in L$, where $\psi$ is the alternating form in \S \ref{subsect:unitary_group}. Similarly, $G(\Z_\ell)$ for \emph{all} $\ell$ could have equivalently been defined by the condition $\psi(gx,gy)= \psi(x,y)$. We define $P_0 \sub G_0$ to be the stabilizer of $\cO_F^n \oplus 0 \sub L$, and $M_0 \sub P_0$ to be the subgroup which additionally stabilizes $0 \oplus \cO_F^n$. One sees easily that $M_0 \cong \mathrm{Res}_{\cO_F/\cO_{F^+}}\GL_n$; see \cite[Lemma 5.1(2)]{newton-thorne} for the imaginary CM case, the same proof works in the totally real case. We then define $P \defeq \mathrm{Res}_{\cO_{F^+}/\Z}P_0$ and $M \defeq \mathrm{Res}_{\cO_{F^+}/\Z}M_0 \cong \mathrm{Res}_{\cO_F/\Z}\GL_n$; these are subgroups of $G$. Whenever $\ell$ is unramified in $F$, $P_{Z_{(\ell)}}$ is a parabolic subgroup of $G_{\Z_{(\ell)}}$, and $M_{\Z_{(\ell)}}$ is a Levi subgroup of $P_{\Z_{(\ell)}}$. We finish by defining some compact open subgroups in $G(\Qp)$. Set
\begin{align*}
  \Gamma_0(p^m) & \defeq \left\{g\in G(\Zp) \mid (g \bmod p^m) \in P(\Z/p^m) \right\}; \\
  \Gamma_1(p^m) & \defeq \left\{g \in G(\Zp) \mid (g \bmod p^m) \in N(\Z/p^m) \right\}; \\
  \Gamma(p^m) & \defeq \left\{g \in G(\Zp) \mid g \equiv I_{2n} \bmod p^m \right\}.
\end{align*}
Note that, in essence, most of what we have done here is make the ``same'' definitions for symplectic/unitary groups as we did for the corresponding similitude groups earlier. We remark that $G(\R)$ is connected. We finish this preamble by defining some principal congruence subgroups for $M$. We set
\[
K_{M,m,p} \defeq \mathrm{Ker}(M(\Zp) \to M(\Z/p^m))
\]
for all $m\geq 0$, and for a fixed $K_M^p \sub M(\A_f^p)$, which we will allow to be defined by the context in what follows, we set
\[
K_{M,m} \defeq K_M^p K_{M,m,p}.
\]

\subsection{Introduction} \label{nilpotent intro}

To state our results, we first discuss the definitions of locally symmetric spaces that we will use in this section. If $\mathrm{G}$ is a general connected linear algebraic group over $\Q$, 
we consider a space of type $S-\Q$ for $\mathrm{G}$,
in the sense of~\cite[\S 2]{borel-serre}. This is a pair consisting of a homogeneous
space for $\mathrm{G}(\R)$ and a family of Levi subgroups
of $\mathrm{G}_{\R}$ satisfying certain conditions; we will suppress the Levi subgroups from the notation. The underlying homogeneous space of a space of type $S-\Q$ only depends on $\mathrm{G}$ (up to isomorphism of homogeneous spaces); this follows from \cite[Lem. 2.1]{borel-serre}. Moreover, this homogeneous space is a symmetric space for $\mathrm{G}$ and we will simply refer to it as \emph{the} symmetric space for $\mathrm{G}$, and denote it by $X^{\mathrm{G}}$.  
Whenever $K \sub \mathrm{G}(\A_f)$ is a compact open subgroup, we set 
\[
X^\mathrm{G}_{K}\defeq \mathrm{G}(\Q)\backslash \left(X^\mathrm{G}\times \mathrm{G}(\A_f)\right)/K;
\] 
this is a locally symmetric space (and in particular a Riemannian manifold) when $K$ is neat. In the cases we will consider, we will have $\mathrm{G}=\mathrm{Res}_{F^+/\Q}\mathrm{G}_0$ and $\mathrm{G}_0$ will have a natural model over $\cO_{F^+}$; for the purpose of this discussion let us denote it by $\mathrm{G}_0$ as well. We will assume throughout this section that all compact open subgroups $K\sub \mathrm{G}(\A_f)$ appearing are decomposed as
\[
K = \prod_v K_v, \,\,\,\,\, K_v \sub \mathrm{G}_0(\cO_{F^+,v}),
\]
where $v$ runs over the finite places of $F^+$. This is mostly done to simplify the exposition. We say that $K \sub \mathrm{G}(\A_f)$ (decomposed as above) is \emph{small} (following \cite[Definition 5.6]{newton-thorne}) if there exists a rational prime $q\neq p$ such that $K_v \sub \mathrm{Ker}(\mathrm{G}(\cO_{F^+,v}) \to \mathrm{G}(\cO_{F^+}/q^\epsilon))$ for all $v\mid q$, where $\epsilon=1$ if $q\neq 2$ and $\epsilon =2 $ if $q=2$. Such subgroups will always be neat in all cases we consider. We will need this notion to apply the main result of \cite{scholze-galois}; for everything else neatness will be enough.

\medskip

Let us now return to the groups $G$ and $M$ defined in the beginning of this section. We will write $X^G$ for what, if one strictly followed the notation above, would write $X^{G_\Q}$, and similarly with levels and for $M$. To discuss the relation between the $X_K^G$ and the Shimura varieties considered in previous sections, let us write $(G^{\mathrm{old}},X^{\mathrm{old}})$ for the Shimura datum denoted by $(G_\Q,X)$ in \S \ref{subsubsect:Sh_datum}. We have $G_\Q \sub G^{\mathrm{old}}$ and one checks easily that $X^G$ is a connected component of $X^{\mathrm{old}}$. If $K^{\mathrm{old}} \sub G^{\mathrm{old}}(\A_f)$ is neat and $K = K^{\mathrm{old}} \cap G(\A_f)$, then one checks that the natural map
\[
X_K^G \to G^{\mathrm{old}}(\Q) \backslash X^{\mathrm{old}} \times G^{\mathrm{old}}(\A_f)/K^{\mathrm{old}}
\]
is injective, and hence the left hand side is a union of connected components of the right hand side. Set 
\[
d\defeq\frac{1}{2}\mathrm{dim}_{\R}(X^G) = \dim_{\C}(X^G).
\]
Since the injections above are compatible with changing $K^\mathrm{old}$, we obtain the following as a direct corollary to Theorem \ref{main thm}, along the lines of Corollary \ref{cor to main}.

\begin{thm}\label{main thm non-similitude} Let $r\in \Z_{\geq 1}$ and fix $K^p \sub G(\A_f^p)$ small. If $i>d$ and $K_{p,m} \sub \Gamma_1(p^m)$ is compact open for all $m$, then 
\[
 \varinjlim_{m} H^{i}_{c}(X_{K^p K_{p,m}}(\C),\Z/p^r)=0.
 \]
\end{thm}

\noindent In this section, we will only use the special case when  $K_{p,m}= \Gamma_1(p^m)$ for all $m$.

\medskip

Let us now turn our attention to the Levi $M\sub G$. Let $K_M\subseteq M(\A_f)$ be a small compact open subgroup. The space $X^M_{K_M}$ has dimension $d-1$ (as a real manifold; it has no complex structure in general). Let $\lambda$ be a Weyl orbit of weights of $M$. This determines an irreducible algebraic representation $\sigma_{\lambda}$ of $M$, which can
be defined over $\Q_p$ (as $M$ is split over $\Q_p$). There is a natural $M(\Z_p)$-stable 
lattice $\sigma^\circ_{\lambda}\subset \sigma_{\lambda}$, known as the dual Weyl module (see~\cite{jantzen} for more details). This gives rise in 
the usual way to a local system on $X^M_{K_M}$, which we denote by $\cV_{\lambda}$; for further discussion see \S \ref{subsec: loc sym space}. 

\medskip

Let $\iota\colon \overline{\Q}_p\toisom \C$ be an isomorphism. We let $S$ be a finite set of primes of $\Q$, containing $p$, all the primes which ramify in $F$, and all the primes where the level $K_M$ is not hyperspecial. We let $\mathbb{T}^S_{M}$ denote the abstract Hecke algebra for $M$ over $\Z_p$, defined as the product of spherical Hecke algebras at places away from $S$
\[
\mathbb{T}^S_{M}\defeq  \otimes_{l\not\in S,w\mid l} \mathbb{T}_{M,w}, \quad \mathbb{T}_{M,w} 
\defeq \Z_p[\GL_n(F_w)//\GL_n(\cO_{F,w})],
\]
where $w$ runs over the primes of $F$ above $l\not\in S$. 
For such a prime $w$, we let $q_w$ be the cardinality of its residue field and $q_{w}^{1/2}\in \overline{\Z}_p$ denote the inverse image of the positive 
square root of $q_{w}$ in $\C$ under $\iota$. The Satake transform gives a canonical isomorphism
\[
\mathbb{T}_{M,w}[q^{1/2}_{w}]\simeq \Z_p[q_{w}^{1/2}][X^{\pm 1}_1,\dots, X^{\pm 1}_{n}]^{S_{n}},
\]
with $S_n$ the symmetric group on $n$ elements.
For $i=1,\dots, n$, define $T_{i,w}$ to be $q_w^{i(n-i)/2}$ times the $i$-th elementary symmetric polynomial in $X_1,\dots,X_{n}$.

\medskip

The Hecke algebra $\mathbb{T}^S_{M}$ acts in the usual way on $H^*(X^M_{K_M}, \cV_\lambda)$ (see \S \ref{subsec: loc sym space} for further discussion). We consider its image 
\[
\mathbb{T}^S_{M}\left(K_M,\lambda\right)\defeq \mathrm{Im}\left(\mathbb{T}^S_{M}\to \mathrm{End}_{\Z_p}\left(H^*(X^M_{K_M}, \cV_\lambda)\right)\right).
\] 
Let $\m\subset \mathbb{T}^S_{M}\left(K_M,\lambda\right)$ be a maximal ideal. The following is~\cite[Cor.~5.4.3]{scholze-galois} 
(with slightly different normalizations, which are consistent with~\cite{newton-thorne}).

\begin{thm}\label{residual Galois rep} There exists a unique continuous semisimple Galois representation 
\[
\bar{\rho}_{\m}\colon \mathrm{Gal}(\overline{F}/F)\to \GL_n(\overline{\F}_p)
\]
such that, for every prime $w$ of $F$ above $l\not\in S$, the characteristic polynomial of $\bar{\rho}_{\m}(\Frob_w)$ is equal to the image of
\[
P_{M,w}(X) = X^{n} - T_{1,w}X^{n-1}+\dots + (-1)^iq_{w}^{i(i-1)/2}T_{i,w}X^{n-i} +\dots + q_{w}^{n(n-1)/2}T_{n,w}
\]
modulo $\m$. 
\end{thm}

\noindent From now on, we also impose the following \emph{non-Eisenstein} condition on $\m$. 

\begin{assumption}\label{non-Eisenstein} The Galois representation $\bar{\rho}_{\m}$ is absolutely irreducible. 
\end{assumption} 

The goal of this section is to prove a 
strengthening of~\cite[Thm.~5.4.4]{scholze-galois} and~\cite[Thm.~1.3]{newton-thorne} under our assumption on $p$. 
In order to state the main result, we need to introduce a derived Hecke algebra\footnote{This is in the sense of~\cite{newton-thorne}, who consider an enhancement of the
usual notion of Hecke algebra living in the derived category, rather than in the sense of 
Venkatesh, who considers additional ``derived'' Hecke operators.}, which we define as follows: 
\[
\mathbb{T}^S_{M}\left(K_M, \lambda\right)^{\mathrm{der}}\defeq \mathrm{Im}\left(\mathbb{T}^S_{M}\to \mathrm{End}_{D(\Z_p)}\left(R\Gamma(X^M_{K_M},\mathcal{V}_{\lambda})\right)\right).
\]
Note that we have a surjection 
\[
\mathbb{T}^S_{M}\left(K_M,\lambda\right)^{\mathrm{der}}\twoheadrightarrow \mathbb{T}^S_{M}\left(K_M,\lambda\right),
\]
with nilpotent\footnote{Nilpotency follows from \cite[Lemma 2.5]{khare-thorne}.} kernel; in particular this surjection induces a bijection of maximal ideals. Thus, we can take our non-Eisenstein maximal ideal Let $\m\subset \mathbb{T}^S_{M}\left(K_M,\lambda\right)$ and form the localization $\mathbb{T}^S_{M}\left(K_M,\lambda\right)^{\mathrm{der}}_\mf{m}$. The following is the main theorem of this section.

\begin{thm}\label{eliminating nilpotent ideal} Let $p$ be a rational prime which splits completely in $F$ and $\m$ a non-Eisenstein maximal ideal. 
Then there exists a unique continuous Galois 
representation 
\[
\rho_{\m}\colon \mathrm{Gal}(\overline{F}/F)\to \GL_n\left(\mathbb{T}^S_{M}\left(K_M,\lambda\right)^{\mathrm{der}}_{\m}\right)
\]
such that, for every prime $w$ of $F$ above $l\not\in S$, the characteristic polynomial of $\rho_{\m}(\Frob_w)$ is equal to 
\[
P_{M,w}(X) = X^{n} - T_{1,w}X^{n-1}+\dots + (-1)^iq_{w}^{i(i-1)/2}T_{i,w}X^{n-i} +\dots + q_{w}^{n(n-1)/2}T_{n,w}.
\]
\end{thm}

\begin{rem}
In \cite[Thm.~5.4.4]{scholze-galois}, a Galois representation valued in $\mathbb{T}^S_{M}\left(K_M,\lambda\right)/I$ is constructed, where $I$ is an ideal with $I^n =0$ for some (computable) $n$ independent of $\lambda$. In \cite[Thm.~1.3]{newton-thorne}, a Galois representation valued in $\mathbb{T}^S_{M}\left(K_M,\lambda\right)^\mathrm{der}/J$ with $J^4 =0$ is constructed (for $F$ imaginary CM). Theorem \ref{eliminating nilpotent ideal} thus refines these results by removing the nilpotent ideal; see Remark \ref{relevance of nilpotent ideal} for some further comments.
\end{rem}

\noindent The proof of Theorem \ref{eliminating nilpotent ideal} will take up the remainder of Section \ref{Borel-Serre}. We note that localizing $\mathbb{T}^S_{M}\left(K_M,\lambda\right)^{\mathrm{der}}\twoheadrightarrow \mathbb{T}^S_{M}\left(K_M,\lambda\right)$ at $\mf{m}$ gives a surjection $\mathbb{T}^S_{M}\left(K_M,\lambda\right)^{\mathrm{der}}_\mf{m}\twoheadrightarrow \mathbb{T}^S_{M}\left(K_M,\lambda\right)_\mf{m}$, so Theorem~\ref{eliminating nilpotent ideal} also gives rise to a Galois representation valued in $\mathbb{T}^S_{M}\left(K_M,\lambda\right)_\mf{m}$.

\medskip

Let us now discuss some notation that will be used throughout this section. First, we define $\Lambda \defeq \Z /p^r$ throughout this section (the choice of $r\geq 1$ being arbitrary but fixed). We will also need some notations for various categories that will arise. Whenever $\mathrm{G}$ is a group and $R$ is a commutative ring, $\mathrm{Mod}(\mathrm{G},R)$ will denote the category of (left) $\mathrm{G}$-representations on $R$-modules. If $\mathrm{G}$ is profinite, $\mathrm{Mod}_{sm}(\mathrm{G},R)$ will denote the full subcategory of smooth $\mathrm{G}$-representations. If $S$ is a (possibly noncommutative) ring, then $\mathrm{Mod}(S)$ will denote the category of left $S$-modules. If $S$ is an $R$-algebra, then we will write $\mathrm{Mod}(S \times \mathrm{G},R)$ for the category of left $S \otimes_R R[\mathrm{G}]$-modules. Moving on, if an (abstract) group $\mathrm{G}$ acts (from the right) on a topological space $X$, then $\mathrm{Sh}_{\mathrm{G}}(X)$ will denote the category of $\mathrm{G}$-equivariant sheaves on $X$ (see \cite[\S 2.3]{newton-thorne}). When $R=\Z$, we will typically omit it from the notation.

\medskip

All of these are abelian categories, and their homotopy and bounded below derived categories will be denoted using self-explanatory notations involving the letters $K$ and $D$, respectively. Please do note that we will write $D$ instead of $D^+$ for bounded below derived categories; we apologize for this choice, which was made to avoid making already heavy notation even heavier.
Our convention for shifts of complexes is that $(C^\bullet[1])^n = C^{n+1}$. We will freely use the fact morphisms between complexes of injective objects in the derived category may be computed in the homotopy category. Also, group cohomology will refer to \emph{discrete} group cohomology unless otherwise stated.

\medskip

We end this introduction with a brief overview of this section and of our argument. Subsections \ref{subsec: loc sym space}-\ref{subsec: ordinary parts} are preliminary. In \S \ref{subsec: loc sym space} we recall some material on locally symmetric spaces, in particular the method to compute Hecke actions introduced in \cite{newton-thorne}, and in \S \ref{comm alg} we collect various homological algebra results that we will need. In \S \ref{subsec: ordinary parts} we set up some theory of ordinary parts (Hida theory) in the derived setting, essentially following \cite{khare-thorne}. We then begin our argument to prove Theorem \ref{eliminating nilpotent ideal}. Like in \cite{scholze-galois,newton-thorne}, our strategy consists of three main steps. 
\begin{enumerate}
\item The first, carried out in \S \ref{determinants}, is to construct Galois representations for the cohomology of the $X^G_K$; our work here is a slight refinement of \cite{scholze-galois,newton-thorne}. 
\item The second step is to relate cohomology of $X^G_K$ to cohomology of $X_{K_M}^M$ by studying the boundary of the Borel--Serre compactification. This is done in \S \ref{subsec: boundary} and represents the main innovation of our work in this section. The two key new ingredients are Theorem \ref{main thm non-similitude} and the use of the ordinary parts functor from $G$ to $M$, as studied in \S \ref{subsec: ordinary parts}. The final result of this analysis, Theorem \ref{direct summand}, should also be relevant to the study of local-global compatibility at $\ell = p$ for the representations $\rho_\m$ above. 
\item Finally, \S \ref{determinants 2} finishes the proof of Theorem \ref{eliminating nilpotent ideal}, following the argument in \cite{newton-thorne}; the key point here is to extract an $n$-dimensional determinant from a $2n$- or $(2n+1)$-dimensional one using character twists.
\end{enumerate}

\subsection{Locally symmetric spaces}\label{subsec: loc sym space} In this subsection we recall some generalities on locally symmetric spaces that we will need later. We go back to the situation when $\mathrm{G}$ be a connected linear algebraic group over $\Q$. Let us briefly recall the Borel--Serre compactification (see \cite{borel-serre}; some useful summaries are \cite[\S 3.1]{newton-thorne} and \cite[\S 2.1]{accghlnstt}). The space $X^\mathrm{G}$ admits a partial compactification $\ol{X}^\mathrm{G}$ \cite[\S 7.1]{borel-serre} with an action of $G(\Q)$, and every torsion-free arithmetic subgroup of $G(\Q)$ acts freely on $\ol{X}^\mathrm{G}$ \cite[\S 9.5]{borel-serre}. We define
\[
\ol{X}^\mathrm{G}_{K}\defeq \mathrm{G}(\Q)\backslash \left(\ol{X}^\mathrm{G}\times \mathrm{G}(\A_f)\right)/K;
\]
this is a compact differentiable manifold with corners, whose interior is $X^\mathrm{G}_{K}$ (see \cite[\S 7.1, Theorem 9.3, \S 9.5]{borel-serre}). As a consequence\footnote{It is a standard consequence of the (global) collar neighbourhood theorem \cite[Theorem 2]{brown} that the inclusion of the interior inside a topological manifold with boundary is a homotopy equivalence. To apply this to $X^\mathrm{G}_{K} \hookrightarrow \ol{X}^\mathrm{G}_{K}$, note that ``corners'' and ``boundary'' are the same thing for topological manifolds.}, the inclusion $X^\mathrm{G}_{K} \hookrightarrow \ol{X}^\mathrm{G}_{K}$ is a homotopy equivalence. We will write $\partial \ol{X}^\mathrm{G}_{K} \defeq \ol{X}^\mathrm{G}_{K} \setminus X^\mathrm{G}_{K}$ for the boundary.

\medskip

We now recall some material on local systems on $X_K^\mathrm{G}$ and $\ol{X}_K^\mathrm{G}$, and their cohomology, following \cite{newton-thorne}. Consider the space
\[ 
\ol{\mf{X}}^\mathrm{G} \defeq \mathrm{G}(\Q)\backslash \left(\ol{X}^\mathrm{G}\times \mathrm{G}(\A_f)\right),
\]
where $\mathrm{G}(\A_f)$ is given the \emph{discrete} topology prior to taking the quotient. Since $\mathrm{G}(\Q) \times K$ acts freely on $\ol{X}^\mathrm{G}\times \mathrm{G}(\A_f)$,
$K$ acts freely on $\ol{\mf{X}}^\mathrm{G}$. We will also use the subspace
\[ 
\mf{X}^\mathrm{G} \defeq \mathrm{G}(\Q)\backslash \left( X^\mathrm{G}\times \mathrm{G}(\A_f)\right),
\]
which $K$ acts freely on as well. A direct consequence of the freeness of these actions is the following lemma.

\begin{lemma}\label{free action} Let $K, K'\sub \mathrm{G}(\A_f)$ be neat compact open subgroups with $K'\sub K$ normal. Then the finite group $K/K'$ acts freely on $\ol{X}^{\mathrm{G}}_{K'}$ with quotient $\ol{X}^{\mathrm{G}}_{K}$.   
\end{lemma}

We now discuss another consequence of the fact that $K$ acts freely on $\ol{\mf{X}}^\mathrm{G}$. Any $K$-equivariant sheaf $\cF$ on $\ol{\mf{X}}^\mathrm{G}$ descends to a sheaf on $\ol{X}^\mathrm{G}_{K}$, which we will also denote by $\cF$ by abuse of notation. In practice, our $K$-equivariant sheaves will come by restriction from $\mathrm{G}(\A_f)$-equivariant sheaves, and this gives rise to Hecke actions on cohomology. We recall how these Hecke operators may be defined and computed from \cite[\S 2.3]{newton-thorne}, to which we refer for more details. There is a diagram
  \[
    \xymatrix@C+1pc{ \mathrm{Sh}_{\mathrm{G}(\A_f)}(\ol{\mf{X}}^\mathrm{G}) \ar[r]^-{\Gamma(\ol{\mf{X}}^\mathrm{G},-)} \ar[d]^-{\forget} & \mathrm{Mod}(\mathrm{G}(\A_f)) \ar[r]^-{\Gamma(K,-)} & \mathrm{Mod}(\cH(\mathrm{G}(\A_f),K)) \ar[d]^{\forget} \\ \mathrm{Sh}_{K}(\ol{\mf{X}}^\mathrm{G}) \ar[r]^-{\descent} & \mathrm{Sh}(\ol{X}^{\mathrm{G}}_K) \ar[r]^-{\Gamma(\ol{X}^{\mathrm{G}}_K,-)} &  \mathrm{Mod}(\Z) }
    \]
of categories and functors, commutative up to natural isomorphism. Here $\cH(\mathrm{G}(\A_f),K)$ denotes the Hecke algebra of $K$-biinvariant compactly supported functions $\mathrm{G}(\A_f) \to \Z$. The commutative diagram may be derived, and for a $\mathrm{G}(\A_f)$-equivariant sheaf $\cF$ on $\ol{\mf{X}}^\mathrm{G}$, this gives a canonical homomorphism
\[
\cH(\mathrm{G}(\A_f),K) \to \mathrm{End}_{D(\Z)}(R\Gamma(\ol{X}^{\mathrm{G}}_K,\cF))
\]
which is computed from the object $R\Gamma(K, R\Gamma(\ol{\mf{X}}^\mathrm{G},\cF)) \in D(\cH(\mathrm{G}(\A_f),K))$. This way of defining Hecke operators agrees with the traditional one using correspondences; for all this see \cite[Proposition 2.18]{newton-thorne} and the discussion following it. We remark that there are obvious versions of the above when $\Z$ is replaced by an arbitrary commutative ring $R$. Also, if $K^\prime \sub K$ is an open normal subgroup, one may descend a $K$-equivariant sheaf $\cF$ on $\ol{\mf{X}}^\mathrm{G}$ to a $K/K^\prime$-equivariant sheaf on $\ol{X}^\mathrm{G}_{K^\prime}$, and an obvious version of the above yields a Hecke action
\[
\cH(\mathrm{G}(\A_f^S),K^S) \to \mathrm{End}_{D(K/K^\prime)}(R\Gamma(\ol{X}^{\mathrm{G}}_{K^\prime},\cF)),
\]
where $S$ is a finite set of places such that $K^S=(K^\prime)^S$. Finally, also note that the entire discussion of equivariant sheaves and Hecke actions apply equally well to the (locally) symmetric spaces themselves and not just their Borel--Serre compactifications. 
\medskip

We have inclusions $\mf{X}^\mathrm{G} \hookrightarrow \ol{\mf{X}}^\mathrm{G}$ and $X^\mathrm{G}_K \hookrightarrow \ol{X}^\mathrm{G}_K$, both of which we will denote by $j$. We will be interested in two types of $\mathrm{G}(\A_f)$-equivariant sheaves on $\ol{\mf{X}}^\mathrm{G}$. The first are local systems; any right $\mathrm{G}(\A_f)$-module gives rise to a $\mathrm{G}(\A_f)$-equivariant sheaf on $\ol{\mf{X}}^\mathrm{G}$. When $\cF$ is a local system on $\ol{X}^{\mathrm{G}}_K$ obtained this way, pullback along the homotopy equivalence $j:X^{\mathrm{G}}_{K}\hookrightarrow \ol{X}^{\mathrm{G}}_{K}$ induces a Hecke-equivariant isomorphism between $R\Gamma(\ol{X}^{\mathrm{G}}_K, \cF)$ and $R\Gamma(X^{\mathrm{G}}_K, j^*\cF)$. 
The second type are extensions by zero of $\mathrm{G}(\A_f)$-equivariant local systems on $\mf{X}^\mathrm{G}$; these will be used to compute Hecke actions on compactly supported cohomology on the $X^\mathrm{G}_K$. They give exact functors $j_! : \mathrm{Sh}_{\mathrm{G}(\A_f)}(\mf{X}^\mathrm{G}) \to \mathrm{Sh}_{\mathrm{G}(\A_f)}(\ol{\mf{X}}^\mathrm{G})$ and $j_! : \mathrm{Sh}(X^\mathrm{G}_K) \to \mathrm{Sh}(\ol{X}^\mathrm{G}_K)$ such that the diagram
\[
    \xymatrix{ \mathrm{Sh}_{\mathrm{G}(\A_f)}(\mf{X}^\mathrm{G}) \ar[r]^-{j_!} \ar[d]^{\descent} & \mathrm{Sh}_{\mathrm{G}(\A_f)}(\ol{\mf{X}}^\mathrm{G}) \ar[d]^{\descent} \\ \mathrm{Sh}(X^{\mathrm{G}}_K) \ar[r]^-{j_!} & \mathrm{Sh}(\ol{X}^{\mathrm{G}}_K) }
\]
commutes up to natural isomorphism. If $\cF$ is $\mathrm{G}(\A_f)$-equivariant sheaf on $\mf{X}^\mathrm{G}$ with descent $\cF$ to $X_K^\mathrm{G}$, this gives us a Hecke action 
\[
\cH(\mathrm{G}(\A_f),K) \to \mathrm{End}_{D(\Z)}(R\Gamma_c(X^{\mathrm{G}}_K,\cF))
\]
which is computed from the object $R\Gamma(K, R\Gamma(\ol{\mf{X}}^\mathrm{G},j_!\cF)) \in D(\cH(\mathrm{G}(\A_f),K))$. Once again this action agrees with the traditional one defined using correspondences, and as above we have obvious versions for sheaves of $R$-modules and in the $K/K^\prime$-equivariant setting when $K^\prime \sub K$ is open and normal. 

\medskip

Let us record a few results on computations of Hecke actions and changing levels.

\begin{prop}\label{finite to finite basic}
Let $S$ be a finite set of places and let $K=K^S K_S \sub \mathrm{G}(\A_f)$ be a neat compact open subgroup. Let $K^\prime = K^S K_S^\prime$ be another compact open subgroup with $K^\prime_S \sub K_S$ normal. Let $\cF \in D(\mathrm{Sh}_{\mathrm{G}(\A_f)}(\ol{\mf{X}}^\mathrm{G})$. Then there is a natural isomorphism
\[ 
R\Gamma(K/K^\prime , R\Gamma(\ol{X}^\mathrm{G}_{K^\prime},\cF)) \cong R\Gamma(\ol{X}^\mathrm{G}_K,\cF) 
\]
which is equivariant for the action of $\cH(\mathrm{G}(\A_f^S),K^S)$ on both sides.
\end{prop}

\begin{proof}
By the formalism above, $R\Gamma(\ol{X}^\mathrm{G}_K,\cF)$ with its $\cH(\mathrm{G}(\A_f^S),K^S)$-action is computed by the object $R\Gamma(K, R\Gamma(\ol{\mf{X}}^\mathrm{G},\cF)) \in D(\cH(\mathrm{G}(\A_f^S),K^S))$, and $R\Gamma(\ol{X}^\mathrm{G}_{K^\prime},\cF)$ with its actions of $\cH(\mathrm{G}(\A_f^S),K^S)$ and $K/K^\prime$ is computed by the object $R\Gamma(K^\prime, R\Gamma(\ol{\mf{X}}^\mathrm{G},\cF)) \in D(\cH(\mathrm{G}(\A_f^S),K^S)\times K/K^\prime)$. To conclude, we use the formalism, noting that we have a (natural) isomorphism
\[
R\Gamma(K, R\Gamma(\ol{\mf{X}}^\mathrm{G},\cF)) \cong R\Gamma(K/K^\prime, R\Gamma(K^\prime, R\Gamma(\ol{\mf{X}}^\mathrm{G},\cF))) 
\]
in $D(\cH(\mathrm{G}(\A_f^S),K^S))$ and that
\[
    \xymatrix@C+2pc{ D(\cH(\mathrm{G}(\A_f^S),K^S)\times K/K^\prime) \ar[r]^-{R\Gamma(K/K^\prime,-)} \ar[d]^{\forget} & D(\cH(\mathrm{G}(\A_f^S),K^S)) \ar[d]^{\forget} \\ D(K/K^\prime) \ar[r]^-{R\Gamma(K/K^\prime,-)} & D(\Z) }
\]
commutes up to natural isomorphism. To see this last point, note that the corresponding underived diagram commutes up to natural isomorphism, and then use that the forgetful morphism $\mathrm{Mod}(\cH(\mathrm{G}(\A_f^S),K^S)\times K/K^\prime) \to \mathrm{Mod}(K/K^\prime)$ preserves injectives since its left adjoint $\cH(\mathrm{G}(\A_f^S),K^S)\otimes_\Z -$ is exact.
\end{proof}

Again, there are obvious versions with $\Z$ replaced by a commutative ring $R$, and equivariant versions. For the next proposition, we continue to use the notation of Proposition \ref{finite to finite basic}; $K = K^S K_S$ and $K^\prime =K^S K_S^\prime$ with $K_S^\prime \sub K_S$ normal. Below, tensor products are given diagonal actions.

\begin{prop}\label{going from finite to finite} Let $V$ and $W$ be left $\Z[G(\A_f^S)\times K_S]$-modules which are finite and free as $\Z$-modules, and assume that the $K^\prime$-action on $W$ is trivial. We may view both $V$ and $W$ as $G(\A_f^S)\times K_S$-equivariant local systems on $\ol{\mf{X}}^\mathrm{G}$. Then we have a $\cH(\mathrm{G}(\A_f^S),K^S)$-equivariant isomorphism
\[
R\Gamma(\ol{X}_K^\mathrm{G},V\otimes_\Z W) \cong R\Gamma(K/K^\prime, R\Gamma(\ol{X}_{K^\prime}^\mathrm{G}, V) \otimes_\Z W)
\]
in $D(\Z)$. 
\end{prop}

\begin{proof} 
By Proposition \ref{finite to finite basic} it suffices to show that 
\[
R\Gamma(\ol{X}_{K^\prime}^\mathrm{G},V\otimes_\Z W) \cong R\Gamma(\ol{X}_{K^\prime}^\mathrm{G}, V) \otimes_\Z W
\]
$\cH(\mathrm{G}(\A_f^S),K^S)$-equivariantly in $D(K/K^\prime)$, which would follow from an isomorphism
\[
R\Gamma(K^\prime, R\Gamma(\ol{\mf{X}}^\mathrm{G},V\otimes_\Z W)) \cong R\Gamma(K^\prime, R\Gamma(\ol{\mf{X}}^\mathrm{G}, V)) \otimes_\Z W
\]
in $D(\cH(\mathrm{G}(\A_f^S),K^S)\times K/K^\prime)$, where $W$ is viewed as a trivial $\cH(\mathrm{G}(\A_f^S),K^S))$-module. To prove this, we argue as follows. First, note that trivial local systems are acyclic for $\Gamma(\ol{\mf{X}}^\mathrm{G},-)$, and so one computes that
\[
R\Gamma(\ol{\mf{X}}^\mathrm{G},V\otimes_\Z W) \cong \Gamma(\ol{\mf{X}}^\mathrm{G},V \otimes_\Z W) = \Gamma(\ol{\mf{X}}^\mathrm{G},V) \otimes_\Z W \cong R\Gamma(\ol{\mf{X}}^\mathrm{G},V) \otimes_\Z W
\]
where the middle equality follows from the equality $\Gamma(\ol{\mf{X}}^\mathrm{G},V) = \mathrm{Fun}(\pi_0(\ol{\mf{X}}^\mathrm{G}),V)$ and similarly for $V\otimes_\Z W$, using that $W$ is a finite free $\Z$-module (here $\mathrm{Fun}(X,Y)$ denotes the functions between two sets $X$ and $Y$). To conclude from here, it suffices to show that the diagram 
\[
    \xymatrix@C+2pc{ D(\mathrm{G}(\A_f^S)\times K_S) \ar[r]^-{R\Gamma(K^\prime,-)} \ar[d]^{-\otimes_\Z W} & D(\cH(\mathrm{G}(\A_f^S),K^S)\times K/K^\prime) \ar[d]^{-\otimes_\Z W} \\ D(\mathrm{G}(\A_f^S)\times K_S) \ar[r]^-{R\Gamma(K^\prime,-)} & D(\cH(\mathrm{G}(\A_f^S),K^S)\times K/K^\prime) }
\]
commutes. The corresponding underived diagram commutes, so it suffices to show that there are enough injectives $I$ in $\mathrm{Mod}(\mathrm{G}(\A_f^S)\times K_S)$ for which $I\otimes_\Z W$ is injective. For this, we may take the injectives $I_M \defeq \mathrm{Fun}(\mathrm{G}(\A_f^S) \times K_S,M)$ for injective $\Z$-modules $M$, with the left action induced by right translation. These are enough, and by a standard untwisting argument (using that $W$ is a finite free $\Z$-module) $I_{M}\otimes_\Z W \cong I_{M\otimes_\Z W}$; note that $M\otimes_\Z W$ is an injective $\Z$-module since $W$ is finite free.
\end{proof}

Once again we have an obvious analogue for arbitrary commutative rings; we will use it for $\Lambda = \Z/p^r$.

\subsection{Homological algebra}\label{comm alg}

In this subsection we collect some algebraic facts and constructions that we will need. We start by discussing idempotents and direct summands in derived categories. Let $\cD$ be an additive category. We assume that $\cD$ is \emph{idempotent complete}, meaning that if $X$ is an object of $\cD$ and $e\in \mathrm{End}_{\cD}(X)$ is idempotent, then there exists a direct sum decomposition $X \cong Y \oplus Z$ given in part by morphisms $i : Y \to X$ and $p : X \to Y$ such that $ip=e$. A choice of $i$ and $p$ gives an additive homomorphism
\[
 pr_{X,Y} : \mathrm{End}_{\cD}(X) \to \mathrm{End}_{\cD}(Y)
\]
given by $f \mapsto pfi$. If $R$ is a commutative ring, then we define an action of $R$ on $X \in \cD$ to be a ring homomorphism $a_X : R \to \mathrm{End}_{\cD}(X)$. Suppose that $R$ acts on $X$ and that $Y$ is a direct summand of $X$ corresponding to an idempotent $e=ip$ as above. We single out two situations that will be relevant for this paper. 
\begin{enumerate}
\item Assume that we also have an action $a_Y$ of $R$ on $Y$. Then we say that $Y$ is an $R$-equivariant direct summand of $X$ if $a_Y = pr_{X,Y}\circ a_X$. 
\item Assume that $e$ commutes with $a_X(R)$, without assuming that there is a given action of $R$ on $Y$. Then one may define an action of $R$ on $Y$ by $a_Y \defeq pr_{X,Y}\circ a_X$; one checks easily that this is a ring homomorphism. By construction, $Y$ is then an $R$-equivariant direct summand of $X$.
\end{enumerate}
In the situation when $\cD$ is a derived category of an abelian category $\cA$, one has the following compatibilities. The induced morphism on cohomology $H^\ast(i) : H^\ast(Y) \to H^\ast(X)$ is an isomorphism onto $\mathrm{Im}(H^\ast(e))$. One can define two actions of $R$ on $\mathrm{Im}(H^\ast(e))$: One by restriction from $H^\ast(X)$, since the action $R$ commutes with $H^\ast(e)$, and one by transporting the action on $H^\ast(Y)$ coming from $a_Y$ via the isomorphism $H^\ast(i)$. These actions are easily seen to agree.

\medskip

Our source of idempotents will come from the following situations. Keep the notation above. Our first situation is when $T \sub \mathrm{End}_{\cD}(X)$ is a commutative subring with finitely many maximal ideals, and we have a decomposition 
$$ T = \prod_{\mf{m}}T_\mf{m} $$
where $\mf{m}$ runs over the maximal ideals of $T$. This happens for example if $T$ has finite cardinality, or is a finite $\Zp$-algebra. Fix a particular maximal ideal $\mf{m}$, then the identity in $T_\mf{m}$ is an idempotent $e\in T$. The other situation is inspired by Hida theory. Suppose that $f \in \mathrm{End}_{\cD}(X)$ generates a subring of finite cardinality (for example when $\mathrm{End}_{\cD}(X)$ is finite). Then, for large enough $k$, $f^{k!}$ stabilizes and is an idempotent. For further discussion on idempotents, see \cite[\S 2.4]{khare-thorne} and \cite[\S 3.2]{newton-thorne}.

\medskip

We move on. Later in Theorem \ref{direct summand}, which is the main technical result of this section, we will need to ``glue'' complexes computing cohomology groups of towers of locally symmetric spaces at finite level, as the level varies, and deduce consequences in the limit from information at finite level. The following two lemmas will provide a rather concrete context for this, and in \S \ref{subsec: ordinary parts} we will discuss explicit complexes computing cohomology of locally symmetric spaces to which they may be applied. Let $\mathrm{G}_\infty$ be a compact $p$-adic analytic group, which is the inverse limit of a countable sequence
\[
\dots \to \mathrm{G}_3 \to \mathrm{G}_2 \to \mathrm{G}_1
\]
of finite groups. Let $\mathrm{H}_c = \mathrm{Ker}(\mathrm{G}_\infty \to \mathrm{G}_c)$.  The following is a variation on \cite[Lemma 2.13]{khare-thorne}. Throughout, we will say that a complex $C^\bullet$ is \emph{concentrated} in an interval $[a,b]$ (with $a,b$ always assumed to be integers) if $C^i=0$ whenever $i \notin [a,b]$.

\begin{lemma}\label{gluing complexes}
Let $V_c^\bullet$ be a complex of finitely generated injective $\mathrm{G}_c$-representations over $\Lambda$, together with $\mathrm{G}_{c}$-equivariant isomorphisms $f_{c} : V_c^\bullet \to \Gamma(\mathrm{H}_c,V_{c+1}^\bullet)$, for $1\leq c < \infty$. Assume that $V_c^\bullet$ is concentrated in $[a,b]$ for all $c$.  Then:

\smallskip

\begin{enumerate}
\item $V_\infty^\bullet = \varinjlim_{c<\infty}V_c^\bullet $ is a complex of smooth admissible injective $\mathrm{G}_\infty$-representations, concentrated in $[a,b]$. Moreover, we have natural isomorphisms $F_c : V_c^\bullet \to \Gamma(\mathrm{H}_c,V_{\infty}^\bullet)$ for all $c$. 

\smallskip

\item The natural map $\varinjlim_{c<\infty}H^\ast(V_c^\bullet) \to H^\ast(V_\infty^\bullet)$ is an isomorphism.

\smallskip

\item Suppose that $W_c^\bullet$, $1\leq c < \infty$, is another system of complexes of finitely generated injective $\mathrm{G}_c$-representations over $\Lambda$, concentrated in $[a^\prime,b^\prime]$, together with $\mathrm{G}_{c}$-equivariant isomorphisms $g_{c} : W_c^\bullet \to \Gamma(\mathrm{H}_c,W_{c+1}^\bullet)$. Suppose that we have morphisms $t_c \in \Hom_{K(\mathrm{G}_c,\Lambda)}(V_c^\bullet, W_c^\bullet)$ such that $t_{c+1}\circ f_c = g_c \circ t_c$ in $K(\mathrm{G}_{c},\Lambda)$ for $1\leq c <\infty$. Then there exists a unique $t_\infty \in \Hom_{K(\mathrm{G}_\infty,\Lambda)}(V_\infty^\bullet, W_\infty^\bullet)$ such that the natural diagram 
\[
    \xymatrix{ \Gamma(\mathrm{H}_c,V_\infty^\bullet) \ar[r]^-{t_\infty} \ar[d] & \Gamma(\mathrm{H}_c,W_\infty^\bullet) \ar[d] \\ V_c^\bullet \ar[r]^-{t_c} & W_c^\bullet }
\]
commutes in $K(\mathrm{G}_c,\Lambda)$.

\smallskip

\item Suppose that we have $V_c^\bullet$ and $W_c^\bullet$ as above for $1\leq c \leq \infty$, but now suppose that we have a $\mathrm{G}_\infty$-equivariant map $u_\infty : V_\infty^\bullet \to W_\infty^\bullet$. Set $u_c = \Gamma(H_c,u_\infty): V_c^\bullet \to W_c^\bullet$ and let $\mathrm{Cone}(u_c)$ denote the mapping cone (for $1\leq c\leq \infty$). Then $\mathrm{Cone}(u_\infty)$ is a complex of smooth  admissible injective $G_\infty$-representations, and we have $\Gamma(H_c,\mathrm{Cone}(u_\infty))= \mathrm{Cone}(u_c)$, which is a complex of finitely generated injective $\mathrm{G}_c$-representations.

\end{enumerate}
\end{lemma}

\begin{proof}
We start with part (1). The transition maps $V_c^\bullet \to V_{c+1}^\bullet$ are the $f_c$ composed with the natural inclusions $\Gamma(\mathrm{H}_c,V_{c+1}^\bullet) \hookrightarrow V_{c+1}^\bullet$. Then everything is clear, apart from injectivity of the $V_\infty^i$. For this, first note that by \cite[Proposition 2.1.9]{emerton-ord} it suffices to prove that the $V_\infty^i$ are injective in the category of admissible $G_\infty$-representations. Then note that if $M$ is an another admissible $G_\infty$-representation over $\Lambda$, then $\Hom_{\mathrm{G}\infty}(M, V_\infty^i) \cong \varprojlim_c \Hom_{\mathrm{G}_c}(\Gamma(\mathrm{H}_c,M),V_c^i)$, and the right hand side gives an exact functor in $M$ by the injectivity of $V_c^i$, the finiteness of the $\Hom_{\mathrm{G}_c}(\Gamma(\mathrm{H}_c,M),V_c^i)$ and the Mittag-Leffler condition. This finishes part (1). Part (2) then follows from exactness of direct limits. Part (3) is proved exactly as \cite[Lemma 2.13(3)]{khare-thorne}. Finally part (4) follows immediately from the definition of the mapping cone.
\end{proof}

\begin{lemma}\label{idempotents}
Assume that we have $(V_c,f_c)_{1\leq c \leq \infty}$ as in Lemma \ref{gluing complexes}. Assume further that we have idempotents $e_c \in \End_{K(\mathrm{G}_c,\Lambda)}(V_c^\bullet)$ such that $e_{c+1}\circ f_c = f_c \circ e_c$ in $K(\mathrm{G}_{c+1},\Lambda)$. By Lemma \ref{gluing complexes}(3), this data is equivalent to an idempotent $e_\infty \in \End_{K(\mathrm{G}_\infty,\Lambda)}(V_\infty^\bullet)$. Then the direct summand of $V_\infty^\bullet$ in $D_{sm}(\mathrm{G}_\infty,\Lambda)$ cut out by $e_\infty$ is isomorphic to a complex $W_\infty^\bullet$ of admissible injective $\mathrm{G}_\infty$-representations concentrated in $[a,b]$. Moreover, for each $1\leq c < \infty$, $W_c^\bullet \defeq \Gamma(\mathrm{H}_c,W_\infty^\bullet)$ is a complex of finitely generated injective $\mathrm{G}_c$-representations, and it is isomorphic in $D(\mathrm{G}_c,\Lambda)$ to the direct summand of $V_c^\bullet$ cut out by $e_c$.
\end{lemma}

\begin{proof}
We show that the direct summand of $V_\infty^\bullet$ in $D_{sm}(\mathrm{G}_\infty,\Lambda)$ cut out by $e_\infty$ is isomorphic to a complex $W_\infty^\bullet$ of admissible injective $\mathrm{G}_\infty$-representations concentrated in $[a,b]$; the rest then follows, using functoriality and the fact that $\Gamma(\mathrm{H}_c,-)$ preserves injectives. We may work in $D_{\mathrm
{adm}}(\mathrm{G}_\infty, \Lambda)$. By \cite[Equation (2.2.12)]{emerton-ord1}, the category of admissible $\mathrm{G}_\infty$-representations over $\Lambda$ is anti-equivalent to the category of finitely generated left modules for the (left and right) Noetherian ring $\Lambda \llbracket \mathrm{G}_\infty \rrbracket$. The assertion then follows from Lemma \ref{direct summands of perfect complexes} below.
\end{proof}

For convenience, we use chain complexes in the next lemma. Truncations always refer to good truncations.

\begin{lemma}\label{direct summands of perfect complexes}
Let $R$ be a left Noetherian ring (in particular, $R$ need not be commutative); all $R$-modules will be left modules. Assume that $C_\bullet$ is a complex of finitely generated projective $R$-modules concentrated in $[a,b]$, and assume that we have a direct sum decomposition $C_\bullet \cong D_\bullet \oplus E_\bullet$ in $D(R)$. Then $D_\bullet$ is isomorphic in $D(R)$ to a complex of finitely generated projective $R$-modules $D_\bullet^{\prime}$ concentrated in $[a,b]$.
\end{lemma}

\begin{proof}
This is presumably well known, we sketch a proof. Note that if $i\notin [a,b]$, then $H_i(D_\bullet)=0$ and $\mathrm{Ext}_R^i(D_\bullet, M)=0$ for any $R$-module $M$, since both these things are true for $C_\bullet$ and are preserved under taking direct summands. We explain how to modify $D_\bullet$ to satisfy the requirements of the lemma. First, by truncation, we may assume that $D_i=0$ in degrees $<a$. Hence, we may replace it with a quasi-isomorphic complex of projective $R$-modules which is $0$ in degrees $<a$. By truncation, using the facts in the second sentence of the proof and some standard dimension shifting arguments, we may assume that $D_i = 0$ when $i \notin [a,b]$. Finally, applying \cite[Chapter II, \S 5, Lemma 1]{mumford} (whose proof does not require commutativity of the ring in question) and remembering that finitely generated flat $R$-modules are projective, we may assume that the $D_i$ are in addition finitely generated. This finishes the sketch of proof.
\end{proof}

\medskip

We now move on to discuss some results related to continuity of actions on objects in derived categories which we will need. We regard $\Lambda$ as a topological ring with the discrete topology. Let $\mathbb{T}$ be a commutative topological $\Lambda$-algebra, with the topology defined by a decreasing sequence of ideals $\dots I_k\subseteq I_{k-1}\subseteq \dots \subseteq I_1\subset \mathbb{T}$ such that $\mathbb{T}/I_k$ is a finite free $\Lambda$-algebra for every $k\in \Z_{\geq 1}$.

\begin{lemma}\label{injective resolution with continuous action}  Let $\mathrm{G}$ be a compact $p$-adic analytic group.
Let $K^\bullet$ be a bounded below complex of smooth $\mathrm{G}$-representations with $\Lambda$-coefficients and with a continuous action of $\mathbb{T}$ on each term that commutes with the $\Lambda[\mathrm{G}]$-module structure. Then there exists a bounded below complex $I^\bullet$ of injective objects in $\mathrm{Mod}_{\mathrm{sm}}(\mathrm{G},\Lambda)$ with a continuous $\mathbb{T}$-action on each term, that commutes with the $\Lambda[G]$-module structure, and such that there is a $\mathbb{T}$-equivariant quasi-isomorphism $K^\bullet \to I^\bullet$.
\end{lemma}

\begin{proof} First, we prove that for every $M\in \mathrm{Mod}_{\mathrm{sm}}(\mathrm{G},\Lambda)$ with a continuous $\mathbb{T}$-action there is an injection 
\[
M\hookrightarrow I(M) 
\] 
where $I(M)\in \mathrm{Mod}_{\mathrm{sm}}(\mathrm{G},\Lambda)$ is an injective object with a continuous $\mathbb{T}$-action. For each $k\in \Z_{\geq 1}$, let $M_k$ denote the $\Lambda[\mathrm{G}]$-submodule of $M$ on which the $\mathbb{T}$-action factors through $\Lambda_k\defeq \mathbb{T}/I_{k}$. Since $M$ is a discrete $\mathbb{T}$-module, we have $M=\varinjlim_{k}M_k$. We will construct $I(M)\defeq \varinjlim_{k}I(M)_k$ inductively, where each $I(M)_k$ will be an injective object in $\mathrm{Mod}_{\mathrm{sm}}(\mathrm{G},\Lambda_k)$. The forgetful functor from $\mathrm{Mod}_{\mathrm{sm}}(\mathrm{G}, \Lambda_k)$ to $\mathrm{Mod}_{\mathrm{sm}}(\mathrm{G},\Lambda)$ has an exact left adjoint $-\otimes_{\Lambda}\Lambda_k$ (exact because $\Lambda_k$ is finite free over $\Lambda$), so it preserves injectives, and hence the $I(M)_k$ will be injective in $\mathrm{Mod}_{\mathrm{sm}}(\mathrm{G},\Lambda)$ as well. We conclude that $\varinjlim_{k}I(M)_k$ is also an injective object in $\mathrm{Mod}_{\mathrm{sm}}(\mathrm{G},\Lambda)$ by~\cite[Lemma 2.1.3]{emerton-ord}. 

\medskip

For $k=1$, simply let $M_1 \hookrightarrow I(M)_1$ be any embedding of $M_1$ into an injective object of $\mathrm{Mod}_{\mathrm{sm}}(\mathrm{G},\Lambda_1)$. Now let $k\geq 2$ and assume that we have already constructed $I(M)_{k-1}$. Define $M'_{k}$ to be the pushout $M_k\oplus_{M_{k-1}}I(M)_{k-1}$ in $\mathrm{Mod}_{\mathrm{sm}}(\mathrm{G},\Lambda_k)$. Note that $M_{k-1}$ injects into both $M_k$ and $I(M)_{k-1}$, so that we have injections $M_k\hookrightarrow M'_k$ and $I(M)_{k-1}\hookrightarrow M'_k$. Then we let $I(M)_k$ be an injective object in $\mathrm{Mod}_{\mathrm{sm}}(\mathrm{G},\Lambda_k)$ such that $M'_{k}\hookrightarrow I(M)_k$. This gives us  a compatible system of injections $M_k \hookrightarrow I(M)_k$, and hence the desired $I(M)$.

\medskip

This now allows us to prove the lemma. If we forget about the topology on $\mathbb{T}$, the category $\mathrm{Mod}_{\mathrm{sm}}(\mathrm{G}, \mathbb{T})$ is abelian. The full subcategory $\cA$ consisting of objects $M$ for which the action map 
\[
\mathbb{T}\times M\to M
\]
is continuous, when $\mathbb{T}$ is equipped with the above topology and $M$ is equipped with the discrete topology, is an abelian subcategory.
One way to see this is by noting that $\mathbb{T}$ acts continuously on $M$ if and only if the annihilator ${\rm Ann}_{\mathbb{T}}(m)$ is open in $\mathbb{T}$ for every $m\in M$, and using this criterion it is easy to check that $\cA$ is closed under taking kernels and cokernels, contains the zero object and is closed under taking the direct sum of two objects. Let $\cI$ denote the subset of objects of $\cA$ that are injective after applying the forgetful functor to $\mathrm{Mod}_{\mathrm{sm}}(\mathrm{G},\Lambda)$. The set $\cI$ contains the zero object. We have shown above that any object in $\cA$ admits an injective morphism to an object of $\cI$. The fact that any bounded below complex $K^\bullet$ with terms in $\cA$ admits a quasi-isomorphism to a bounded below complex with terms in $\cI$ then follows from~\cite[Tag 05T6]{stacks-project}.
\end{proof}

\begin{remark}
We remark that we crucially used that the ring $\Lambda[\![\mathrm{G}]\!]$ is Noetherian in the proof. This would fail if we worked with $\cO_C/p^r$-coefficients instead of $\Lambda$-coefficients.
\end{remark}

The following proposition and its proof are due to James Newton. We thank him for allowing us to include it here.

\begin{prop}\label{getting rid of C} Suppose $\mathrm{G}$ is a finite group and $P^\bullet$ is a perfect complex in $D(\mathrm{G},\Lambda)$ with an action of $\mathbb{T}$. Suppose $I^\bullet$ is a bounded below complex of injectives in $D(\mathrm{G}, \Lambda)$ with a continuous $\mathbb{T}$-action on each term and such that $P^\bullet$ is a $\mathbb{T}$-equivariant
direct summand of $I^\bullet$ in $D(\mathrm{G},\Lambda)$. Then the map
\[
\mathbb{T}\to \mathrm{End}_{D(\mathrm{G},\Lambda)}(P^\bullet)
\]
is continuous for the discrete topology on the target.
\end{prop}

\begin{proof} We can replace $P^\bullet$ with a bounded complex of finite projective $\Lambda[\mathrm{G}]$-modules, which are also injective $\Lambda[\mathrm{G}]$-modules since $\Lambda$ is self-injective and $\mathrm{G}$ is a finite group. The fact that $P^\bullet$ is a direct summand of $I^\bullet$ in $D(\mathrm{G},\Lambda)$ implies that we have $\Lambda[\mathrm{G}]$-equivariant morphisms of complexes 
\[
f\colon P^\bullet \to I^\bullet,
\]
\[
g\colon I^\bullet \to P^\bullet
\] 
giving the $\mathbb{T}$-equivariant splitting in $D(\mathrm{G},\Lambda)$. (In particular, $g\circ f$ is homotopic to the identity on $P$). For each of the finitely many $n$ such that $P^n\not=0$, we choose generators $x^n_1, \dots, x^n_{k_n}$ of the finitely generated $\Lambda[\mathrm{G}]$-modules $P^n$. Since $\mathbb{T}$ acts continuously on $I^n$, the annihilator in $\mathbb{T}$ of $f(x^n_i)$ is an open ideal for each $i,n$. Taking the intersection over all the finitely many $i$ and $n$, we see that the annihilator in $\mathbb{T}$ of the subcomplex $f(P^\bullet)\subset I^\bullet$ is an open ideal.  The fact that the map 
\[
\mathbb{T}\to \mathrm{End}_{D(\mathrm{G},\Lambda)}(P^\bullet)
\]
has open kernel is now an immediate consequence of Lemma~\ref{maps to 0} below. 
\end{proof}

\begin{lemma}\label{maps to 0} Suppose $T \in \mathbb{T}$ annihilates $f(P^\bullet)$. Then $T$ maps
to zero in $\mathrm{End}_{D(\mathrm{G},\Lambda)}(P^\bullet)$.
\end{lemma}

\begin{proof} Choose a lift $T_P$ of $T$ to an endomorphism of the complex $P^\bullet$ (such a lift is well-defined up to homotopy). Since $P^\bullet$ is a complex of injectives and $f$ and $g$ cut out $P^\bullet$ as a $\mathbb{T}$-direct summand of $I^\bullet$, $T_P$ coincides with $g \circ T \circ f$ up to a
$\Lambda[\mathrm{G}]$-linear null homotopy of $P^\bullet$. Since $T$ annihilates $f(P^\bullet)$, this shows that $T_P$ is a null homotopy of $P^\bullet$, which is what we needed to prove. 
\end{proof}

\subsection{Explicit complexes and ordinary parts}\label{subsec: ordinary parts} We now recall some material on explicit complexes that can be used to compute cohomology of locally symmetric spaces; this may be found (in varying levels of generality) in \cite{ash-stevens,hansen-thesis,khare-thorne}. We will then discuss some aspects of Hida theory in the situation we require.

As in \S \ref{subsec: loc sym space}, let $\mathrm{G}$ be a connected linear algebraic group over $\Q$. We let $C_{\A,\bullet}$ denote the complex of singular chains of $\ol{X}^\mathrm{G} \times \mathrm{G}(\A_f)$ with $\Z$-coefficients; this is naturally a left $\mathrm{G}(\Q) \times \mathrm{G}(\A_f)$-module under the diagonal left action of $\mathrm{G}(\Q)$ on $\ol{X}^\mathrm{G} \times \mathrm{G}(\A_f)$ and the left action by $\mathrm{G}(\A_f)$ given by inverting the natural right action on $\ol{X}^\mathrm{G} \times \mathrm{G}(\A_f)$ via the second factor. For any neat compact open subgroup $K\sub \mathrm{G}(\A_f)$ and any left $K$-module $V$, we set $C_{\A}^\bullet(K,V) \defeq \Hom_{\mathrm{G}(\Q)\times K}(C_{\A,\bullet},V)$, where $V$ is given the trivial $\mathrm{G}(\Q)$-action. The proof of \cite[Proposition 6.2]{khare-thorne} works in full generality and shows that $C_{\A}^\bullet(K,V) \cong R\Gamma(\ol{X}^\mathrm{G}_K,V)$ in $D(\Z)$ (see also \cite[Proposition 2.1.]{hansen-thesis}). Moreover, whenever $V$ is a $\mathrm{G}(\A_f)$-module, the action of double coset operators on $R\Gamma(\ol{X}^G_K,V)$ can be made explicit at the level of complexes on $C_{\A}^\bullet(K,V)$. More generally, if $K,K^\prime$ are neat compact open subgroups and $g\in \mathrm{G}(\A_f)$, we have a map $[KgK^\prime]^\ast : C_{\A}^\bullet(K^\prime,V) \to C_{\A}^\bullet(K,V)$ given by 
\[
\left( [KgK^\prime]^\ast (\varphi) \right) (\sigma) = \sum_i g_i \varphi(g_i^{-1}\sigma),
\]
where $\varphi \in C_{\A}^\bullet(K^\prime,V)$, $\sigma \in C_{\A,\bullet}$ and $KgK^\prime = \bigsqcup_i g_i K^\prime$; this is easily checked to be independent of the choice of $g_i$. 
If $K^\prime \sub K$ and $V$ is a left $K^\prime$-module, we write $\mathrm{Ind}_{K^\prime}^K V$ for the left $K$-module of functions $f : K \to V$ satisfying $f(k^\prime k) = k^\prime f(k)$ for all $k\in K, k^\prime \in K^\prime$, with the left $K$-action coming from right translation.

\begin{lemma}\label{pushforward of adelic complexes}
Let $K^\prime \sub K$ and let $V$ be a $K^\prime$-module. Then there is a natural isomorphism $C_\A^\bullet(K^\prime,V) \cong C_\A^\bullet(K,\mathrm{Ind}_{K^\prime}^K V)$.
\end{lemma} 

\begin{proof}
One defines maps $\mathrm{ext} : C_\A^\bullet(K^\prime,V) \to C_\A^\bullet(K,\mathrm{Ind}_{K^\prime}^K V)$ and $\mathrm{res} : C_\A^\bullet(K,\mathrm{Ind}_{K^\prime}^K V) \to C_\A^\bullet(K^\prime, V)$ by $\mathrm{ext}(\varphi)(\sigma)(k) = \varphi(k \sigma)$ and $\mathrm{res}(\psi)(\sigma) = \psi (\sigma) (1)$; it is then straightforward to check that they are mutual inverses.
\end{proof}

Note that in the entire discussion above we could have replaced $\ol{X}^\mathrm{G}$ by  $\partial \ol{X}^\mathrm{G}$, 
in which case we denote the corresponding complexes by $C_{\A,\partial}^\bullet(K,V)$. 

The complexes $C_{\A}^\bullet(K,V)$ are rather large. To define complexes computing $R\Gamma(\ol{X}^\mathrm{G}_K,V)$ with good finiteness properties, we fix a neat compact open $K_0$ and only consider compact opens $K \sub K_0$. Fix a finite triangulation of $\ol{X}^\mathrm{G}_{K_0}$ such that $\partial \ol{X}^\mathrm{G}_{K_0}$ is a subcomplex of the triangulation. Choose representatives $g_i$ of $\mathrm{G}(\Q) \backslash \mathrm{G}(\A_f)  / K_0$; the disjoint union of the embeddings $\wt{g}_i : \ol{X}^\mathrm{G} \to \ol{X}^\mathrm{G} \times \mathrm{G}(\A_f)$ induce an isomorphism
\[
\bigsqcup_i \Gamma_{0,i} \backslash \ol{X}^\mathrm{G} \toisom \ol{X}_{K_0}^\mathrm{G},
\]
where $\Gamma_{0,i} = \mathrm{G}(\Q) \cap g_i K_0 g_i^{-1}$. Restricting to a component, the triangulation of $\ol{X}^\mathrm{G}_{K_0}$ gives rise to a $\Gamma_{0,i}$-invariant triangulation of $\ol{X}^\mathrm{G}$ with the boundary as a subcomplex, giving bounded complexes $F_\bullet^i$ and $F_\bullet^{i,\partial}$ of finite free $\Z[\Gamma_{0,i}]$-modules. For each $K_0$-module $V$ and $? \in \{\emptyset,\partial\}$, this gives complexes 
\[
C_{?}^\bullet(K,V) \defeq \bigoplus_i \Hom_{\Z[\Gamma_{0,i}]}(F_\bullet^{i,?},V_i),
\]
where $V_i$ denotes $V$ with the $\Gamma_{0,i}$-module structure obtained by restricting the $K_0$-module structure via $\Gamma_{0,i} \to K_0$, $\gamma \mapsto g_i^{-1}\gamma g_i$. Let $C^?_\bullet(\ol{X}^\mathrm{G})$ denote the complex of singular chains of $?\ol{X}^\mathrm{G}$. We fix once and for all maps $s_i^? : F_\bullet^{i,?} \to C^?_\bullet(\ol{X}^\mathrm{G})$ and $t_i^? : C^?_\bullet(\ol{X}^\mathrm{G}) \to F_\bullet^{i,?}$ of chain complexes of $\Gamma_{0,i}$-modules which are mutual inverses in $K(\Gamma_{0,i})$. The proof of \cite[Proposition 6.2]{khare-thorne} works by showing that
\[
\oplus \wt{g}_{i,\ast} : C_{sing,?}^\bullet(K_0,V) \defeq \bigoplus_i \Hom_{\Z[\Gamma_{0,i}]}(C^?_\bullet(\ol{X}^\mathrm{G}),V_i) \to C^\bullet_{\A,?} (K_0,V)
\]
is an isomorphism. Thus, the $s_i^?$ and $t_i^?$ induce chain homotopy equivalences $C_?^\bullet(K_0,V) \cong C_{\A,?}^\bullet(K_0,V)$, showing that the former are isomorphic to $R\Gamma(?\ol{X}^\mathrm{G}_{K_0},V)$ in $D(\Z)$. Combining this with Lemma \ref{pushforward of adelic complexes}, we see that 
\[
C_?^\bullet(K_0,\mathrm{Ind}_K^{K_0}V)\cong R\Gamma(?\ol{X}^\mathrm{G}_{K},V)
\]
in $D(\Z)$ for $K$-modules $V$. 

We list the most important properties of $C_?^\bullet(K_0,V)$, which follow from the definitions. First, $C_?^\bullet(K_0,V)$ is bounded independently of $K$ and $V$, and each term is isomorphic as a $\Z$-module to a finite sum of copies of $V$. We also have an equivariant version of this. If $K^\prime \sub K \sub K_0$ with $K^\prime \sub K$ normal and $V$ is a $K$-module, then $(k\ast f)(x) = k f(k^{-1}x)$ defines a left $K/K^\prime$-module structure on $\mathrm{Ind}_{K^\prime}^{K_0}V$ commuting with the previously defined $K_0$-module structure, and each term of $C_?^\bullet(K_0,\mathrm{Ind}_{K^\prime}^{K_0}V)$ is isomorphic to a finite sum of copies of $\Z[K/K^\prime]\otimes_{\Z}V$, with $V$ considered as a $\Z$-module, by untwisting. In particular, whenever $V$ is finite projective, this shows that $R\Gamma(?\ol{X}^\mathrm{G}_{K^\prime},V)$ is perfect in $D(K/K^\prime)$. Moreover, we have equalities $\Gamma(K/K^\prime,C_?^\bullet(K_0,\mathrm{Ind}_{K^\prime}^{K_0}V)) = C^\bullet_?(K_0,\mathrm{Ind}_K^{K_0}V)$. Upon noting that $\Lambda[K/K^\prime]$ is self-injective, Lemma \ref{gluing complexes} gives us the following lemma, which can also be proved directly (note that $\Map_{\mathrm{cont}}(K_{0,p},V) = \varinjlim_{K^p=K_0^p} \mathrm{Ind}_K^{K_0}V$):

\begin{lemma}\label{perfect complex for completed cohomology}
Let $V$ be a $K_0$-representation which is finite and free as a $\Lambda$-module, and such that $K_0^p$ acts trivially on $V$. 
The complex $C_?^{\bullet}(K_0, \Map_{\mathrm{cont}}(K_{0,p},V)) = \varinjlim_{K^p=K_0^p} C_?^\bullet(K_0,\mathrm{Ind}_K^{K_0}V)$ is a bounded complex of smooth admissible 
injective $K_{0,p}$-representations.
\end{lemma} 

We may also treat compactly supported cohomology. There is a natural map $C^\bullet(K,V) \to C^\bullet_\partial(K,V)$ coming from $F_\bullet^{i,\partial} \to F_\bullet^{i}$. Using this, we may define
\[
C_c^\bullet(K_0,V) \defeq \mathrm{Cone}(C^\bullet(K_0,V) \to C^\bullet_\partial(K_0,V))[-1],
\]
where $\mathrm{Cone}$ denotes the mapping cone. From the definitions and the excision sequence, $C_c^\bullet(K_0,V) \cong R\Gamma_c(X^\mathrm{G}_{K_0},V)$ in $D(\Z)$. Moreover, $C_c^\bullet(K_0,V)$ inherits all the properties above from $C^\bullet(K_0,V)$ and $C^\bullet_\partial(K_0,V)$, i.e.\ everything above holds with $?=c$. We have already touched upon it, but let us finish by noting explicitly that as in \S \ref{subsec: loc sym space}, we have versions of the above with other coefficients as well as equivariant versions.

The remainder of this section will discuss one particular aspect of Hida theory in the $P$-ordinary setting, analogous to \cite[Lemma 6.10]{khare-thorne}. We work over $\Lambda$, and we will work with our symplectic/unitary group $G$ and ordinary parts with respect to the Siegel parabolic $P$, though our discussion is rather general and we will make use of a few obvious variants in \S \ref{subsec: boundary}. Fix a neat $K^p$. We will need to define some level subgroups. For $m\geq 1$, we define
\[
\Gamma_0(p^m) \defeq \{ g\in G(\Zp) \mid (g\,\, \mathrm{mod} \,\, p^m) \in P(\Z/p^m) \};
\]
\[
\Gamma_m \defeq \Gamma_1(p^m) \defeq \{ g\in G(\Zp) \mid (g\,\, \mathrm{mod} \,\, p^m) \in N(\Z/p^m) \}.
\]
More generally, for a pair of integers $\ell\geq m \geq 0$, $\ell\geq 1$, let $\Gamma_{m,\ell} \sub \Gamma_0(p^\ell)$ be the subgroup defined by the Iwahori factorization $\ol{N}(\Zp)^{p^\ell}K_{M,m,p}N(\Zp) \supseteq \Gamma_1(p^\ell)$. $\Gamma_{m,\ell}$ is normal in $\Gamma_0(p^\ell)$, $\Gamma_{m,\ell} / \Gamma_1(p^\ell) \cong K_{M,m,p}/K_{M,\ell,p}$, and $\Gamma_0(p^\ell)/\Gamma_{m,\ell} \cong M(\Z/p^m)$. In particular, $\Gamma_m = \Gamma_{m,m}$ with the notation above. We also set
\[
\Gamma_0 \defeq \Gamma_0(p).
\]
We make analogous definitions for compact open subgroups of $G(\A_f)$; set $K_{m,\ell}=K^p\Gamma_{m,\ell}$, and $K_m = K^p \Gamma_m$ for $m\geq 0$. We will make use of a particular Hecke operator, which we will define using the matrix
\[
\gamma \defeq \prod_{\iota} \mathrm{diag}(p,\dots,p,p^{-1},\dots,p^{-1})\in M(\Q_p),
\]
where in each matrix $p$ and $p^{-1}$ occur $n$ times on the diagonal and the product runs over all embeddings $\iota \colon F\hookrightarrow \C$ with $\iota\in \Psi$ when $F$ is imaginary CM and over all embeddings $\iota \colon F\hookrightarrow \R$ when $F$ is totally real. We consider the double coset $\Gamma_{m,\ell}\gamma \Gamma_{m,\ell^\prime}$ when $\ell^\prime -1 \leq \ell$. Using the Iwahori factorization one easily verifies the decomposition
\[
\Gamma_{m,\ell}\gamma \Gamma_{m,\ell^\prime} = \bigsqcup_i \beta_i \gamma \Gamma_{m,\ell^\prime},
\]
where the $\beta_i$ run through a set of coset representatives of $N(\Zp)^{p^2}$ in $N(\Zp)$. This gives us maps $[K_{m,\ell}\gamma K_{m,\ell^\prime}]^\ast : C_\A^\bullet(K_{m,\ell^\prime},\Lambda) \to C_\A^\bullet(K_{m,\ell},\Lambda)$. When $\ell=\ell^\prime$ we will denote this Hecke operator by $U_p$; since $\gamma$ is central in $M$ the $U_p$-action commutes with the action on $M(\Z/p^m)$. One has the following lemma.

\begin{lemma}\label{hida contraction}
$U_p$ commutes with the inclusion maps $C^\bullet_\A(K_{m,\ell},\Lambda) \to C_\A^\bullet(K_{m^\prime,\ell^\prime},\Lambda)$ for $m\leq m^\prime$, $\ell \leq \ell^\prime$. When $\ell\geq 2$ and $\ell>m$, the image of $U_p$ acting on $C_\A^\bullet(K_{m,\ell},\Lambda)$ is inside $C_\A^\bullet(K_{m,\ell-1},\Lambda)$.
\end{lemma}

\begin{proof}
The first part follows from the formula for $U_p$ since the coset decomposition does not change. The second follows from the same fact, since the formula for $U_p$ and for $[K_{m,\ell-1}\gamma K_{m,\ell}]^\ast : C_\A^\bullet(K_{n,\ell},\Lambda) \to C_\A^\bullet(K_{m,\ell-1},\Lambda)$ are the same.
\end{proof}

Recall that $K_0=K_{0,1}=K^p\Gamma_0(p)$. We may transport the $U_p$-action from $C_\A^\bullet(K_{m,\ell},\Lambda) \cong C_\A^\bullet(K_0,\mathrm{Ind}_{K_{m,\ell}}^{K_0}\Lambda)$ to $C^\bullet(K_0,\mathrm{Ind}_{K_{a,b}}^{K_0}\Lambda)$ via the chosen chain homotopy equivalences; this gives an operator $\wt{U}_p$ which acts as $U_p$ up to homotopy. Tedious but straightforward calculations using the definitions above and the maps in Lemma \ref{pushforward of adelic complexes} show that when $m\leq m^\prime$ and $\ell \leq \ell^\prime$, the diagram 
\[
    \xymatrix{ C^\bullet(K_0,\mathrm{Ind}_{K_{m,\ell}}^{K_0}\Lambda) \ar[r]^-{\wt{U}_p} \ar[d] & C^\bullet(K_0,\mathrm{Ind}_{K_{m,\ell}}^{K_0}\Lambda) \ar[d] \\ C^\bullet(K_0,\mathrm{Ind}_{K_{m^\prime,\ell^\prime}}^{K_0}\Lambda) \ar[r]^-{\wt{U}_p} & C^\bullet(K_0,\mathrm{Ind}_{K_{m^\prime,\ell^\prime}}^{K_0}\Lambda) }
\]
commutes, where the vertical maps are injective, induced by the inclusion $\mathrm{Ind}_{K_{m,\ell}}^{K_0}\Lambda \to \mathrm{Ind}_{K_{m^\prime,\ell^\prime}}^{K_0}\Lambda$, and also that $\wt{U}_p$ commutes with the action of $M(\Z/p^m)$. By finiteness, $\wt{U}_p^{k!}$ stabilizes to an idempotent as $k\to \infty$ as discussed in \S \ref{comm alg}. We will denote the corresponding direct summand of $C^\bullet(K_0,\mathrm{Ind}_{K_{m,\ell}}^{K_0}\Lambda)$ by $C^\bullet(K_0,\mathrm{Ind}_{K_{m,\ell}}^{K_0}\Lambda)^{\mathrm{ord}}$. Since the diagram above commutes, we have inclusions $C^\bullet(K_0,\mathrm{Ind}_{K_{m,\ell}}^{K_0}\Lambda)^\mathrm{ord} \to C^\bullet(K_0,\mathrm{Ind}_{K_{m^\prime,\ell^\prime}}^{K_0}\Lambda)^\mathrm{ord}$ when $m \leq m^\prime$, $\ell\leq \ell^\prime$. Let $V_\infty$ be the smooth induction $\mathrm{sm-Ind}_{N(\Zp)}^{K_0} \Lambda = \varinjlim_{m,\ell}\mathrm{Ind}_{K_{m,\ell}}^{K_0}\Lambda$. $V_\infty$ is an injective smooth $M(\Zp)$-representation over $\Lambda$. We set 
\[
C^\bullet(K_0, V_\infty)^\mathrm{ord} \defeq \varinjlim_{m,\ell}C^\bullet(K_0,\mathrm{Ind}_{K_{m,\ell}}^{K_0}\Lambda)^\mathrm{ord};
\]
this is a direct summand of $C^\bullet(K_0,V_\infty)$ and hence a bounded complex of injective smooth $M(\Zp)$-representations. The following finiteness result is key.

\begin{prop}\label{ordinary finiteness}
The inclusion $C^\bullet(K_0,\mathrm{Ind}_{K_m}^{K_0}\Lambda)^\mathrm{ord} \to \Gamma(K_{M,m,p},C^\bullet(K_0, V_\infty)^\mathrm{ord})$ is a quasi-isomorphism of complexes of $M(\Z/p^m)$-representations.
\end{prop}

\begin{proof}
One directly computes $\Gamma(K_{M,m,p},C^\bullet(K_0, V_\infty)^\mathrm{ord}) = \varinjlim_{\ell\geq m} C^\bullet(K_0,\mathrm{Ind}_{K_{m,\ell}}^{K_0}\Lambda)^\mathrm{ord}$ (the transition maps are injective). Thus, it suffices to show that $C^\bullet(K_0,\mathrm{Ind}_{K_{m,\ell}}^{K_0}\Lambda)^\mathrm{ord} \to C^\bullet(K_0,\mathrm{Ind}_{K_{m,\ell+1}}^{K_0}\Lambda)^\mathrm{ord}$ is a quasi-isomorphism for all $\ell$, i.e that $H^i(\ol{X}_{K_{m,\ell}}^G,\Lambda)^\mathrm{ord} \to H^i(\ol{X}_{K_{m,\ell+1}}^G,\Lambda)^\mathrm{ord}$ is an isomorphism for all $i$. This follows from Lemma \ref{hida contraction} as in the proof of \cite[Lemma 6.10]{khare-thorne}.
\end{proof}

\begin{prop}\label{hida admissibility}
$C^\bullet(K_0, V_\infty)^\mathrm{ord}$ is quasi-isomorphic to a bounded complex of injective admissible $M(\Zp)$-representations. 
\end{prop}

\begin{proof}
Since $C^\bullet(K_0, V_\infty)^\mathrm{ord}$ is a bounded complex of injective smooth $M(\Zp)$-representations, the proof of \cite[Chapter II, \S 5, Lemma 1]{mumford} shows that it suffices to show that its cohomology groups are admissible. Consider the (first quadrant) hypercohomology spectral sequence
\[
H^j_{cts}(K_{M,m,p},H^i(C^\bullet(K_0, V_\infty)^\mathrm{ord})) \implies H_{cts}^{i+j}(K_{M,m,p},C^\bullet(K_0, V_\infty)^\mathrm{ord}).
\]
By Proposition \ref{ordinary finiteness}, the terms in the abutment are finite. It follows that $H^0(C^\bullet(K_0, V_\infty)^\mathrm{ord})$ is admissible. Admissibility of $H^i(C^\bullet(K_0, V_\infty)^\mathrm{ord})$ for $i\geq 1$ then follows by induction, since admissible $M(\Zp)$-representations have finite continuous $K_{M,m,p}$-cohomology.
\end{proof}

In other words, $C^\bullet(K_0, V_\infty)^\mathrm{ord}$ is quasi-isomorphic to a ``good'' complex of the form appearing as the output in Lemma \ref{gluing complexes}(1),
which will allow us (for example) to glue morphisms. We note that this entire discussion applies equally well to boundary cohomology, giving us a complex $C_\partial^\bullet(K_0, V_\infty)^\mathrm{ord}$ quasi-isomorphic to a bounded complex of injective smooth admissible $M(\Zp)$-representations. Taking cones and shifting by $-1$ as above, we obtain ``good'' complexes for compactly supported cohomology as well by Lemma \ref{gluing complexes}(4).

\subsection{Constructing determinants}\label{determinants} We start on the road towards proving Theorem \ref{eliminating nilpotent ideal}. To do so, we will use Chenevier's notion of a determinant~\cite{chenevier} and their basic properties freely. Let $G_{F,S}$ denote the Galois group of the maximal extension of $F$ unramified outside primes above $S$. By \cite[Theorem 2.22(i)]{chenevier} and our non-Eisenstein assumption on the maximal ideal $\mf{m}$, it suffices to construct a $\mathbb{T}^S_M(K_M,\lambda)_\mf{m}^{\mathrm{der}}$-valued continuous determinant of $G_{F,S}$ with the correct characteristic polynomial of Frobenii in order to get the desired Galois representation $ \rho_{\m,r}\colon\mathrm{Gal}(\overline{F}/F)\to \GL_n\left(\mathbb{T}^S_M(K_M,\lambda)_\mf{m}^{\mathrm{der}}\right)$; this is what we will do.

\medskip

Our goal in this subsection is to construct determinants valued in (derived) Hecke algebras for $G$. The results are slight refinements of results from \cite{scholze-galois,newton-thorne}. For this, we start by first recalling some material on Hecke algebras for $G$. Recall the groups $G_0$ that we defined in the beginning of \S \ref{Borel-Serre}. In this subsection, all compact open subgroups $K$ appearing are assumed to be small, and the prime-to-$p$ part $K^p$ is assumed to be fixed throughout. We define the (abstract, spherical) Hecke algebra $\TT_G^S$ by:
\[
\mathbb{T}^S_G \defeq \otimes_{l\not\in S, \bar{w}\mid l}\mathbb{T}_{G,\bar{w}}, \quad \mathbb{T}_{G,\bar{w}} = \Z_p\left[G_0(F^+_{\bar{w}})/ /G_0(\cO_{F^+,\bar{w}})\right],
\]
where the product runs over primes $\bar{w}$ of $F^+$ above a prime $l\not\in S$. Let us recall the explicit description of $\TT_{G,\bar{w}}$ given by the Satake isomorphism. This can be found in
\cite[Lem.\ 5.1.6]{scholze-galois} and in
\cite[Prop.-Def. 5.2]{newton-thorne} for $F$ imaginary CM; we use
a normalization for the Hecke operators that is consistent with the latter reference (but not the former). If $F$ is imaginary CM and $\bar{w}$ splits in $F$,
\[ \TT_{G,\bar{w}}[q_{\bar{w}}^{1/2}] \cong \ZZ_p[q_{\bar{w}}^{1/2}][Y_1^{\pm 1},\dotsc,Y_{2n}^{\pm 1}]^{S_{2n}} \,. \]
We write $T_{G,\bar{w},i}$ for $q^{i(2n-i)/2}_{\bar{w}}$ times the $i$th elementary
symmetric polynomial in $Y_1,\dotsc,Y_{2n}$.
If $F$ is imaginary CM and $\bar{w}$ is inert in $F$,
\[  \TT_{G,\bar{w}}[q_{\bar{w}}^{1/2}] \cong \ZZ_p[q_{\bar{w}}^{1/2}][X_1^{\pm 1},\dotsc,X_{2n}^{\pm 1}]^{S_n \rtimes (\ZZ/2\ZZ)^n} \,. \]
and the unramified endoscopic transfer from
$G(F_{\bar{w}})$ to $\GL_{2n}(F_{\bar{w}})$ is dual to the map
\[ \ZZ_p[q_{\bar{w}}^{1/2}][Y_1^{\pm 1},\dotsc,Y_{2n}^{\pm 1}]^{S_{2n}}
\to \ZZ_p[q_{\bar{w}}^{1/2}][X_1^{\pm 1},\dotsc,X_{2n}^{\pm 1}]^{S_n \rtimes (\ZZ/2\ZZ)^n}
\]
that sends $\{Y_1, \dotsc, Y_{2n} \}$ to $\{X_1^{\pm 1}, \dotsc, X_n^{\pm 1} \}$.
We write $T_{G,\bar{w},i}$ for $q^{i(2n-i)/2}_{\bar{w}}$ times the $i$th elementary
symmetric polynomial in $Y_1,\dotsc,Y_{2n}$.
If $F$ is totally real,
\[ \TT_{G,\bar{w}}[q_{\bar{w}}^{1/2}] \cong \ZZ_p[q_{\bar{w}}^{1/2}][X_1^{\pm 1},\dotsc,X_{n}^{\pm 1}]^{S_n \ltimes (\ZZ/2\ZZ)^n} \]
and the unramified endoscopic transfer from $G(F_{\bar{w}})$ to $\GL_{2n+1}(F_{\bar{w}})$ is dual to the map
\[ \ZZ_p[q_{\bar{w}}^{1/2}][Y_1^{\pm 1},\dotsc,Y_{2n+1}^{\pm 1}]^{S_{2n+1}} \to
\ZZ_p[q_{\bar{w}}^{1/2}][X_1^{\pm 1},\dotsc,X_{n}^{\pm 1}]^{S_n \ltimes (\ZZ/2\ZZ)^n} \]
that sends $\{Y_1,\dotsc,Y_{2n+1}\}$ to $\{X_1^{\pm 1},\dotsc,X_n^{\pm 1},1\}$.
We write $T_{G,\bar{w},i}$ for $q^{i(2n+1-i)/2}_{\bar{w}}$ times the $i$th elementary
symmetric polynomial in $Y_1,\dotsc,Y_{2n+1}$.

\medskip

Recall our complete, algebraically closed extension $C$ of $\overline{\Q}_p$. We fix once and for all embeddings $\overline{\Q}\subset C$ and $\ol{\Q}\sub \C$. Set $A\defeq \cO_C^a/p^r$. The locally symmetric space $X^{G}_K$ is a ``Shimura variety of Hodge-type'' in the sense of the non-standard definition in \cite[\S 4, Introduction]{scholze-galois}. In particular, this gives us a complex structure on $X_K^G$, making it canonically into an algebraic variety over $\C$. This algebraic variety has a canonical model over $\overline{\Q}$, which can be base changed and analytified to an adic space $\cX_K$ over $(C,\cO_C)$. A similar discussion holds for minimal compactifications; we let $\cX^{*}_{K}$ be the minimal compactification of $\cX_K$. We let $\cI\subset \cO_{\cX^*_{K}}$ be the ideal sheaf of the boundary and $\omega_K$ the usual ample line bundle. Following ~\cite[Thm. 4.3.1]{scholze-galois}, we fix some sufficiently divisible $m'\in \Z_{\geq 1}$ and define $\mathbb{T}^S_{cl,c}$ to be $\mathbb{T}^S_G$ equipped with the weakest topology for which all the maps
\[
\mathbb{T}^S_G \to \mathrm{End}_C\left(H^0\left(\cX^*_{K}, \omega_{K}^{\otimes m'k}\otimes \cI\right)\right) 
\] 
are continuous, for varying $k\in \Z_{\geq 1}$ and $K_p$, where the right-hand side is a finite-dimensional $C$-vector space endowed with the $p$-adic topology\footnote{Our $\mathbb{T}^S_{cl,c}$ is denoted by $\mathbb{T}_{cl}$ is \cite{scholze-galois}}. We recall~\cite[Cor. 5.1.11]{scholze-galois} on the existence of determinants for certain quotients of $\mathbb{T}^S_{cl,c}$, but with our choice of normalizations.

\begin{thm}\label{existence of determinant} \leavevmode
\begin{enumerate}%
\item Assume $F$ is imaginary CM. For any continuous quotient $\mathbb{T}^S_{cl,c}\to B$ with $B$ discrete there exists a unique $2n$-dimensional continuous determinant $D$ of $G_{F,S}$ with values in $B$, such that for any prime $w$ of $F$ away from $S$ and above a prime $\bar{w}$ of $F^+$ we have
\[
D(X-\mathrm{Frob}_{w}) = X^{2n} - T_{G, \bar{w}, 1}X^{2n-1}+\dots +q^{n(2n-1)}_w T_{G,\bar{w},2n}.
\]
\item  Assume $F$ is totally real. For any continuous quotient $\mathbb{T}^S_{cl,c}\to B$ with $B$ discrete there exists a unique $2n+1$-dimensional continuous determinant $D$ of $G_{F,S}$ with values in $B$, such that for any prime $\bar{w}$ of $F^+$ away from $S$ we have
\[
D(X-\mathrm{Frob}_{w}) = X^{2n+1} - T_{G, \bar{w}, 1}X^{2n}+\dots -q^{(2n+1)n}_{\bar{w}} T_{G,\bar{w},2n+1}.
\]
\end{enumerate}
\end{thm}

\noindent The main point of this section is Theorem \ref{factors through classical} below, which will be our starting point in the proof of Theorem \ref{eliminating nilpotent ideal}. It is essentially a minor refinement of results of~\cite{newton-thorne}, but we need to use the results of \S \ref{comm alg} to correct an error in the handling of the Hecke algebra there\footnote{This error stems from a mistake in the first version of \cite{scholze-galois}.}. We thank James Newton for his help with the proof below; any mistake is due to the authors.

\medskip

It will be convenient to also work with a (possibly) different topology
on $\mathbb{T}^S_G$.  Let $\iota \colon \mathbb{T}^S_G \to \mathbb{T}^S_G$
be the involution that sends a Hecke operator $[K_{\bar{w}}gK_{\bar{w}}]$ to $[K_{\bar{w}}g^{-1}K_{\bar{w}}]$. In terms of the action of $\TT_G^S$ by correspondences, this corresponds to sending
$X^G_{K^p K_p} \xleftarrow{f} X^G_{K^{p \prime} K_p} \xrightarrow{g}
X^G_{K^p K_p}$ to its opposite
$X^G_{K^p K_p} \xleftarrow{g} X^G_{K^{p \prime} K_p} \xrightarrow{f}
X^G_{K^p K_p}$. 
If $F$ is totally real, then one can use the explicit description
of the Hecke algebra above to check that $\iota$ is
the identity, but $\iota$ is not the identity when $F$ is imaginary CM. Let $\mathbb{T}^S_{cl} \defeq \iota(\mathbb{T}^S_{cl,c})$ as a topological ring\footnote{We emphasize again that $\mathbb{T}_{cl}^S$ is a priori different from the topological ring denoted by $\mathbb{T}_{cl}$ in \cite{scholze-galois}, which is our $\mathbb{T}_{cl,c}^S$.}.

\begin{lem} \label{existence of determinant dual}
Theorem \ref{existence of determinant} holds when $\mathbb{T}^S_{cl,c}$
is replaced by $\mathbb{T}^S_{cl}$.
\end{lem}
\begin{proof}
The argument is essentially the same as that
of \cite[Lemma 4.1]{newton-thorne}.  Suppose $F$ is imaginary CM.
Given a continuous map $\mathbb{T}^S_{cl} \to B$, we can compose with
$\iota$ to get a continuous map $\mathbb{T}^S_{cl,c} \to B$,
and then apply Theorem \ref{existence of determinant} to construct
a determinant $D$.
We have $\iota(T_{G,\bar{w},i})=\frac{T_{G,\bar{w},{2n-i}}}{T_{G,\bar{w},2n}}$ for $0 \le i \le 2n$
(here we take $T_{G,\bar{w},0} = 1$), so $D$ satisfies
\[
D(X - \mathrm{Frob}_{w}) = X^{2n}+\dotsb+(-1)^j q_w^{j(j-1)/2} \frac{T_{G, \bar{w}, {2n-j}}}{T_{G,\bar{w},2n}}  +\dotsb +q^{n(2n-1)}_w T_{G,\bar{w},2n}^{-1} \,.
\]
Now compose this determinant with the antiautomorphism
$B[G_F] \to B[G_F]$ that sends a group-like element $g$ to
$\chi(g)^{2n-1} g^{-1}$, where $\chi$ is the cyclotomic character.
The resulting determinant $D'$ satisfies
\[
D'(X - \mathrm{Frob}_{w}) = X^{2n}+\dotsb+(-1)^j q_{w}^{j(j-1)/2} T_{G, \bar{w}, j}+\dotsb +q^{n(2n-1)}_w T_{G,\bar{w},2n} \,,
\]
which is what we need.

\medskip

In the totally real case, a similar argument applies, or one can just use the above observation that
$\iota$ is the identity.
\end{proof}

We now come to the main theorem of this subsection. Our determinants will be constructed modulo $p^r$ first for every $r$, and then glued together at the end. Recall the subgroups $K_m$ from \S \ref{subsec: ordinary parts}.

\begin{thm}\label{factors through classical} For every $m\geq 0$, the map 
\[
\mathbb{T}^S_{cl}\to \mathrm{End}_{D(M(\Z/p^m), \Lambda)}\left(R\Gamma(X_{K_m}^G,\Lambda)\right)
\] 
is continuous for the discrete topology on the target.
\end{thm}

We will need the following lemma for the proof.

\begin{lem} \label{verdier}
Let $K_p \sub G(\Zp)$ be an open normal subgroup and set $\mathrm{G}=G(\Zp)/K_p$. There is a commutative diagram
\[
\begin{tikzcd}
\mathbb{T}^S_G \arrow[r] \arrow[d,"\iota"] & \mathrm{End}_{D(\mathrm{G},\Lambda)}(R\Gamma_c(X_{K^p K_p}^G,\Lambda)) \arrow[d,dashed] \\
\mathbb{T}^S_G \arrow[r] & \mathrm{End}_{D(\mathrm{G},\Lambda)}(R\Gamma(X_{K^p K_p}^G,\Lambda))
\end{tikzcd}
\]
where the horizontal maps are the usual Hecke actions,
and dashed arrow is the anti-isomorphism induced by Verdier duality.
\end{lem}

\begin{proof}
We claim that it is possible to choose an orientation on
$X^G \times G(\A_f)$
that is preserved by the left $G(\QQ)$- and right $G(\A_f)$-actions.
The group $G(\RR)$ is connected, so any choice of orientation on $X^G$
is preserved by $G(\RR)$, and in particular by $G(\QQ)$.
The orientation can then
be pulled back to $X^G \times G(\A_f)$, proving the claim.

\medskip

We therefore get compatible, $K/K_p$-equivariant choices of orientation on
$X_{K^{p\prime} K_p}^G$ for each $K^{p \prime} \subseteq K^p$.
Hence by \cite[\S 3.3]{kashiwara-schapira}
, we get an isomorphism between
$\Lambda[\dim X^G]$ and the dualizing complex on each of these
manifolds, and this isomorphism is $\mathrm{G}=G(\Zp)/K_p$-equivariant.
Applying $\mathrm{G}$-equivariant Verdier duality (see for example
\cite[\S 3.5]{bernstein-lunts})
gives an anti-isomorphism
\[ \mathrm{End}_{D(\mathrm{G},\Lambda)}(R\Gamma_c(X_{K^p K_p}^G,\Lambda))
\to \mathrm{End}_{D(\mathrm{G},\Lambda)}(R\Gamma(X_{K^p K_p}^G, \Lambda)) \,. \]
which is the dashed arrow in the statement of the lemma.
Verdier duality interchanges pullback and exceptional pushforward,
so it reverses the arrows in Hecke correspondences.
\end{proof}

\begin{proof}[Proof of Theorem \ref{factors through classical}] This proof uses some almost mathematics for $\cO_C$ with respect to its maximal ideal for which we refer to \cite{gabber-ramero}; the almost algebra/module corresponding to an $\cO_C$-algebra/module $M$ will be denoted by $M^a$. Assume first that $F$ is imaginary CM. Lemma 5.23 of \cite{newton-thorne}, together with the comparison results of Lemmas 5.15 to 5.18 of~\cite{newton-thorne}, gives a bounded complex $\check{C}$ of smooth $G(\Z_p)$-representations with $\cO_C/p^r$-coefficients, with an action of $\mathbb{T}^S_{cl,c}$ on each term and with an isomorphism
\[
\check{C}^a   \toisom R\Gamma_{\ol{X}_{K^p}^G}(j_! \Lambda) \otimes_{\Lambda} (\cO_C/p^r)^a
\]
in $D_{\mathrm{sm}}(G(\Z_p),(\cO_C/p^r)^a)$ (we refer to \cite[\S 2.6]{newton-thorne} for the notion of smooth representations over $(\cO_C/p^r)^a$). Here $R\Gamma_{\ol{X}_{K^p}^G}(j_!\Lambda)$ denotes a complex computing compactly supported completed cohomology on the tower
\[ 
X_{K^p}^G = \varprojlim_{K_p}X_{K^p K_p}^G \overset{j}{\longrightarrow} \varprojlim_{K_p}\ol{X}_{K^p K_p}^G = \ol{X}_{K^p}^G ,
\]
constructed by the formalism of \cite[\S 2.5]{newton-thorne}. The topology on $\TT^S_{cl,c}$ is defined so that the termwise action on $\check{C}$ is continuous; see \cite[Proof of Theorem 4.3.1]{scholze-galois}. Recall the exact functor $(\ )_{!}$ from $\cO_{C}^{a}$-modules to $\cO_{C}$-modules from \cite[\S 2.2.21]{gabber-ramero}. Applying it, we obtain an isomorphism
\[
\check{C}\otimes_{\cO_C/p^r}\m_C/p^r\m_C \toisom R\Gamma_{\ol{X}_{K^p}^G}(j_! \Lambda) \otimes_{\Lambda} \m_C/p^r \m_C
\]
in $D_{\mathrm{sm}}(G(\Z_p),\cO_C/p^r)$, which we may think of as an isomorphism in $D_{\mathrm{sm}}(G(\Z_p), \Lambda)$
via the forgetful functor.  Note that $\mathbb{T}^S_{cl,c}$ still acts continuously on each term of $\check{C}\otimes_{\cO_C/p^r}\m_C/p^r\m_C$. By Lemma~\ref{injective resolution with continuous action} applied to $\mathbb{T}_c\defeq \mathbb{T}^S_{cl,c}\otimes_{\Z_p}\Lambda$ with the induced topology, we can replace $\check{C}\otimes_{\cO_C/p^r}\m_C/p^r\m_C$ with a quasi-isomorphic bounded below complex $I^\bullet$ of injective objects in $\mathrm{Mod}_{\mathrm{sm}}(G(\Z_p), \Lambda)$ with a continuous $\mathbb{T}_c$-action on each term. Let $K_p\subseteq G(\Z_p)$ be a normal open subgroup. We have a sequence of isomorphisms
\[
R\Gamma_c(X_{K^p K_p}^G,\Lambda) \otimes_{\Lambda}\m_C/p^r\m_C\toisom R\Gamma\left(K_p, R\Gamma_{\ol{X}_{K^p}^G}(j_! \Lambda)\right)\otimes_{\Lambda}\m_C/p^r\m_C
\]
\[
\toisom R\Gamma\left(K_p, R\Gamma_{\ol{X}_{K^p}^G}(j_! \Lambda)\otimes_{\Lambda}\m_C/p^r\m_C \right) \toisom \Gamma(K_p,I^\bullet)
\]
in $D(G(\Z_p)/K_p, \Lambda)$; the first is \cite[Lemma 2.39]{newton-thorne}, the second is \cite[Lemma 2.38]{newton-thorne}, and the third follows from the definition of $I^\bullet$ and the discussion above. 

\medskip

Set $\mathrm{G}\defeq G(\Z_p)/K_p$. Since $\Gamma(K_p, I^\bullet)$ is a $\mathbb{T}_c$-stable subcomplex of $I^\bullet$, it still has a continuous action of $\mathbb{T}_c$ on each term. We apply Lemma~\ref{injective resolution with continuous action} again, for the finite group $\mathrm{G}$, in order to replace $\Gamma(K_p,I^\bullet)$ with a quasi-isomorphic bounded below complex of injective $\Lambda[\mathrm{G}]$-modules $J^\bullet$ with a continuous action of $\mathbb{T}_c$ on each term\footnote{Applying Lemma \ref{injective resolution with continuous action} twice in this proof is necessary: The first application sets up the possibility of applying it to $R\Gamma_c(X_{K^p K_p}^G,\Lambda) \otimes_{\Lambda}\m_C/p^r\m_C$.}. Now $R\Gamma_c(X_{K^p K_p}^G,\Lambda)$ is isomorphic to a bounded complex of finite projective $\Lambda[\mathrm{G}]$-modules $P^\bullet$, by the discussion in \S \ref{subsec: ordinary parts}.  Choose $\alpha\in \m_C\setminus p\m_C$. We have an injection of $\Lambda$-modules
\[
\Lambda\alpha \hookrightarrow \m_C/p^r\m_C
\] 
and, since $\Lambda$ is self-injective, this admits a splitting. Thus, $P^\bullet\otimes_{\Lambda}\Lambda\alpha$ is a $\mathbb{T}_c$-equivariant direct summand of $P^\bullet \otimes_{\Lambda}\m_C/p^r\m_C\simeq J^\bullet$ in $D(\mathrm{G},\Lambda)$.
By Proposition~\ref{getting rid of C}, the map 
\[
\mathbb{T}_c \to \mathrm{End}_{D(G,\Lambda)}(P^\bullet) = \mathrm{End}_{D(G,\Lambda)}(R\Gamma_c(X_{K^p K_p}^G,\Lambda))
\]
is continuous for the discrete topology on the target.
By Lemma \ref{verdier},
\[
\mathbb{T}\to \mathrm{End}_{D(G,\Lambda)}(P^\bullet) = \mathrm{End}_{D(G,\Lambda)}(R\Gamma(X_{K^p K_p}^G,\Lambda))
\]
is also continuous for the discrete topology on the target,
where $\mathbb{T} \defeq \mathbb{T}^S_{cl} \otimes_{\Zp} \Lambda$.

\medskip

Now take $K_p=\Gamma(p^m)\defeq \mathrm{Ker}(G(\Zp) \to G(\Z/p^m))$ when $m\geq 1$. We have a $\mathbb{T}^S_{cl}$-equivariant isomorphism
\[
R\Gamma(X_{K_m}^G,\Lambda)\toisom R\Gamma\left(N(\Z/p^m\Z), R\Gamma(X_{K^p \Gamma(p^m)}^G,\Lambda) \right)
\]
in $D(M(\Z/p^m\Z),\Lambda)$, by (an equivariant version of) Proposition \ref{finite to finite basic}. The map in the statement of the theorem
factors as 
\[
\mathbb{T}^S_{cl}\to \mathrm{End}_{D(G(\Z/p^m\Z),\Lambda)}\left(R\Gamma(X_{K^p \Gamma(p^m)}^G,\Lambda)\right)\to \mathrm{End}_{D(M(\Z/p^m\Z), \Lambda)}\left(R\Gamma(X_{K_m}^G,\Lambda)\right),
\]
where we have seen above that the first map is continuous and the second map is continuous as both the source and the target have the discrete topology. This finishes the argument when $m\geq 1$; for $m=0$ one chooses $K_p=\Gamma(p)$ and argues in the same way using that $\Gamma_0/\Gamma(p) \cong P(\Z/p)$.

\medskip

This finishes the proof when $F$ is imaginary CM. The argument when $F$ is totally real is the same; the argument to prove of~\cite[Lemma 5.23]{newton-thorne} for $\mathrm{Res}_{F/\Q}\mathrm{Sp}_{2n}$-Shimura varieties is exactly the same, since \cite[\S 4]{scholze-galois} applies to them.
\end{proof}

\subsection{The boundary of the Borel--Serre compactification}\label{subsec: boundary}

Our task now is to prove a relation between the Hecke actions on the cohomology of $X^G_K$ and $X_{K_M}^M$. We begin with a quick recap of the method of \cite{newton-thorne}, which may be viewed as a derived version of that of \cite{scholze-galois}. The link between $X^G_K$ and $X_{K_M}^M$ is through the boundary $\partial \ol{X}^G_K$ of the Borel--Serre compactification. There is an open subset $X_K^{G,P} \sub \partial \ol{X}_K^G$ and, if $K_P \defeq K \cap P(\A_f)$ and $K_M \defeq K \cap M(\A_f)$, maps
\[
 X_{K_M}^M \twoheadleftarrow X_{K_P}^P \hookrightarrow X_K^{G,P}.
\]
The surjective map is a torus fibration and the injective map is an open and closed immersion, and one may use this to realize the cohomology of  $X_{K_M}^M$ as a direct summand of the cohomology of $X_K^{G,P}$. After localization at a maximal ideal, the cohomologies of $X_K^{G,P}$ and $\partial \ol{X}_K^G$ are equal, so it remains  to relate the cohomology of $X_K^{G}$ and $\partial \ol{X}_K^G$. In \cite{newton-thorne} this is done via the excision sequence. This is the source of the nilpotent ideal in their method, but they also note that a suitable vanishing results could remove this (cf. Theorem 1.4 of \emph{loc.\ cit.}).

\medskip

In light of this, we will use our Theorem \ref{main thm non-similitude} to remove the nilpotent ideal. Since our result is at \emph{infinite level}, we need an analysis along the above lines in which we have compatibility as the level changes, but this is not the case for the method outlined above (the main issue is the realization of the cohomology of  $X_{K_M}^M$ as a direct summand of the cohomology of $X_K^{G,P}$). To remedy this we take ordinary parts from $G$ (or $P$) to $M$ (as discussed in \S \ref{subsec: ordinary parts}). This simultaneously solves the compatibility problem and cuts down the cohomology groups to objects with good finiteness properties in the limit, which is useful for the technical aspects of the analysis. The resulting method is the main invention of this section; the end result is Theorem \ref{direct summand}. Beyond the applications in this paper, we believe that our method will be useful in reducing questions of local-global compatibility from $M$ to $G$. 

\medskip

Let us now turn to the mathematics, starting with a discussion of the structure of $\partial \ol{X}_K^G$. As a set, 
\[
\ol{X}^G = \bigsqcup_Q X^Q,
\]
where $Q$ runs through the parabolics of $G$ defined over $\Q$. For any parabolic $Q$, the subset $X(Q) \defeq \bigsqcup_{Q^\prime \supseteq Q} X^{Q^\prime}$ is open in $\ol{X}^G$. Since the Siegel parabolic $P$ is maximal, it follows that $X^P$ is open inside $\partial \ol{X}^G = \ol{X}^G \setminus X^G$. We denote by $X^{G,P}$ the union of all translates of $X^P$ inside $\partial \ol{X}^G$ by the action of $G(\Q)$; as a set this is the disjoint union of $X^Q$ where $Q$ ranges over the $\Q$-parabolic subgroups of $G$ which are $\Q$-conjugate to $P$. If $K$ is a neat level, we obtain an open subset
\[
X^{G,P}_K \defeq G(\Q) \backslash X^{G,P} \times G(\A_f)/K \sub \partial \ol{X}^G_K,
\]
and one sees that $X^{G,P}_K= P(\Q) \backslash X^{P} \times G(\A_f)/K$. Note that if $K^\prime \sub K$, then $X^{G,P}_{K^\prime} = X^{G,P}_K \times_{\ol{X}^G_K} \ol{X}^G_{K^\prime}$. For each $g\in  G(\A_f)$ we have an open and closed immersion
\[
\iota_g : P(\Q) \backslash X^{P} \times P(\A_f)/(gKg^{-1} \cap P(\A_f)) \to P(\Q) \backslash X^{P} \times G(\A_f)/K
\]
given by $\iota_g(x,p) = (x,pg)$. Two maps $\iota_g$ and $\iota_h$ have equal image if $h\in P(\A_f)gK$, and otherwise the images are disjoint. We will think of $X^P_{K\cap P(\A_f)}$ as an open and closed subset of $X^{G,P}_K$ via $\iota_1$. To simplify notation, we will simply write $X^P_K$ for $X^P_{K\cap P(\A_f)}$. Note that if $K^\prime \sub K$, then inclusion $X^P_{K^\prime} \sub X^P_K \times_{X^{G,P}_K}X^{G,P}_{K^\prime}$ is not an equality in general; this is the ``failure of compatibility'' when changing $K$ that we will have to work around later.

\medskip

Our first goal is Proposition \ref{XP to XGP}, relating $R\Gamma(X^P_K,\Lambda)$ and $R\Gamma(X^{G,P}_K,\Lambda)$ and their Hecke actions after taking ordinary parts, in a way that is compatible with changing $K$ (for certain choices of $K$). $X_K^P$ is acted on by $\cH(P(\A_f),K\cap P(\A_f))$ and $X^{G,P}_K$ is acted on by $\cH(G(\A_f),K)$. Recall $P_0\subset G_0$ defined in the beginning of \ref{Borel-Serre}. Inside $\cH(P(\A_f),K\cap P(\A_f))$ we have
the abstract spherical Hecke algebra $\mathbb{T}^S_P$, defined by 
\[
\mathbb{T}^S_P \defeq \otimes_{l\not\in S, \bar{w}\mid l}\mathbb{T}_{P,\bar{w}}, \quad \mathbb{T}_{P,\bar{w}} = \Z_p\left[P_0(F^+_{\bar{w}})/ /P_0(\cO_{F^+,\bar{w}})\right].
\]
$\mathbb{T}_P^S$ is related to its counterpart $\mathbb{T}_G^S$ for $G$ via the homomorphism $r_P \colon \mathbb{T}^S_G\to \mathbb{T}^S_{P}$, which restricts a function on $G(\A_f^S)$ to $P(\A_f^S)$. For further discussion of $r_P$ and a proof that it is a homomorphism we refer to \cite[\S 2.2.3]{newton-thorne}.

\medskip

We now recall some notation from \S \ref{subsec: ordinary parts}. We recall the notation $K_{m,\ell} = K^p \Gamma_{m,\ell}$ (for $\ell\geq m$, $\ell \geq 1$, $m\geq 0$) and $K_m = K_{m,m} = K^p \Gamma_1(p^m)$, $K_0 = K^p\Gamma_0 = K^p \Gamma_0(p)$. Also recall the matrix
\[
\gamma \defeq \prod_{\iota} \mathrm{diag}(p,\dots,p,p^{-1},\dots,p^{-1})\in M(\Q_p),
\]
where in each matrix $p$ and $p^{-1}$ occur $n$ times on the diagonal and the product runs over all embeddings $\iota \colon F\hookrightarrow \C$ with $\iota\in \Psi$ when $F$ is imaginary CM and over all embeddings $\iota \colon F\hookrightarrow \R$ when $F$ is totally real. We have the decomposition
\[
\Gamma_{m,\ell} \gamma \Gamma_{m,\ell} = \bigsqcup_i \beta_i \gamma \Gamma_{m,\ell},
\]
where the $\beta_i$ run through a set of coset representatives of $p^2 N(\Zp)$ in $N(\Zp)$. We will also use the double coset $(P(\Qp) \cap \Gamma_{m,\ell}) \gamma (P(\Qp) \cap \Gamma_{m,\ell})$, which defines an operator in $\cH(P(\A_f),K \cap P(\A_f))$. We have the decomposition
\[
(P(\Qp) \cap \Gamma_{m,\ell}) \gamma (P(\Qp) \cap \Gamma_{m,\ell}) = \bigsqcup_i \beta_i \gamma (P(\Qp) \cap \Gamma_{m,\ell}),
\] 
again independent of $m$ and $\ell$. For simplicity, we will denote both Hecke operators/correspondences by $U_p$. In \S \ref{subsec: ordinary parts} we discussed ordinary parts for $G$ with respect to $U_p$; the same discussion holds for $P$ and $U_p$ (in fact everything is easier, since $X^P_{K_m}$ is the quotient of $X_{K_{m+1}}^P$ by $K_{M,m,p}$). Since the action of $U_p$ on $R\Gamma(X^P_{K_{m,\ell}},\Lambda)$ commutes with the action of $\mathbb{T}_P^S$, $R\Gamma(\XP,\Lambda)^{\mathrm{ord}}$ inherits an action of $\mathbb{T}_P^S$ as discussed in \S \ref{comm alg}. Similar remarks apply for $G$.

\begin{prop} \label{XP to XGP}
$R\Gamma(\XP,\Lambda)^{\ord}$ is a direct summand
of $R\Gamma(\XGP,\Lambda)^{\ord}$ in $D(M(\Z/p^{m}), \Lambda)$, compatibly with changing the level and Hecke-equivariant for the restriction of functions $r_P \colon \mathbb{T}^S_G \to \mathbb{T}^S_P$.
\end{prop}

Let us spell out what compatibility when changing the level means. When $\ell \geq m$, we have canonical isomorphisms 
\[
R\Gamma(\XP,\Lambda)^{\ord}\cong R\Gamma(K_{M,m,p}/K_{M,\ell,p}, R\Gamma(\XP[\ell],\Lambda)^\mathrm{ord})
\]
and 
\[
R\Gamma(\XGP,\Lambda)^{\ord}\cong R\Gamma(K_{M,m,p}/K_{M,\ell,p}, R\Gamma(\XGP[\ell],\Lambda)^\mathrm{ord})
\]
in $D(M(\Z/p^{m}), \Lambda)$, by Proposition \ref{finite to finite basic} and Proposition \ref{ordinary finiteness}. In general, let $\epsilon_i$ be the idempotent on $R\Gamma(\XGP[i],\Lambda)^\mathrm{ord}$ cutting out $R\Gamma(\XP[i],\Lambda)^\mathrm{ord}$. Compatibility then means that $\epsilon_m = R\Gamma(K_{M,m,p}/K_{M,\ell,p},e_\ell)$ under the isomorphisms above. This will allows us to get idempotents in the limit later, using Lemma \ref{gluing complexes}. The same notion of compatibility, with the obvious modifications, will be used in Propositions \ref{XM to XP} and \ref{XG to XGP} later.

\begin{proof}
To break up the proof into parts, it will be convenient for us to define an intermediate space $Y^{G,P}_K$ by
\[ 
Y^{G,P}_K \defeq P(\QQ) \backslash X^P \times P(\A_f^p) G(\Qp) / (K \cap P(\A_f^p) G(\Qp)),. 
\]
for $K =K^p K_p$ compact open. Then $Y^{G,P}_K$ is an open and closed subset of $X^{G,P}_K$. Compared to $X^P_K$, it has the advantage that $Y^{G,P}_{K^\prime} = Y^{G,P}_K \times_{X^{G,P}_K} X^{G,P}_{K^\prime}$ whenever $K^\prime = K^p K_p^\prime \sub K$. Note that $\cH(G(\Qp),\Gamma_{m,\ell})$ and $\cH(P(\A_f^p),K^p \cap P(\A_f^p))$ act on $Y^{G,P}_{K_{m,\ell}}$. In particular, it has a $U_p$-action. Since the decomposition of the double coset of $\gamma$ into single right cosets is identical for $G$ and $P$, the $U_p$-actions on the spaces $X_{K_{m}}^P$, $Y_{K_{m}}^{G,P}$ and $X_{K_{m}}^{G,P}$ commute with pullback along the open and closed immersions $X_{K_m}^P \sub Y_{K_m}^{G,P} \sub X_{K_m}^{G,P}$.

\medskip

We begin by exhibiting $R\Gamma(\XP, \Lambda)^\mathrm{ord}$ as a direct summand of $R\Gamma(\YGP,\Lambda)^\mathrm{ord}$ compatibly with changing $m$.
Let $i_m \colon \XP \hookrightarrow \YGP$ be the natural open and closed
immersion.  Let $e_m$ be the ordinary projector on $R\Gamma(\YGP,\Lambda)$.
Then by the commutativity of $U_p$ and the pullback functor $i_m^*$
and since $i_m^*\circ i_{m*}$ is equal to the identity on $R\Gamma(\XP, \Lambda)$,
we have $i_m^*\circ U_p\circ i_{m*} = U_p$. We conclude that
$i_m^* \circ e_m \circ i_{m*}$ is the ordinary projector on
$R\Gamma(\XP,\Lambda)$.
For $1 \le k \le m$, let $r_{k m}$ be the idempotent on
$R\Gamma(\YGP,\Lambda)$ corresponding to the open and closed immersion
$X^P_{K_{0}} \times_{Y^{G,P}_{K_{0,k}}} Y^{G,P}_{K_m} \hookrightarrow Y^{G,P}_{K_m}$.

Consider the maps
\[ \begin{tikzcd}[column sep=6em]
R\Gamma(\XP, \Lambda) \arrow[r,"e_m \circ r_{1m} \circ e_m \circ i_{m*}",shift left=.6ex] & \arrow[l,"i_m^*",shift left=.6ex] R\Gamma(\YGP,\Lambda)\,.
\end{tikzcd} \]
We claim that $i_m^* \circ (e_m \circ r_{1m} \circ e_m \circ i_{m*})$
is the ordinary projector $i_m^* \circ e_m \circ i_{m*}$, giving
a realization of $R\Gamma(\XP, \Lambda)^{\ord}$ as a direct summand of 
$R\Gamma(\YGP,\Lambda)^{\ord}$.
By commutativity of $U_p$ and the pullback functor $i_m^*$, we have
$i_m^* \circ e_m = i_m^* \circ e_m \circ r_{mm}$.
So
\begin{IEEEeqnarray*}{c}
i_m^* \circ (e_m \circ r_{1m} \circ e_m \circ i_{m*}) =
i_m^* \circ e_m \circ r_{mm} \circ r_{1m} \circ e_m \circ i_{m*} \\
= i_m^* \circ e_m \circ r_{mm} \circ e_m \circ i_{m*} = i_m^* \circ e_m \circ e_m \circ i_{m*} = i_m^* \circ e_m \circ i_{m*} \,.
\end{IEEEeqnarray*}

It remains to check that the maps $i_m^*$ and
$e_m \circ r_{1m} \circ e_m \circ i_{m*}$ are compatible with change of
level.  Let $\ell > m$.
Observe that
\[ R\Gamma(K_{M,m,p}/K_{M,\ell,p},\Gamma(Y^{G,P}_{K_{\ell}}, \Lambda)) \cong R\Gamma(Y^{G,P}_{K_{m,\ell}}, \Lambda) \]
in $D(M(\Z/p^{m}), \Lambda)$ by the argument of Proposition \ref{finite to finite basic}.
When we take the quotient by $K_{M,m,p}/K_{M,\ell,p}$,
the map $i_{\ell}$ descends to the natural open and closed immersion
$i_{m\ell} \colon \XP \hookrightarrow Y^{G,P}_{K_{m,\ell}}$,
$e_{\ell}$ descends to the ordinary projector $e_{m\ell} \in \End R\Gamma(Y^{G,P}_{K_{m,\ell}},\Lambda)$,
and $r_{k\ell}$ descends to the idempotent $r_{km\ell} \in \End R\Gamma(Y^{G,P}_{K_{m,\ell}},\Lambda)$
corresponding to
the open and closed immersion $X^P_{K_0} \times_{Y^{G,P}_{K_{0,k}}} Y^{G,P}_{K_{m,\ell}} \hookrightarrow Y^{G,P}_{K_{m,\ell}}$.  Let $f_{\ell m} \colon Y^{G,P}_{K_{m,\ell}} \to Y^{G,P}_{K_m}$ denote the projection map.

It is straightforward to check that the maps
$i_m^*$, $e_m$, and $r_{1m}$ are
compatible with change of level in the sense that
$i_m^* = i_{m\ell}^* \circ f_{\ell m}^*$, $f_{\ell m}^* \circ e_m = e_{m \ell} \circ f_{\ell m}^*$,
$f_{\ell m}^* \circ r_{1m} = r_{1m \ell} \circ f_{\ell m}^*$.  However, $i_{m*}$ is not compatible with the
change of level; instead $i_{m\ell*} = r_{\ell m\ell} \circ f_{\ell m}^* \circ i_{m*}$.

To prove compatibility of the direct sum decompositions, we need to show that
\[ f_{\ell m}^* \circ e_m \circ r_{1m} \circ e_m \circ i_{m*} = e_{m \ell} \circ r_{1m\ell} \circ e_{m\ell} \circ i_{m\ell *} \,. \]
Using the above identities, we find that the left-hand side equals
$e_{m \ell} \circ r_{1m \ell} \circ e_{m \ell} \circ f_{\ell m}^* \circ i_{m*}$,
while the right-hand side equals
$e_{m \ell} \circ r_{1m \ell} \circ e_{m \ell} \circ r_{\ell m\ell} \circ f_{\ell m}^* \circ i_{m*}$.
So it suffices to show that
$r_{1m \ell} \circ e_{m \ell} \circ r_{\ell m \ell} = r_{1m \ell} \circ e_{m\ell}$.
Then it suffices to show that for large enough $k$, in the
$U_p^k$-correspondence
\[ Y^{G,P}_{K_{m,\ell}} \xleftarrow{\gamma^k} Y^{G,P}_{K_{m,\ell}\cap \gamma^k K_{m,\ell}\gamma^{-k}} \to Y^{G,P}_{K_{m,\ell}} \,, \]
the preimage of $X^P_{K_{0}} \times_{Y^{G,P}_{K_{0,\ell}}} Y^{G,P}_{K_{m,\ell}}$
along the first map contains the preimage
of $X^P_{K_0} \times_{Y^{G,P}_{K_{0,1}}} Y^{K,P}_{K_{m,\ell}}$ along the second map.
This is equivalent to the statement
\[ P(\Qp) \Gamma_{0,1} \gamma^k \subseteq P(\Qp) \Gamma_{0,\ell}, \]
which holds for $k \ge \left\lceil \frac{\ell-1}{2} \right\rceil$
by Lemma \ref{contract}.

\medskip

This finishes the proof that the direct summand $R\Gamma(\XP, \Lambda)^\mathrm{ord}$ of $R\Gamma(\YGP,\Lambda)^\mathrm{ord}$ is compatible with changing $m$. Showing that the direct summand $R\Gamma(\YGP,\Lambda)^\mathrm{ord}$ of $R\Gamma(\XGP,\Lambda)^\mathrm{ord}$ is compatible with changing $m$
is easier; it follows from $U_p$-equivariance of pullback and pushforward along all inclusions, since in this case we have $Y^{G,P}_{K_{m,\ell}} = \YGP \times_{\XGP} X^{G,P}_{K_{m,\ell}}$.

\medskip

It remains to check that our realization of $R\Gamma(\XP,\Lambda)^{\ord}$ as a direct summand of $R\Gamma(\XGP,\Lambda)^\mathrm{ord}$ is compatible with $r_P: \mathbb{T}^S_G \to \mathbb{T}^S_P$. It is immediate that $R\Gamma(\XP,\Lambda)^\mathrm{ord}$
is a $\mathbb{T}^S_P$-equivariant direct summand of $R\Gamma(\YGP,\Lambda)^\mathrm{ord}$, since $\XP \sub \YGP$ is $\mathbb{T}_P^S$-stable and $U_p$ (and hence the ordinary projector) with commutes $\mathbb{T}^S_P$. So it remains to check that the realization of $R\Gamma(\YGP,\Lambda)^\mathrm{ord}$ as a direct summand of $R\Gamma(\XGP,\Lambda)^\mathrm{ord}$ is compatible with $r_P$, which we may do before applying ordinary projectors. 

\medskip

For this we will compute Hecke actions by the methods of \S \ref{subsec: loc sym space}. $R\Gamma(\XGP,\Lambda)$ with its $\mathbb{T}_G^S$-action is computed by the object $R\Gamma(K_m, R\Gamma(\frakX^{G,P},\Lambda))\in D(\mathbb{T}_G^S \times M(\Z/p^m))$, where $\frakX^{G,P} \defeq P(\Q)\backslash X^P \times G(\A_f)$. $R\Gamma(\YGP,\Lambda)$ can be viewed in two ways. On its own, its $\mathbb{T}_P^S$-action is naturally computed by $R\Gamma(K_m\cap P(\A_f), R\Gamma(\mf{Y}_1^{G,P},\Lambda))\in D(\mathbb{T}_P^S \times M(\Z/p^r))$, where $\mf{Y}_1^{G,P} \defeq P(\Q)\backslash X^P \times P(\A_f^p)G(\Qp)$. However, viewing it as direct summand of $R\Gamma(\XGP,\Lambda)$, it corresponds to the direct summand $R\Gamma(K_m, R\Gamma(\mf{Y}_2^{G,P},\Lambda))$ of $R\Gamma(K_m, R\Gamma(\frakX^{G,P},\Lambda))$, where 
\[
\mf{Y}_2^{G,P} \defeq \YGP \times_{\XGP} \mf{X}^{G,P} = P(\Q) \backslash X^P \times G(\A_f^S)P(\A_S)K_{m,S} G(\Qp)
\]
(the equality on the right follows by direct computation, noting that $G(\A_f^S)=P(\A_f^S)K^S$). It remains to show that $R\Gamma(K_m, R\Gamma(\mf{Y}_2^{G,P},\Lambda))\cong r_P^\ast R\Gamma(K_m\cap P(\A_f), R\Gamma(\mf{Y}_1^{G,P},\Lambda))$, where $r_P^\ast$ denotes restriction of scalars via $r_P$. The right hand side is isomorphic to 
\[
r_P^\ast R\Gamma\left( K_P^S, R\Gamma \left( K_{P,S}, R\Gamma(\mf{Y}_1^{G,P},\Lambda)\right) \right),
\]
where $K_P^S = K^S \cap P(\A_f^S)$ and $K_{P,S} = K_{m,S} \cap P(\A_S)G(\Qp)$. This is in turn is isomorphic to 
\[
R\Gamma\left( K^S, \mathrm{Ind}_{P(\A_f^S)}^{G(\A_f^S)} R\Gamma \left( K_{P,S}, R\Gamma(\mf{Y}_1^{G,P},\Lambda)\right) \right)
\]
by \cite[Corollary 2.6]{newton-thorne}. It is then straightforward to show that $R\Gamma(K_{P,S},R\Gamma(\mf{Y}_1^{G,P},\Lambda)) = H^0(P(\Q)\backslash P(\A_f^p)G(\Qp)/K_{P,S}, \Lambda)$, using that $P(\Q)\backslash P(\A_f^p)G(\Qp)$ is a free $K_{P,S}$-space\footnote{Also recall that in these computations, all groups are given the discrete topology.}. One then writes $R\Gamma(K_m, R\Gamma(\mf{Y}_2^{G,P})) \cong R\Gamma(K^S, R\Gamma(K_{m,S},R\Gamma(\mf{Y}_2^{G,P})))$. As above, one sees that
\[
R\Gamma(K_{m,S},R\Gamma(\mf{Y}_2^{G,P},\Lambda)) = H^0(P(\Q)\backslash G(\A_f^S)P(\A_S)K_{m,S}/K_{m,S}, \Lambda),
\]
and the proof is complete upon noting (with a little bit of computation) that the right hand side is isomorphic to $\mathrm{Ind}_{P(\A_f^S)}^{G(\A_f^S)}H^0(P(\Q)\backslash P(\A_f^p)G(\Qp)/K_{P,S}, \Lambda)$.
\end{proof}

The following lemma was used in the proof of Proposition \ref{XP to XGP}.
\begin{lemma}\label{contract}~
For any integer $k \ge 0$, we have $P(\Qp) \Gamma_{0,1} \gamma^k = P(\Qp) \Gamma_{0,1+2k}$.
\end{lemma}
\begin{proof}
The Iwahori factorization gives
$\Gamma_{0,1+2k} = N(\Zp)K_{M,0,p}\ol{N}(\Zp)^{p^{1+2k}}$.
We then compute
\begin{IEEEeqnarray*}{c}
P(\Qp) \Gamma_{0,1} \gamma^k
= P(\Qp) \ol{N} (\Zp)^p \gamma^k
= P(\Qp) \gamma^k \ol{N} (\Zp)^{p^{1+2k}} = \\
= P(\Qp) \ol{N}( \Zp)^{p^{1+2k}}
= P(\Qp) \Gamma_{0,1+2k} \,.
\end{IEEEeqnarray*}
\end{proof}

Our next goal is to compute $R\Gamma(\XP,\Lambda)^\mathrm{ord}$, which will be done in Proposition \ref{XM to XP}. To do so, it will be desirable to understand the ordinary projector in the context of the Hecke action computations of \S \ref{subsec: loc sym space}. For this we need some preparation. The quotient map $P \to M$ induces a map $\Phi : \mf{X}^P \to \mf{X}^M$, which is $P(\A_f)$-equivariant when $\mf{X}^M$ is given the $P(\A_f)$-action obtained by inflation from $M(\A_f)$, which we do from now on. For now, let us simplify the notation by omitting $m$; We set $K_{P} = P(\A_f)\cap K_m$, $K_{M} = M(\A_f)\cap K_m$ and $K_{N} = N(\A_f)\cap K_m$. For any $P(\A_f)$-space $\mf{S}$ on which $N(\A_f)$ acts trivially, there is a derived functor
\[
R\Gamma(K_{N,S},-) : D_{P(\A_f)}(\mf{S},\Lambda) \to D_{P(\A_f^S)\times K^\prime_{M,S}}(\mf{S},\Lambda),
\]
of $K_{N,S}$-invariants, where $K^\prime_{M,S} \defeq K_{M,S}M(\Zp)$. The examples of such spaces $\mf{S}$ that we will use are the point $pt$ and $\mf{X}^M$. We set
\[
Z_M^+ \defeq \{m\in Z_M(\Qp) \mid mN(\Zp)m^{-1} \sub N(\Zp) \},
\]
where $Z_M \sub M$ is the center. Note that $Z_M^+$ is a monoid, and that $\gamma \in Z_M^+$. If $V \in \mathrm{Sh}_{P(\A_f)}(\mf{S},\Lambda)$, we may define a homomorphism  
\[
Z_M^+ \to \End_{\mathrm{Sh}_{P(\A_f^S)\times K^\prime_{M,S}}(\mf{S},\Lambda)}(\Gamma(K_{N,S},V))
\]
of multiplicative monoids using the formula $m.v = \sum n_i mv $, where the sum runs over a set of coset representatives $n_i$ of $N(\Zp)/mN(\Zp)m^{-1}$. We refer to \cite[Definition 3.1.3, Lemma 3.1.4]{emerton-ord1} and the surrounding discussion for a proof that this is well defined and does give a homomorphism of monoids. The definition naturally extends to complexes and gives us an action of $Z_M^+$ on $R\Gamma(K_{N,S},V)$ for any $V \in D_{P(\A_f)}(\mf{S},\Lambda)$. Consider the $P(\A_f)$-equivariant map $\pi : \mf{X}^M \to pt$. We have a diagram 
\[
    \xymatrix@C+2pc{ D(P(\A_f),\Lambda) \ar[r]^-{R\Gamma(K_{N,S},-)} \ar[d]^{\pi^\ast} & D(P(\A_f^S)\times K_{M,S}^\prime, \Lambda ) \ar[d]^{\pi^\ast} \\ D_{P(\A_f)}(\mf{X}^M,\Lambda) \ar[r]^-{R\Gamma(K_{N,S},-)} & D_{P(\A_f^S)\times K_{M,S}^\prime}(\mf{X}^M,\Lambda) }
\]
which commutes up to natural isomorphism, where we have identified $D_H(pt,\Lambda)$ with $D(H,\Lambda)$ for any group $H$ (note that the standard injectives $\mathrm{Fun}(P(\A_f),M) \in \mathrm{Mod}(P(\A_f),\Lambda)$, $M$ injective $\Lambda$-module, pull back to $R\Gamma(K_{N,S},-)$-acyclic sheaves under $\pi$). By direct computation, the composition
\[
Z_M^+ \to \End_{D(P(\A_f^S)\times K^\prime_{M,S})}(R\Gamma(K_{N,S},V)) \overset{\pi^\ast}{\to} \End_{D_{P(\A_f^S)\times K^\prime_{M,S}}(\mf{X}^M,\Lambda)}(R\Gamma(K_{N,S},\pi^\ast V))
\]
is the action $Z_M^+ \to \End_{D_{P(\A_f^S)\times K^\prime_{M,S}}(\mf{X}^M,\Lambda)}(R\Gamma(K_{N,S},\pi^\ast V))$ defined above.

\medskip

Before proceeding to Proposition \ref{XM to XP}, we prove a lemma that will be used in the proof. For the proof, it will be convenient to note that (discrete) group cohomology for $N(\Zp)$ is often equal to continuous group cohomology: the natural map
\[
 H_{cts}^{i}(N(\Zp),A) \to H^{i}(N(\Zp),A) 
\]
is an isomorphism for all $i$ and all $N(\Zp)$-modules $A$ of finite cardinality by \cite[Theorem 2.10]{fkrs} (since $N(\Zp)\cong \Zp^{d}$). For the lemma, we also need to define a character. There is an algebraic character $\psi \colon M \to \GG_m$ over $\Zp$, which is defined on points as the composition
\[
 \GL_{n}(\cO_{F} \otimes_{\Zp} R) \overset{{\rm det}}{\longrightarrow} (\cO_{F} \otimes_{\Zp}R)^{\times} \overset{{\rm Nm}}{\longrightarrow} R^{\times},
\]
where $R$ is any $\Zp$-algebra, ${\rm det}$ is the determinant and ${\rm Nm}$ is the norm map. For any $k\in \Z$, we write $\Lambda(\psi^k)$ for $\Lambda$ with $M(\Zp)$ acting via $\psi^k$. For any group $H \twoheadrightarrow M(\Zp)$, we regard $\Lambda(\psi^k)$ as an $H$-module by inflation. Continue to use the notation above.

\begin{lemma}\label{splitting off the top degree} $H^d(K_{N,S}, \Lambda)[-d]$ is a direct summand of $R\Gamma(K_{N,S},\Lambda)$ in $D(P(\A_f^S)\times K_{M,S}^\prime, \Lambda)$, and the idempotent cutting it out is a polynomial in $\gamma \in Z_M^+$ which reduces to $\gamma^{k!}$, for sufficiently divisible $k$, on cohomology. Moreover, $H^d(K_{N,S}, \Lambda)\cong \Lambda(\psi^{-n})$ as $P(\A_f^S) \times K_{M,S}^\prime$-modules.
\end{lemma}

\begin{proof}
We start with the first part. Since $R\Gamma(K_{N,S},-) \cong R\Gamma(K_{N,S}^P,R\Gamma(N(\Zp),-))$, we claim that it suffices to prove the first part when $S=\{p\}$\footnote{Of course, in the larger context of the paper this need not be the case, but this lemma and the discussion preceding makes sense for arbitrary finite $S$ containing $p$.}. To see this, it suffices to show that $H^d(K_{N,S},\Lambda) = R\Gamma(K_{N,S}^p, H^d(N(\Zp),\Lambda))$. This follows from the fact that $H^d(N(\Zp),\Lambda)$ is a finite $\Lambda$-module with trivial $K_{N,S}^p$-action and $K_{N,S}^p$ is isomorphic to a finite product of groups of the form $\Z_\ell$, $\ell \in S\setminus \{p\}$. So, we need to construct an idempotent 
\[ 
e\in \End_{D(P(\A_f^p) \times M(\Zp),\Lambda)}(R\Gamma(N(\Z_p), \Lambda))
\]
of the desired form such that $e$ acts as $0$ on $H^i(N(\Z_p), \Lambda)$ for $i\neq d$ and as the identity on $H^d(N(\Z_p), \Lambda)$. Consider the endomorphism (induced by) $\gamma$. We may write it as a composition in the following way: Let $\Lambda \to I^{\bullet}$ be an $H$-acyclic resolution of $\Lambda$ as a $P(\A_f)$-representation; then $(I^{\bullet})^H$ represents $R\Gamma(H,\Lambda)$ in $D_{\mathrm{sm}}(P(\A_f^p)\times M(\Zp),\Lambda)$ for $H=N(\Zp),\gamma N(\Zp)\gamma^{-1}$. We may define a homomorphism
$$ m_{\gamma} \colon R\Gamma(N(\Z_p), \Lambda) \to R\Gamma(\gamma N(\Z_p)\gamma^{-1}, \Lambda) $$
by termwise multiplication by $\gamma$ from $(I^{\bullet})^{N(\Zp)}$ to $(I^{\bullet})^{\gamma N(\Zp) \gamma^{-1}}$. This is an isomorphism with inverse $m_{\gamma^{-1}}$ (defined in the same way). Next, there is a corestriction map
$$ {\rm cores} \colon R\Gamma(\gamma N(\Z_p)\gamma^{-1}, \Lambda) \to  R\Gamma(N(\Z_p), \Lambda) $$
given explicitly by $v \mapsto \sum_{n \in N(\Zp)/\gamma N(\Zp) \gamma^{-1}}nv $ termwise from $(I^{\bullet})^{\gamma N(\Zp) \gamma^{-1}}$ to $(I^{\bullet})^{N(\Zp)}$; from the definitions we see that $\gamma = {\rm cores} \circ m_{\gamma}$. Writing $\gamma$ in this way allows us to analyze the $p$-divisibility of the action of $\gamma$ on cohomology. Since $m_{\gamma}$ is an isomorphism, this boils down to computing ${\rm cores}$ on cohomology groups, which we may do after restricting $\Lambda$ to an $N(\Zp)$-representation. So we wish to understand
$$ {\rm cores} \colon H^{i}(\gamma N(\Zp) \gamma^{-1},\Lambda) \to H^{i}(N(\Zp),\Lambda) $$
for all $0\leq i\leq d$. Choose an isomorphism $N(\Zp)\cong \Zp^{d}$; then $\gamma N(\Zp) \gamma^{-1} \cong (p^{2}\Zp)^{d}$ and, using the K\"unneth formula, it will be sufficient to understand 
$$ {\rm cores} \colon H^{i}(p^{2}\Zp,\Lambda) \to H^{i}(\Zp,\Lambda). $$
We resolve $\Lambda$ by $\Map_{cts}(\Zp,\Lambda) \to \Map_{cts}(\Zp,\Lambda)$ (with the translation action on both terms), where the map is the difference operator $\Delta(f)(x)=f(x+1)-f(x)$; this computes continuous group cohomology of $\Lambda$ and hence discrete group cohomology as well by the discussion preceding the lemma. From the definitions, ${\rm cores}$ is then computed by taking cohomology of the rows of the diagram
\begin{center}
\begin{tikzcd}
\Map(\Z/p^{2}\Z,\Lambda) \arrow[d,"\mathrm{cores}"] \arrow[r,"\Delta"] & \arrow[d,"\mathrm{cores}"] \Map(\Z/p^{2}\Z,\Lambda) \\
\Lambda \arrow[r,"0"] & \Lambda.
\end{tikzcd}
\end{center}
Here the top row is the $p^{2}\Zp$-invariants of $\Map_{cts}(\Zp,\Lambda) \to \Map_{cts}(\Zp,\Lambda)$ and the bottom row is the $\Zp$-invariants (with the obvious identifications), and ${\rm cores}$ is given by ${\rm cores}(f)=\sum_{i\in \Z/p^{2}\Z}f(i)$. As the kernel of $\Delta \colon \Map(\Z/p^{2}\Z,\Lambda) \to \Map(\Z/p^{2}\Z,\Lambda)$ is the constant functions and the image of $\Delta$ is precisely ${\rm Ker}(\mathrm{cores})$, one sees that ${\rm cores}$ is divisible by $p$ (indeed by $p^{2}$) on $H^{0}$'s and is an isomorphism on $H^{1}$'s. Returning to $N(\Zp)\cong \Zp^{d}$, the K\"unneth formula now implies that $ {\rm cores} \colon H^{i}(\gamma N(\Zp) \gamma^{-1},\Lambda) \to H^{i}(N(\Zp),\Lambda)$ is divisible by $p$ when $i<d$ and is an isomorphism when $i=d$. It follows that the action of $\gamma^m$ on $H^{i}(N(\Zp),\Lambda)$ is zero when $i<d$ and the identity when $i=d$, for any $m\geq 1$ divisible by $rp^{r-1}(p-1)$. Fix such an $m$; we have shown that $\gamma^m \in \End_{D(P(\A_f)\times M(\Zp),\Lambda)}(R\Gamma(N(\Z_p), \Lambda))$ gives rise on cohomology to an idempotent with the desired property. Since the natural ring homomorphism
\[
\End_{D(P(\A_f^p)\times M(\Zp),\Lambda)}(R\Gamma(N(\Z_p), \Lambda)) \to \End_{P(\A_f^p)\times M(\Zp)}(H^{\ast}(N(\Z_p), \Lambda))
\]
has nilpotent kernel (for example by the proof of \cite[Lemma 2.5(2)]{khare-thorne}), the idempotent lifting lemma \cite[Tag 00J9]{stacks-project} shows that there is a unique idempotent $e\in \End_{D(P(\A_f^p)\times M(\Zp),\Lambda)}(R\Gamma(N(\Z_p), \Lambda))$ which is a polynomial in $\gamma$ and reduces to $\gamma^m$ on cohomology. This finishes the proof of the first part.

\medskip
For the second part, observe that any $K_{N,S}$-acyclic resolution of $\Lambda$ as a $P(\Qp)$-module to inflates to a $K_{N,S}$-acyclic resolution by $P(\A_f)$-modules. Together with the observation above that 
\[
H^d(K_{N,S},\Lambda) = R\Gamma(K_{N,S}^p, H^d(N(\Zp),\Lambda)) = H^0(K_{N,S}^p,H^d(N(\Zp),\Lambda)),
\] 
this shows that it suffices to show that $H^d(K_{N,S}, \Lambda)\cong \Lambda(\psi^{-n})$ as $M(\Zp)$-modules, when $\Lambda$ is considered as a $P(\Qp)$-module. This is a special case of equation (3.6.5) in the proof of \cite[Proposition 3.6.2]{emerton-ord}, upon noting that the algebraic character denoted by $\alpha$ in \emph{loc.\ cit} is equal to $\psi^{n}$.
\end{proof}

Before proceeding to Proposition \ref{XM to XP}, we recall that there is a ring homomorphism $r_M \colon \mathbb{T}^S_P\to \mathbb{T}^S_M$ defined in \cite[\S 2.2.4, 2.2.5]{newton-thorne}, to which we refer for the precise definition.

\begin{prop} \label{XM to XP} For any $m$, we have an isomorphism
\[
R\Gamma(X_{K_{M,m}}^M,\Lambda(\psi^{-n}))[-d] \cong R\Gamma(\XP,\Lambda)^{\ord}
\]
in $D(M(\Z/p^{m}\Z), \Lambda)$, compatible with changing the level, and Hecke-equivariant for the action of $\mathbb{T}_P^S$, where the $\mathbb{T}_P^S$-action on $R\Gamma(X_{K_{M,p,m}}^M,\Lambda(\psi^{-n}))[-d]$ is given by restricting the natural $\mathbb{T}_M^S$-action along $r_M$. 
\end{prop}

\begin{proof}
We continue to use the notation above. This proof follows a part of the proof of \cite[Lemma 4.4]{newton-thorne} closely; we sketch that argument\footnote{Strictly speaking, it is assumed in \cite[Lemma 4.4]{newton-thorne} that the ambient group is $\mathrm{Res}_\Q^F\GL_n$ and that $r=1$, but the parts that we use are completely general.} and add the extra arguments needed. $R\Gamma(\XP,\Lambda)$ with its $\mathbb{T}_P^S$-action is computed by $R\Gamma(K_P,R\Gamma(\mf{X}^P,\Lambda))$; we will need another description. Recall the quotient map $\Phi : \mf{X}^P \to \mf{X}^M$. By \cite[Equation (4.4)]{newton-thorne}, we have (using the notation in the proof of Lemma \ref{splitting off the top degree})
\[
R\Gamma(K_P,R\Gamma(\mf{X}^P,\Lambda)) \cong R\Gamma \left( K_P^S \times K_{M,S}, R\Gamma \left( \mf{X}^M,R\Gamma(K_{N,S},R\Phi_\ast \Lambda) \right) \right)
\]
in $D(\mathbb{T}_P^S \times M(\Z/p^m),\Lambda)$. Further, by \cite[Equation (4.7)]{newton-thorne}, the right hand side is isomorphic to $R\Gamma \left( K_P^S \times K_{M,S}, R\Gamma \left( \mf{X}^M,R\Gamma(K_{N,S},\Lambda) \right) \right) $. By the discussing preceding Lemma \ref{splitting off the top degree}, $R\Gamma(K_{N,S},\Lambda) \in D_{P(\A_f^S)\times K_{M,S}^\prime}(\mf{X}^M,\Lambda)$ carries an action of $\gamma \in Z_M^+$ which comes from the action of $\gamma$ on $R\Gamma(K_{N,S},\Lambda) \in D(P(\A_f^S)\times K_{M,S},\Lambda)$ by pullback along $\pi : \mf{X}^M \to pt$. By explicit computation using an injective resolution, we see that applying  $R\Gamma \left( K_P^S \times K_{M,S}, R\Gamma ( \mf{X}^M,-) \right)$ to the action of $\gamma$ one gets the action of $U_p$ on $R\Gamma(K_P,R\Gamma(\mf{X}^P,\Lambda))$. By uniqueness in the idempotent lifting lemma \cite[Tag 00J9]{stacks-project}, it follows that the idempotent constructed in Lemma \ref{splitting off the top degree} maps to the ordinary projector on $R\Gamma(\XP,\Lambda)$, since both are equal to $U_p^{k!}$ (for sufficient divisible $k$) on cohomology. In particular, using Lemma \ref{splitting off the top degree}, we see that $R\Gamma(\XP,\Lambda)^\mathrm{ord}$ is computed by 
\[
R\Gamma \left( K_P^S \times K_{M,S}, R\Gamma \left( \mf{X}^M,\Lambda(\psi^{-n}) \right) \right)[-d].
\]
It remains to show that 
\[
R\Gamma \left( K_P^S \times K_{M,S}, R\Gamma \left( \mf{X}^M,\Lambda(\psi^{-n}) \right) \right) \cong r_M^\ast R\Gamma(K_{M,m},R\Gamma(\mf{X}^M,\Lambda(\psi^{-n}))),
\] 
since the right hand side computes $R\Gamma(X_{K_{M,m}}^M,\Lambda(\psi^{-n}))$. This is proven in \cite[p. 58-59]{newton-thorne}, more specifically it is the argument beginning in the last paragraph of p. 58 and continuing onto p. 59. In fact that argument shows that we have an isomorphism  
\[
R\Gamma \left( K_P^S, R\Gamma \left( \mf{X}^M,\Lambda(\psi^{-n}) \right) \right) \cong r_M^\ast R\Gamma(K^S_{M,m},R\Gamma(\mf{X}^M,\Lambda(\psi^{-n})))
\]
and the previous isomorphism is obtained by applying $R\Gamma(K_{M,S},-)$. This implies the compatibility of the isomorphism with changing the level, and finishes the proof.
\end{proof}

The last part of our analysis of the boundary will be to compare $R\Gamma(\partial \ol{X}_{K_m}^G, \Lambda)$ and $R\Gamma(\XGP,\Lambda)$. The open immersion $\XGP \to \partial \ol{X}_{K_m}^G$ gives us a $\mathbb{T}_G^S$-equivariant pullback map $R\Gamma(\partial \ol{X}_{K_m}^G, \Lambda) \to R\Gamma(\XGP,\Lambda)$ in $D(M(\Z/p^m),\Lambda)$. Our next result, Proposition \ref{XG to XGP}, asserts that this becomes an isomorphism after localization at a certain maximal ideal of $\mathbb{T}_G^S$. To define this maximal ideal, we recall our non-Eisenstein maximal ideal $\mf{m}\sub \mathbb{T}_M^S$ from \S \ref{nilpotent intro}. The composition
\[
\cS = r_M \circ r_P
\]
defines a ring homomorphism $\mathbb{T}_G^S \to \mathbb{T}_M^S $ and we put $\mf{M} \defeq \cS^{-1}(\mf{m})$; this is a maximal ideal since the residue field of $\mf{m}$ is finite. We remark that the  localizing $R\Gamma(\XGP,\Lambda)$ and $R\Gamma(\partial \ol{X}_{K_m}^G, \Lambda)$ at $\mathfrak{M}$ gives direct summands by the discussion in \S \ref{comm alg}, since $\TT_G^S$ acts through a finite quotient, and that $\TT_G^S$ acts on the localization. Similar remarks apply to all localizations in this paper (the reader may consult \cite[\S 3.2]{newton-thorne} for a more extensive discussion). The homomorphism $\cS$ is the unnormalized Satake transform. Let us give an explicit description of $\cS$ place by place, using the Satake isomorphisms for $M$ and $G$, following \cite[Lem. 5.2.5]{scholze-galois} and \cite[Prop.-Def. 5.3]{newton-thorne} (we adopt the normalizations of the latter reference). If $F$ is imaginary CM and $v$ splits in $F$, then $\cS$ is determined by the map
\[ 
\Zp[q_v^{1/2}][Y_1^{\pm 1},\dots,Y_{2n}^{\pm 1}]^{S_{2n}} \to \Zp[q_v^{1/2}][W_1^{\pm 1},\dots,W_n^{\pm 1},Z_1^{\pm 1},\dots,Z_n^{\pm 1}]^{S_n \times S_n} 
\]
that sends the set $\{Y_1,\dots,Y_{2n}\}$ to the set $\{q_v^{-n/2} W_1,\dots,q_v^{-n/2} W_n,q_v^{n/2} Z_1^{-1},\dots,q_v^{n/2} Z_n^{-1} \}$.
If $F$ is imaginary CM and $v$ is inert in $F$, $\cS$ is determined by the map
\[ 
\Zp[q_v^{1/2}][X_1^{\pm 1},\dotsc,X_n^{\pm 1}]^{S_n \rtimes (\ZZ/2\ZZ)^n}
\to \Zp[q_v^{1/2}][W_1^{\pm 1},\dotsc,W_n^{\pm 1}]^{S_n} 
\]
that sends the set $\{ X_1,\dotsc,X_n \}$ to the set
$\{ q_v^{-n/2} W_1, \dotsc, q_v^{-n/2} W_n \}$. Finally, if $F$ is totally real, $\cS$ is determined
by the map
\[ \Zp[q_v^{1/2}][X_1^{\pm 1},\dots,X_n^{\pm 1}]^{S_n \ltimes (\ZZ/2\ZZ)^n} \to \Zp[q_v^{1/2}][W_1^{\pm 1},\dots,W_n^{\pm 1}]^{S_n} \]
that sends the set $\{X_1,\dots,X_n\}$ to the set
$\{q_v^{-(n+1)/2} W_1,\dots,q_v^{-(n+1)/2} W_n \}$. We now come to Proposition \ref{XG to XGP}.

\begin{prop} \label{XG to XGP}
The pullback map $R\Gamma(\partial \ol{X}_{K_m}^G, \Lambda) \to R\Gamma(\XGP,\Lambda)$ in $D(M(\Z/p^m),\Lambda)$ induces an isomorphism $R\Gamma(\partial \ol{X}_{K_m}^G, \Lambda)_{\mf{M}} \to R\Gamma(\XGP,\Lambda)_{\mf{M}}$ in $D(M(\Z/p^m),\Lambda)$ after localization at $\mf{M}$, equivariant for the action of
$\cH(G(\A_f),K_m)$. It therefore induces an isomorphism
\[
R\Gamma(\partial \ol{X}_{K_m}^G, \Lambda)^{\mathrm{ord}}_{\mf{M}} \to R\Gamma(\XGP,\Lambda)^{\mathrm{ord}}_{\mf{M}}
\]
in $D(M(\Z/p^m),\Lambda)$, which is compatible with the action of $\mathbb{T}_G^S$ and with changing levels.
\end{prop}

\begin{proof}
Let $F$ be imaginary CM. The statement that $R\Gamma(\partial \ol{X}_{K_m}^G, \Lambda)_{\mf{M}} \to R\Gamma(\XGP,\Lambda)_{\mf{M}}$ is an isomorphism in $D(M(\Z/p^m),\Lambda)$ can be checked on cohomology groups, hence after passing to $D(\Lambda)$, and this is then \cite[Theorem 2.4.2]{accghlnstt} (for the trivial weight, after quotienting out by $p^r$). Note here that the space denoted by $\wt{X}_{\wt{K}}^P$ in \emph{loc.\ cit} is our $\XGP$, and not our $\XP$. Note further that the isomorphism in the statement of \cite[Theorem 2.4.2]{accghlnstt} is the inverse of our pullback map: The proof of \emph{loc.\ cit} shows simultaneously that the left map and the composition below
\[
R\Gamma_c(\XGP,\Lambda)_{\mf{M}} \to R\Gamma(\partial \ol{X}_{K_m}^G, \Lambda)_{\mf{M}} \to R\Gamma(\XGP,\Lambda)_{\mf{M}}
\]
are isomorphisms, where the left map is the natural map on compactly supported cohomology induced by the open immersion $\XGP \to \partial \ol{X}_{K_m}^G$ and the right map is the pullback map. The isomorphism in \emph{loc.\ cit} is then obtained by inverting the composition and then applying the left map. Compatibility with the $\cH(G(\A_f),K_m)$-action is clear, and this directly gives us the $\mathbb{T}_G^S$-equivariant isomorphism $R\Gamma(\partial \ol{X}_{K_m}^G, \Lambda)^{\mathrm{ord}}_{\mf{M}} \to R\Gamma(\XGP,\Lambda)^{\mathrm{ord}}_{\mf{M}}$. Finally, compatibility with changing levels follows from the fact that 
\[
    \xymatrix@C+2pc{ R\Gamma(K/K^\prime, R\Gamma(\partial \ol{X}_{K^\prime}^G,\Lambda)) \ar[r]^-{R\Gamma(K/K^\prime,j^\ast)} \ar[d] & R\Gamma(K/K^\prime, R\Gamma(X_{K^\prime}^{G,P},\Lambda)) \ar[d] \\ R\Gamma(\partial \ol{X}_{K}^G,\Lambda) \ar[r]^-{j^\ast} & R\Gamma(X_{K}^{G,P},\Lambda) }
\]
commutes whenever $K^\prime \sub K$ is normal, where $j^\ast$ denotes the pullback map and the vertical maps are the isomorphisms of Proposition \ref{finite to finite basic}. This finishes the proof when $F$ is imaginary CM.

\medskip
The proof when $F$ is totally real is formally identical, as \cite[Theorem 2.4.2]{accghlnstt} may be proven in exactly the same way, using the description of the unnormalized Satake transform given before the Proposition.
\end{proof} 

We can now put all of the above results together with Theorem \ref{main thm non-similitude} to prove the main technical result of this section. 

\begin{thm}\label{direct summand}
There exists an additive homomorphism
\[
\End_{D(M(\Z/p^m),\Lambda)}(R\Gamma(X_{K_m}^G,\Lambda)_{\mathfrak{M}}^{\mathrm{ord}}) \to \End_{D(M(\Z/p^m),\Lambda)}(R\Gamma(X_{K_{M,m}}^M,\Lambda(\psi^{-n}) )_{\mathfrak{m}})
\]
making the diagram
\[
    \xymatrix{ \mathbb{T}_G^S \ar[r] \ar[d]^{\cS} & \End_{D(M(\Z/p^m),\Lambda)}(R\Gamma(X_{K_m}^G,\Lambda)_{\mathfrak{M}}^{\mathrm{ord}}) \ar[d] \\ \mathbb{T}_M^S \ar[r] & \End_{D(M(\Z/p^m),\Lambda)}(R\Gamma(X_{K_{M,m}}^M,\Lambda(\psi^{-n}) )_{\mathfrak{m}}) }
\]
commute, where the horizontal maps are the natural actions and we recall that $\cS = r_M \circ r_P$ is the Satake transform.
\end{thm}

\begin{proof}
Putting together Propositions \ref{XP to XGP}, \ref{XM to XP} (after localization) and \ref{XG to XGP} together, we see that $\cS^\ast R\Gamma(X_{K_{M,m}}^M,\Lambda(\psi^{-n}) )_{\mathfrak{m}}[-d]$ is a $\mathbb{T}_G^S$-equivariant direct summand of $R\Gamma(\partial \ol{X}_{K_m}^G,\Lambda)_{\mathfrak{M}}^{\mathrm{ord}}$, compatible with changing levels. Using Lemma \ref{gluing complexes}, Lemma \ref{idempotents} (to handle localizations) and Lemma \ref{perfect complex for completed cohomology}, we may glue the $R\Gamma(X_{K_{M,m}}^M,\Lambda(\psi^{-n}) )_{\mathfrak{m}}$ into a complex in $\mathrm{Mod}_{sm}(M(\Z_p),\Lambda)$ which we will call $A_\infty$. Similarly, using Lemma \ref{gluing complexes}. Lemma \ref{idempotents} and Proposition \ref{hida admissibility}, we may glue the $R\Gamma(?\ol{X}_{K_m}^G,\Lambda)_{\mathfrak{M}}^{\mathrm{ord}}$ into a complex $B_\infty^{?}$, for $?\in \{\emptyset, \partial\}$. By compatibility and Lemma \ref{gluing complexes}, the idempotents and the Hecke actions at finite level glue as well, showing that $\cS^\ast A_\infty[-d]$ is a $\mathbb{T}_G^S$-equivariant direct summand of $B_\infty^\partial$.

\medskip
We now work in $D_{sm}(M(\Z_p),\Lambda)$ and apply the truncation functor $\tau_{\geq d+1}$. By Lemma \ref{eliminating degree 0} below, this does not change $\cS^\ast A_\infty[-d]$. Moreover, we claim that the natural map $B_\infty \to B_\infty^\partial$ in $D_{sm}(M(\Z_p),\Lambda)$ becomes an isomorphism after applying $\tau_{\geq d+1}$. This follows from Theorem \ref{main thm non-similitude}, since the cohomology of the cone is $\varinjlim_m H_c^{i+1}(X_{K_m}^G,\Lambda)_{\mathfrak{M}}^{\mathrm{ord}}$ in degree $i$ (using Lemma \ref{gluing complexes}(4)), which vanishes for $i\geq d$ by Theorem \ref{main thm non-similitude} (our $X_{K_m}^G$ are connected components of the Shimura varieties considered there). Thus, $\cS^\ast A_\infty[-d]$ is $\mathbb{T}_G^S$-equivariant direct summand of $\tau_{\geq d+1}B_\infty$. Applying $R\Gamma(K_{M,m,p},-)$ and using Lemma \ref{gluing complexes}, we see that $\cS^\ast R\Gamma(X_{K_{M,m}}^M,\Lambda(\psi^{-n}) )_{\mathfrak{m}}[-d]$ is a $\mathbb{T}_G^S$-equivariant direct summand of $
R\Gamma( K_{M,m,p}, \tau_{\geq d+1}B_\infty )$.
This gives us a projection map on endomorphisms, and the theorem is proved by composing this map with the map
\[
\End_{D(M(\Z/p^m),\Lambda)}\left( R\Gamma(K_{M,m,p},B_\infty) \right) \to \End_{D(M(\Z/p^m),\Lambda)}\left( R\Gamma(K_{M,m,p},\tau_{\geq d+1}B_\infty ) \right)
\]
given by $R\Gamma(K_{M,m,p},-) \circ \tau_{\geq d+1}$, upon noting that $R\Gamma(K_{M,m,p},B_\infty) \cong R\Gamma(\ol{X}_{K_m}^G,\Lambda)_{\mathfrak{M}}^{\mathrm{ord}}$ by Proposition \ref{ordinary finiteness}.
\end{proof}

The following lemma was used in the proof.

\begin{lemma}\label{eliminating degree 0} Let $K_M\sub M(\widehat{\Z})$ be a neat open compact subgroup and let $V$ be a $\Lambda[K_{M,S}]$-module which is finite over $\Lambda$. Equip it with the trivial $M(\A_f^S)$-action. Then $H^0(X^M_{K_M},V)_{\m} = 0$.
\end{lemma}

\begin{proof}
By \cite[Thm.\ 4.2]{newton-thorne}, the map $H^0_c(X^M_{K_M},V)_{\m} \to H^0(X^M_{K_M},V)_{\m}$ is an isomorphism. Every connected component of $X^M_{K_M}$ is noncompact and $V$ is a local system on $X_{K_M}^M$, so $H^0_c(X^M_{K_M},V)=0$.
\end{proof}

\subsection{Proof of Theorem \ref{eliminating nilpotent ideal}}\label{determinants 2}

Armed with Theorem \ref{direct summand}, we may now complete the proof of Theorem \ref{eliminating nilpotent ideal} following Scholze and Newton--Thorne. We fix $K^p$ so that $K=K^pK_p$ is small for any ope compact $K_p \sub G(\Zp)$ and we set $K^p_M = K^p \cap M(\A_f^p)$; then $K_M = K_M^p K_{M,p}$ is small for any $K_{M,p}\sub M(\Zp)$. Let $\lambda$ be a Weyl orbit of weights for $M$. We define a mod $p^r$ Hecke algebra
\[
\mathbb{T}^S_{M}\left(K_M,\lambda,r\right)^\mathrm{der}\defeq \mathrm{Im}\left(\mathbb{T}^S_{M}\to \mathrm{End}_{D(\Lambda)}\left(R\Gamma(X^M_{K_M}, \cV_\lambda/p^r)\right)\right)
\] 
similar to the derived Hecke algebra $\mathbb{T}^S_{M}\left(K_M,\lambda\right)^\mathrm{der}$ defined in \S \ref{nilpotent intro}. There we also defined $\mathbb{T}^S_{M}\left(K_M,\lambda\right)$, and noted that there is a surjection $\mathbb{T}^S_{M}\left(K_M,\lambda\right)^\mathrm{der} \to \mathbb{T}^S_{M}\left(K_M,\lambda\right)$ with nilpotent kernel. In particular, the maximal ideals of  $\mathbb{T}^S_{M}\left(K_M,\lambda\right)^\mathrm{der}$ and $\mathbb{T}^S_{M}\left(K_M,\lambda\right)$ coincide. By \cite[Lemma 3.11]{newton-thorne}, $\mathbb{T}^S_{M}\left(K_M,\lambda\right)^\mathrm{der} = \varprojlim_r \mathbb{T}^S_{M}\left(K_M,\lambda,r\right)^\mathrm{der}$, so it will suffice to construct Galois determinants valued in $\mathbb{T}^S_{M}\left(K_M,\lambda,r\right)^\mathrm{der}$. The first step is the following proposition.

\begin{prop}\label{big determinant} Let $\chi\colon G_{F,S}\to \Z_p^\times$ be a continuous character of finite odd order, prime to $p$.  Let $\m\subset \mathbb{T}^S_M(K_M,\lambda)$ be a non-Eisenstein maximal ideal. Then there exists a continuous group determinant, of dimension $2n$ if $F$ is imaginary CM and of dimension $2n+1$ if $F$ is totally real,  
\[
D_{M, \chi}\colon G_{F,S}\to \mathbb{T}^S_M(K_M, \lambda, r)_{\m}^{\mathrm{der}}
\]
such that for every prime $w$ of $F$ above a rational prime $l\not \in S$, the characteristic polynomial of $D_{M,\chi}(\mathrm{Frob}_{w})$ is as follows.
Set
\[
P^{\vee}_{M,w}(X)\defeq (-1)^n (q^{n(n-1)/2}_wT_{w,n})^{-1}X^nP_{M,w}(X^{-1}).
\] 
\begin{enumerate}
\item
If $F$ is imaginary CM, then the characteristic polynomial is
\[
\chi(\Frob_w)^nP_{M,w}(\chi(\Frob_{w})^{-1}X)\chi(\Frob_{w^c})^{-n}q^{n(2n-1)}_wP^{\vee}_{M,w^c}(q_{w}^{1-2n}\chi(\Frob_{w^c})X).
\]
\item If $F$ is totally real, then the characteristic polynomial is
\[
\chi(\Frob_w)^n P_{M,w}(\chi(\Frob_w)^{-1} X) \chi(\Frob_w)^{-n} q_w^{2n^2} P^{\vee}_{M,w}(q_w^{-2n} \chi(\Frob_w)X)(X-q^n) \,.
\]
\end{enumerate}  
\end{prop}

\begin{proof} Recall that $\mathbb{T}_{cl}^S$ is the ring $\mathbb{T}_G^S$ equipped with the topology defined in \S \ref{determinants}. We first prove the proposition under the assumption that $\chi=1$. We need to show that the map $\mathbb{T}^S_{cl} \to \End_{D(\Lambda)}(R\Gamma(X_{K_M}^M, \cV_\lambda/p^r)_{\m})$ obtained by composing the natural action map with the Satake transform $\cS$ is continuous, when the target is endowed with the discrete topology. This would show that $\mathbb{T}^S_M(K_M,\lambda,r)_{\m}^{\mathrm{der}}$ is a continuous quotient of $\mathbb{T}^S_{cl}$ that is also discrete, and we then obtain the desired determinant by applying Lemma~\ref{existence of determinant dual}.

\medskip

To prove the continuity, we argue as follows. For sufficiently large $m$ we have, by Proposition~\ref{going from finite to finite}, a $\mathbb{T}^S_{cl}$-equivariant isomorphism
\[
R\Gamma\left( K_{M,p}/K_{M,m,p},R\Gamma(X_{K_{M,m}}^M, \Lambda(\psi^{-n})) \otimes_{\Lambda}(\sigma^\circ/p^r \otimes_{\Lambda}\psi^{n}) \right) \cong R\Gamma(X_{K_M}^M, \cV_\lambda/p^r)
\]
in $D(\Lambda)$. The map $\TT_{cl}^{S} \to \mathrm{End}_{D(\Lambda)}\left(R\Gamma(X_{K_M}^M, \cV_\lambda/p^r)_{\m}\right)$ then factors as  
\[
\mathbb{T}^S_{cl}\to \mathrm{End}_{D(K_{M,p}/K_{M,p,m},\Lambda)}\left(R\Gamma(X_{K_{M,m}}^M, \Lambda(\psi^{-n}))_{\m}\right)\to \mathrm{End}_{D(\Lambda)}\left(R\Gamma(X_{K_M}^M, \cV_\lambda/p^r)_{\m}\right).
\]
The second map is continuous, since both the source and the target have the discrete topology, so it is enough to show that
the first map is continuous. But this follows directly by combining Theorem \ref{factors through classical} and Theorem \ref{direct summand}.

\medskip 

This settles the case $\chi =1$. Deducing the case of general $\chi$ from the case $\chi = 1$ is then the second half of the proof of ~\cite[Prop. 5.8]{newton-thorne}. 
\end{proof}

We may then complete the proof of Theorem~\ref{eliminating nilpotent ideal}. Indeed, the theorem now follows from Proposition~\ref{big determinant} in the same way as~\cite[Thm. 1.3]{newton-thorne} follows from~\cite[Prop. 5.8]{newton-thorne} (see \cite[Theorem 5.9]{newton-thorne} and the discussion following it). This finishes the proof of Theorem \ref{eliminating nilpotent ideal}. 

\bibliographystyle{amsalpha}
\bibliography{ArizonaProjectBib}
\end{document}